\documentclass{article} % Base class for articles (reports, short papers)

\usepackage{silence} % Selectively suppress warnings (use sparingly).
\usepackage{arxiv}

\usepackage[utf8]{inputenc} % Allow UTF-8 input (essential for non-ASCII characters)
\usepackage[T1]{fontenc}    % Use 8-bit T1 fonts (better PDF copy/paste and hyphenation)
\usepackage{lmodern}        % Latin Modern font (vectorized Computer Modern)
\usepackage{slantsc}        % Slanted small caps for emphasis
\usepackage{dsfont}         % Double-stroke font (e.g., sets: ℝ, ℤ via \mathds)
\usepackage{upgreek}        % Upright Greek symbols (e.g., constants like \upmu)

\usepackage[english,provide=*,noconfigs]{babel} % Language settings (hyphenation, captions) – modern internal data, no config files
\usepackage{hyphenat}       % Enhanced hyphenation control

\usepackage{amsmath,amssymb,amsbsy,amsthm,amstext,amscd,amsfonts} % Core AMS math extensions
\usepackage{mathtools}          % Enhancements to amsmath (e.g., \coloneqq, paired delimiters)
\mathtoolsset{showonlyrefs=true, showmanualtags=true} % Number only referenced equations
\usepackage{nicematrix}         % Enhanced matrix/table arrays (borders, blocks)
\usepackage{latexsym}           % Additional math symbols (legacy; safe with AMS)

\usepackage[
bookmarks=true,            % Create bookmarks in PDF
pagebackref=true,          % Back-references from bibliography to citing pages
unicode=true,              % Unicode in PDF strings
pdfencoding=auto,          % Auto-detect encoding
pdfstartview={FitH}        % Start PDF view fitted horizontally
]{hyperref}

\makeatletter
\renewcommand*\backref[1]{%
	\ifx#1\relax
	\else
	\@tempcnta=0
	\@for\@next:=#1\do{\advance\@tempcnta by 1}%
	\ifnum\@tempcnta=1
	[Cited on p.~#1]%
	\else
	[Cited on pp.~#1]%
	\fi
	\fi
}
\makeatother

\newcounter{test}							% Counter for automatic "Test <n>" numbering

\makeatletter
\newcommand{\testsubsection}{%
	\refstepcounter{test}					% Step 'test' counter + set ref anchor (hyperref aware)
	\phantomsection							% Extra anchor at this location for robust linking
	\subsection*{Test~\thetest}				% Unnumbered subsection heading "Test <n>"
	\def\@currentlabelname{Test~\thetest}	% Text used by \nameref for this label
	\label{subsec:test\thetest}				% Label: subsec:test1, subsec:test2, ...
}
\makeatother

\usepackage{xparse}           % Create flexible commands/interfaces
\usepackage{url}              % Robust URL typesetting (already covered by hyperref; kept for clarity)

\usepackage{booktabs}       % Professional table rules (\toprule, \midrule, \bottomrule)
\usepackage{multirow}       % Multi-row cells in tables

\usepackage[nice]{nicefrac} % Inline diagonal fractions (e.g., 1/2)
\usepackage[nopatch=footnote]{microtype} % Improves justification; no footnote patch (avoid conflicts)
\usepackage{tcolorbox}      % Colored/shaded boxes for highlights, theorems, algorithms
\usepackage{ragged2e}       % Advanced ragged-right control (e.g., \RaggedRight)

\usepackage{graphicx}       	 % Include graphics: \includegraphics
\usepackage{float}          	 % Improved float placement control (e.g., [H])
\usepackage[section]{placeins}   % Automatically insert float barriers at each \section
\graphicspath{{Figures/}}   	 % Directory for figures (search path)
\usepackage{adjustbox}      	 % Scale/trim boxes (works with graphics and tables)
\usepackage{caption}        	 % Customize caption appearance
\usepackage{subcaption}     	 % Subfigures with captions (modern alternative to subfigure)

\usepackage{tikz}           % High-quality vector graphics/drawings
\usepackage{algorithm}        % Floating algorithm environment
\usepackage{algpseudocode}    % Pseudocode macros (algorithmicx variant)
\algrenewcommand\algorithmicindent{1.0em} % Visual consistency


\usepackage[dvipsnames]{xcolor} % Extended color names (e.g., NavyBlue, ForestGreen)

\usepackage[inline]{enumitem} % Inline lists and fine-grained list control

\usepackage[nottoc]{tocbibind} % Include bibliography in TOC
\usepackage[square,sort,comma,numbers,sectionbib]{natbib} % Flexible numeric citations
\usepackage{doi}               % DOI hyperlinks

\newcommand{\citepbf}[2]{%
	\ifx\empty#1\relax
	\PackageError{citepbf}{Bold label argument is missing}{}%
	\else
	\citep[\textbf{#1}]{#2}%
	\fi
}

\usepackage{titlesec}        % Section title customization
\usepackage{titlecaps}       % Smart Title Case (with exceptions)
\Addlcwords{a an the and but or nor for yet so at by in of on to up as with}

\titleformat{\section} % Customize the format of \section titles
{\normalfont\Large\bfseries} % Use normal font, large size, bold series
{\thesection} % Print the section number
{1em} % Add 1em of space between the number and the title
{\titlecap} % Apply \titlecap to capitalize the title with exceptions

\titleformat{\subsection} % Customize the format of \subsection titles
{\normalfont\large\bfseries} % Use normal font, large size, bold series
{\thesubsection} % Print the subsection number
{1em} % Add 1em of space between the number and the title
{\titlecap} % Apply \titlecap to capitalize the title with exceptions

\usepackage{appendix}        % Appendix management

\usepackage{orcidlink}       % Link ORCID profiles
\usepackage{mathrsfs}        % Script-style math fonts (\mathscr)
\usepackage{indentfirst}     % Indent first paragraph of sections

\newcommand{\refitem}[2]{%
	\hyperref[#2]{#1~\ref*{#2}}% \ref* prints the bare number; \hyperref makes it a single link
}

\newcommand{\figref}[1]{%
	\refitem{Figure}{#1}%
}

\newcommand{\figsubref}[2]{%
	\hyperref[#2]{Figure~\ref*{#1}(\subref*{#2})}% one clickable "Figure <main>(<sub>)"
}

\newcommand{\lemref}[1]{%
	\refitem{Lemma}{#1}%
}

\newcommand{\thmref}[1]{%
	\refitem{Theorem}{#1}%
}

\newcommand{\remref}[1]{%
	\refitem{Remark}{#1}%
}

\newcommand{\testlabel}[1]{%
	subsec:test#1%
}

\newcommand{\testref}[1]{%
	\refitem{Test}{\testlabel{#1}}%
}

\numberwithin{equation}{section} % Equation numbers as <section>.<eq>

\theoremstyle{plain} % Sets the default theorem style to 'plain' (bold title, italic body) — standard for theorems and lemmas.

\newtheorem{theorem}{Theorem}[section] % Defines 'Theorem' environment; numbering restarts with each section (e.g., Theorem 2.1).
\newtheorem*{theorem*}{Theorem}        % Defines an unnumbered theorem environment (for general results without labels).

\newtheorem{corollary}[theorem]{Corollary} % Defines 'Corollary' environment; shares numbering sequence with Theorem (e.g., Theorem 2.1 → Corollary 2.2).
\newtheorem{lemma}[theorem]{Lemma}         % Defines 'Lemma' environment; shares numbering with Theorem for logical consistency.
\newtheorem*{lemma*}{Lemma}                % Defines an unnumbered 'Lemma' environment (for auxiliary or illustrative lemmas).

\newtheorem{remark}[theorem]{Remark}       % Defines 'Remark' environment; shares numbering with Theorem — useful for commentary following results.
\newtheorem*{remark*}{Remark}              % Defines an unnumbered 'Remark' environment (for informal notes or observations).

\newtheorem*{definition*}{Definition}        % Defines an unnumbered 'Definition' environment (for general background or context).

\makeatletter % Allow use of internal macros (e.g., \@addpunct).
\renewenvironment{proof}[1][\proofname]{% % Redefine 'proof' (title defaults to \proofname).
	\par\pushQED{\qed}\normalfont%          % Start paragraph; register QED; use upright font.
	\topsep6\p@\@plus6\p@\relax             % Compact vertical spacing around the block.
	\trivlist\item[\hskip\labelsep\bfseries#1\@addpunct{.}]% % Bold “Proof.” label with period.
	\ignorespaces                            % Prevent stray space before the body.
}{%
	\popQED\endtrivlist\@endpefalse          % Place QED; close list; tidy end-of-paragraph.
}
\makeatother % Restore normal category codes.

\long\def\d{\mathrm{d}}          % Upright differential 'd' for integrals/derivatives (e.g., \int f\,\d x).
\long\def\bigO{\mathcal{O}}      % Big-O notation in calligraphic O (e.g., \bigO(h^2)).

\long\def\A{\mathcal{A}}         % Calligraphic A for operators/sets (e.g., operator \A).
\long\def\B{\mathcal{B}}         % Calligraphic B for spaces/sets (e.g., \B(\Omega)).
\long\def\I{\mathcal{I}}         % Calligraphic I (e.g., index set or functional).
\long\def\H{\mathcal{H}}         % Calligraphic H (often Hilbert space \H); NOTE: overrides accent macro \H{o} (Hungarian umlaut) in text mode.
\long\def\L{\mathcal{L}}         % Calligraphic L (e.g., operator \L, Lebesgue spaces \L^p); NOTE: overrides Polish letter \L.
\long\def\nat{\mathbb{N}}        % Blackboard bold N for natural numbers.
\long\def\real{\mathbb{R}}       % Blackboard bold R for real numbers.
\long\def\S{\mathcal{S}}         % Calligraphic S (e.g., Schwartz space \S); NOTE: overrides \S (§ section symbol).
\long\def\T{\mathcal{T}}         % Calligraphic T (e.g., linear operator or time map \T).

\long\def\ie{\textit{i.e.}}      % Latin abbreviation “that is”; italicized. (No auto-space—add a trailing space in text.)
\long\def\eg{\textit{e.g.}}      % Latin abbreviation “for example”; italicized. (Add space after use.)
\long\def\cf{\textit{cf.}}       % Latin “compare”; italicized. (Add space after use.)
\long\def\apriori{\textit{a priori}} % Latin phrase “from the former”; italicized. (Add space after use.)

\newcommand{\ansatz}{\textit{ansatz}}  % Singular, mid-sentence; typeset in italics to follow mathematical English conventions.
\newcommand{\brac}[1]{\left( #1 \right)}                % Auto-sized parentheses: \brac{...} → ( ... ); pairs well with tall content.
\newcommand{\norm}[1]{\left\lVert #1 \right\rVert}      % Norm with double vertical bars; requires amsmath; auto-sized.
\newcommand{\abs}[1]{\left\lvert #1 \right\rvert}       % Absolute value with single vertical bars; auto-sized.
\newcommand{\qbrac}[1]{\left[ #1 \right]}               % Auto-sized brackets: \qbrac{...} → [ ... ].
\newcommand{\tps}[1]{\ensuremath{{#1}^{\top}}}          % Transpose superscript ‘ᵀ’; \ensuremath allows use in text and math.
\newcommand{\vm}[1]{\ensuremath{\boldsymbol{#1}}}       % Bold math symbol (vectors/matrices) with correct italic correction.
\newcommand{\diag}[1]{\ensuremath{\operatorname{diag}\!\brac{#1}}} % ‘diag(·)’ as proper math operator; \! tightens spacing to parentheses.

\usepackage{xparse}

\newcounter{constant}
\newcounter{mconst}
\newcommand{\globalsecondarg}{}  % Initially empty

\newcommand{\cst}[1]{%
  \refstepcounter{constant}% Step the counter for constants
  \expandafter\xdef\csname #1\endcsname{c_{\arabic{constant}}}% Define the constant with its name
}

\newcommand{\mcst}[1]{%
  \refstepcounter{mconst}% Step the counter for mconsts
  \expandafter\xdef\csname #1\endcsname{M_{\arabic{mconst}}}% Define the mconst with its name
}

\newcommand{\cstvaridx}[2]{%
  \expandafter\xdef\csname #1\endcsname{c_{\arabic{constant} + #2}}% Define the constant with variable index
  \expandafter\xdef\csname #1variable\endcsname{#2}% Store the variable for later reference
}

\newcommand{\detectvar}[1]{%
  \csname #1variable\endcsname
}

\newcommand{\cststatic}[1]{%
  \cstvaridx{#1}{j}% Define the constant using 'j' as the second argument
  \renewcommand{\globalsecondarg}{j}% Store 'j' as the global second argument
}

\newcommand{\detectstaticarg}{%
  \globalsecondarg% Return the globally stored second argument ('j')
}

\cst{cstCoper}
\cst{cstBoper}
\cst{cstrmkthree}
\cst{cstrmkfora}
\cst{cstrmkforb}

\mcst{cstmone}
\mcst{cstmtwo}
\mcst{cstmthree}
\mcst{cstmfour}

\cst{cstDeltaIneq}
\cst{cstGammaIneq}
\cst{cstFinalGammaIneq}

\mcst{cstmfive}
\mcst{cstmsix}
\cst{cstmainalph}
\cst{csttildelama}
\cst{csttildelamb}
\cst{csthatlam}
\cst{cstmaxMk}
\cst{cstlemtwoa}
\cst{cstlemtwob}
\cst{cstlemtwoc}
\cst{cstlemtwod}
\cst{cstykplusonefin}
\cst{cstmaxab}
\cst{cstfourthpoly}
\cst{cstsqrtfracinxi}
\cst{cstxikfin}
\cst{cstoverlinetau}

\cststatic{cststatthmone}  % Define the constant with variable 'j'
\cst{inthmone}
\cst{inthmtwo}
\cststatic{cststatthmtwo}  % Define the constant with variable 'j'
\cst{inthmthree}
\cst{inthmfour}

\cstvaridx{boundremainders}{j}  % Define the constant with variable 'j'

\cst{inthmremone}
\cst{inthmremtwo}
\cst{inthmremthree}
\cst{inthmremfour}
\cst{inthmsumonetwo}
\cst{inthmmaxhalfa}
\cst{inthmmaxa}
\cst{inthmtilgfin}
\cst{inthmnutwo}
\cst{inthmsumovrldlt}
\cst{inthmsumthreefour}
\cst{inthmfrsumoverldlt}
\cst{inthmsumovrlndlt}
\cst{inthmassoctau}
\cst{inthmalphoneone}
\cst{inthmagammoneone}
\cst{inthmalphtwoone}
\cst{inthmagammtwoone}
\cst{inthmoverldltone}

\cststatic{cstremfollthm}
\cst{cstremfollthmone}
\cst{cstremfollthmtwo}
\cst{cstcorolderivatone}
\cst{cstcorolderivattwo}
\cststatic{cstremholderone}
\cst{inremcstholderone}
\cst{inremcstholdertwo}
\cststatic{cstremholdertwo}
\cst{inremcstholderthree}
\cst{inremcstholderfour}

\cst{cststabone}
\cst{cststabtwo}

\cststatic{cstthmconvspecopone}
\cst{cstspecone}
\cst{cstspectwo}
\cststatic{cstthmconvspecoptwo}
\cst{cstspecthree}
\cst{cstspecfour}

\cst{lgsuetin}
\cst{lgdiffw}
\cst{lgerrorgal}
\cst{lgerrunifnorm}
\cst{lgboundbtwo}
\cst{lgfinalest}

\def\articletitle{On the Numerical Treatment of an Abstract Nonlinear System of Coupled Hyperbolic Equations Associated with the Timoshenko Model}
\def\articleauthorR{Jemal~Rogava}
\def\articleauthorV{Zurab~Vashakidze}

\title{\articletitle}

\date{} 					% Or removing it

\author{
	{\articleauthorR}\,\orcidlink{0000-0001-9460-4283}\\
	Faculty of Exact and Natural Sciences,\\
	Ivane Javakhishvili Tbilisi State University (TSU),\\
	Ilia Vekua Institute of Applied Mathematics (VIAM),\\
	11 University St., Tbilisi 0186, Georgia\\
	\href{mailto:jemal.rogava@tsu.ge}{\texttt{jemal.rogava@tsu.ge}}\\
	\And
	{\articleauthorV}\,\orcidlink{0000-0001-8736-6213}\\
	School of Science and Technology, The University of Georgia (UG),\\
	77a M. Kostava St., Tbilisi 0171, Georgia\\[6pt]
	Ilia Vekua Institute of Applied Mathematics (VIAM),\\
	Ivane Javakhishvili Tbilisi State University (TSU),\\
	11 University St., Tbilisi 0186, Georgia\\
	\href{mailto:z.vashakidze@ug.edu.ge}{\texttt{z.vashakidze@ug.edu.ge}}, \href{mailto:zurab.vashakidze@tsu.ge}{\texttt{zurab.vashakidze@tsu.ge}}
}

\renewcommand{\shorttitle}{Numerical Treatment of a Nonlinear System of Coupled Hyperbolic Equations}

\hypersetup{
	pdftitle={\articletitle},
	pdfsubject={Numerical Analysis (math.NA); Analysis of PDEs (math.AP)},
	pdfauthor={\articleauthorR,~\articleauthorV},
	pdfkeywords={Nonlinear Kirchhoff string equation, Cauchy problem, Three-layer semi-discrete scheme},
	hidelinks,           % Remove link borders; colorlinks overrides this
	colorlinks=true,     % Colored text links (print-friendly with care)
	linkcolor=blue,
	citecolor=red,
	filecolor=cyan,
	urlcolor=teal,
}

\allowdisplaybreaks % Allow multi-line displayed equations to break across pages

\begin{document}
	\maketitle
	\vfill
	\begin{abstract}\label{abstract}
		The present work addresses the Cauchy problem for an abstract nonlinear system of coupled hyperbolic equations associated with the Timoshenko model in a real Hilbert space. Our purpose is to develop and delve into a temporal discretization scheme for approximating a solution to this problem. To this end, we propose a symmetric three-layer semi-discrete time-stepping scheme in which the nonlinear term is evaluated at the temporal midpoint. As a result, at each time step, this approach reduces the original nonlinear problem to a linear one and enables parallel computation of its solution. Convergence is proved, and second-order accuracy with respect to the time-step size is established on a local temporal interval. The proposed scheme is applied to a spatially one-dimensional nonlinear dynamic Timoshenko beam system, and the results obtained for the abstract nonlinear system are extended to this setting. A Legendre–Galerkin spectral approximation is employed for the spatial discretization. By taking differences of Legendre polynomials within the Galerkin framework, the resulting linear system is sparse and can be efficiently decoupled. The convergence of the method is also investigated. Finally, several numerical experiments on carefully chosen benchmark problems are conducted to validate the proposed approach and to confirm the theoretical findings.
	\end{abstract}
	\vfill
	% keywords
	
	% The 2020 Mathematics Subject Classification (MSC2020)
	
	%\newpage
	\section{Introduction}\label{sec:intro}
	\subsection{Formulation of the Problem}\label{subsec:statement}
	Motivated by the theory of beams, plates, and shells, this work presents a natural abstract generalization of the nonlinear dynamic Timoshenko beam system. As a second-generation beam theory, the Timoshenko model \cite{Timoshenko1928} extends beyond the classical Euler-Bernoulli (Kirchhoff-Love) theory by accounting for shear deformation and rotational bending. This approach is particularly suitable for analyzing thick beams, sandwich composite beams, and beams subjected to high-frequency excitation where the wavelength is comparable to the beam thickness. We paid attention to this model, since it has wide-ranging applications in structural and mechanical engineering, stemming from the nonlinear theory of elasticity.
	
	Let us consider the Cauchy problem associated with a nonlinear abstract system of coupled hyperbolic equations defined in the real Hilbert space $\H$:
	\begin{subequations}
		\begin{gather}
			\frac{\d^2 u\brac{t}}{\d t^2} + \brac{\alpha + \beta {\norm{A^{\nicefrac{1}{2}} u}}^2} A u\brac{t} + a_1 B v\brac{t} = f_1\brac{t}\,,\quad t \in \qbrac{0,T}\,,\label{eq:abst_timoshenko1}\\
			\frac{\d^2 v\brac{t}}{\d t^2} + \gamma A v\brac{t} + \delta C v\brac{t} + a_2 B u\brac{t} = f_2\brac{t}\,,\label{eq:abst_timoshenko2}\\
			u\brac{0} = \varphi_0\,,\quad u^{\prime}\brac{0} = \varphi_1\,,\quad v\brac{0} = \psi_0\,,\quad v^{\prime}\brac{0} = \psi_1\,.\label{eq:abst_timoshenko3}
		\end{gather}
	\end{subequations}
	
	Let $\alpha$, $\beta$, $\gamma$, and $\delta$ be positive constants, and let $a_1$ and $a_2$ be constants unrestricted in sign. Consider $A$ to be a self-adjoint, positive-definite operator (independent of $t$ and generally unbounded) with domain $D\brac{A}$, which is everywhere dense in $\H$, {\ie} $\overline{D\brac{A}} = \H$, $A = A^{\ast} \geq \nu I$ with $\nu > 0$ (where $I$ denotes the identity operator). Let $C$ be a symmetric, bounded operator such that $\brac{C \varphi,\varphi} \geq 0$ for all $\varphi \in \H$. The operator $B$ is assumed to be a closed linear operator satisfying the following condition:
	\begin{equation}\label{eq:B_cond}
		{\norm{B \varphi}}^2 \leq \cstBoper^2 \brac{A \varphi,\varphi}\,,\quad \forall \varphi \in D\brac{A} \subset D\brac{B}\,,\quad \cstBoper > 0\,,
	\end{equation}
	here, $\brac{\cdot,\cdot}$ denotes the inner product defined on the Hilbert space $\H$, and $\norm{\cdot}$ denotes the norm associated with this inner product. Moreover, let $\varphi_0$, $\varphi_1$, $\psi_0$, and $\psi_1$ be specified vectors in the Hilbert space $\H$. The unknown functions $u\brac{t}$ and $v\brac{t}$ are continuous and twice continuously differentiable, mapping into $\H$. Additionally, $f_1\brac{t}$ and $f_2\brac{t}$ are specified continuous functions taking values in $\H$.
	
	We next discuss the appearance of the square root of the operator $A$ in equation \eqref{eq:abst_timoshenko1}.
	\begin{remark}
		Since $A$ is a self-adjoint and positive definite operator, there exists a unique square root $A^{\nicefrac{1}{2}}$ (see Kato \citepbf{Theorem V.3.35}{Kato1995}), whose domain contains the domain of $A$. It is clear that, for any vector $u \in D \brac{A}$, the equality $\norm{A^{\nicefrac{1}{2}} u}^2 = \brac{Au,u}$ holds. Let $A$ be a self-adjoint extension in $L^2$ of the symmetric operator $\brac{- \d^2 / \d x^2}$ or of the operator $\brac{- \Delta}$. It is assumed that the domains of these operators consist of $C^2$ class functions satisfying homogeneous boundary conditions. For such functions, using Green's formula (in the one-dimensional case, integration by parts), we obtain that the inner product $\brac{Au,u}$ is equal to the integral of the square of the gradient. Thus, we obtain the Kirchhoff nonlinearity. Hence, this observation naturally motivates the introduction, in the corresponding abstract equation, of a term involving ${\norm{A^{\nicefrac{1}{2}} u}}^2$ as an abstract analogue of the Kirchhoff nonlinearity.
	\end{remark}
	In accordance with the linear case ({\cf} Kre\u{\i}n \citepbf{Chapter~III}{Krein1971}), the vector functions $u\brac{t}$ and $v\brac{t}$, taking values in $\H$ and defined on the interval $\qbrac{0,T}$, are called solutions to the problem \eqref{eq:abst_timoshenko1}-\eqref{eq:abst_timoshenko3} if they fulfill the following conditions:
	\begin{enumerate*}[label=(\alph*)]
		\item $u\brac{t}$ and $v\brac{t}$ are twice continuously differentiable vector functions over the interval $\qbrac{0,T}$;
		\item For each $t \in \qbrac{0,T}$, $u\brac{t} \in D\brac{A}$ and $v\brac{t} \in D\brac{A}$. Moreover, $A u\brac{t}$ and $A v\brac{t}$ are continuous;
		\item The functions $u\brac{t}$ and $v\brac{t}$ satisfy the system of equations \eqref{eq:abst_timoshenko1} and \eqref{eq:abst_timoshenko2} over the interval $\qbrac{0,T}$, along with the prescribed initial conditions \eqref{eq:abst_timoshenko3}.
	\end{enumerate*}
	In this context, continuity and differentiability are considered in relation to a metric defined on the Hilbert space $\H$.
	
	Moving from theory to application, we focus on a specific case of the problem \eqref{eq:abst_timoshenko1}-\eqref{eq:abst_timoshenko3}, corresponding to the spatially one-dimensional initial-boundary value problem \eqref{eq:timosh_spec_nonlinear}-\eqref{eq:timosh_spec_bound_conds}. This model describes the geometrically nonlinear vibrations of a beam and is formulated as a coupled system of nonlinear integro-partial differential equations with homogeneous Dirichlet boundary conditions. The homogeneous version of this nonlinear system was first proposed by Sapir and Reiss in the appendix of \cite{SapirReiss1979} as a mathematical model for the free vibrations of an elastic beam with fixed endpoints. The unknown functions $u$ and $v$ represent the transverse displacement of the beam centerline and the rotational displacement of the cross section, respectively. The constants in the model characterize the physical and mechanical properties of the beam and are determined by its material parameters (the significance of these constants is discussed in \cite{SapirReiss1979} and \cite{Tucsnak1991}).
	
	The same mechanical problem has been analyzed within the framework of the nonlinear Euler-Bernoulli beam theory ({\cf} the articles by Dickey \cite{Dickey1970}, Nayfeh and Mook \cite{Nayfeh1979}, and Woinowsky-Krieger \cite{Woinowsky-Krieger1950}). However, unlike the Timoshenko model, this theory is limited because it does not account for shear deformation and rotary inertia effects. For the resulting initial and boundary value problems, global existence and uniqueness results in appropriate Sobolev spaces were established by Ball \cite{Ball1973} and Dickey \cite{Dickey1970} using the Galerkin method. The equations examined in these studies can be formulated as a semilinear evolution equation in a Hilbert space. This formulation enables the application of semigroup theory to reduce the problem to an integral equation ({\cf} Ball \cite{BallSecond1973} and Fitzgibbon \cite{Fitzgibbon1981}). Although this approach is not applicable to the study of \eqref{eq:abst_timoshenko1}-\eqref{eq:abst_timoshenko3} because the corresponding evolution equation is nonlinear. Issues of existence and uniqueness of solutions to the abstract nonlinear wave equation have been addressed, for instance, in papers of Hochstein-Kind \cite{Hochstein-Kind1988} and Poho\v{z}aev \cite{Pohozhaev1975}.
	
	As far as we know, the issue of solvability in the abstract setting of the dynamic nonlinear Timoshenko beam problem \eqref{eq:abst_timoshenko1}-\eqref{eq:abst_timoshenko3} in a Hilbert space has not been studied. Regarding the Timoshenko beam system \eqref{eq:timosh_spec_nonlinear}-\eqref{eq:timosh_spec_bound_conds}, Tucsnak first raised the question of existence and uniqueness of solutions in \cite{Tucsnak1991} and proved local-in-time solvability under suitable regularity assumptions on the initial data. Subsequently, Ammari \cite{Ammari2002} investigated the global existence and large-time behavior of the system governing the nonlinear vibrations of a Timoshenko beam and, under small initial data, established the global existence of strong solutions together with exponential decay of the associated energy. More recently, Narciso and Cousin \cite{Narciso-Cousin2008} obtained existence and uniqueness results for local solutions in noncylindrical domains by employing the Faedo–Galerkin method. Furthermore, in \cite{Aouragh_at_al_2024}, Aouragh, Segaoui, and Soufyane analyzed the stability of a nonlinear shear beam system and established the well-posedness of the considered model through the Faedo–Galerkin method. In the subsequent work \cite{Aouragh_at_al_2025}, Aouragh, El Baz, and Soufyane investigated a thermoelastic nonlinear shear beam model with thermal dissipation. Using the Faedo–Galerkin method, they established the well-posedness of the system and proved its exponential stability via the multiplier method. For numerical purposes, they employed a finite element discretization in space, combined with the Euler and Crank–Nicolson schemes for temporal discretization.
	
	We would like to state that we are not aware of the works that directly address the design of numerical algorithms and the construction of approximate solutions for the abstract analogue of the Timoshenko system \eqref{eq:abst_timoshenko1}-\eqref{eq:abst_timoshenko3}. Concerning the construction and analysis of numerical schemes for the concrete Timoshenko system \eqref{eq:timosh_spec_nonlinear}-\eqref{eq:timosh_spec_bound_conds}, several articles are devoted to these issues. In particular, Bernardi and Copetti \cite{Bernardi2014,Bernardi2017} proposed a numerical approach for a contact problem arising in the nonlinear dynamic thermoviscoelastic Timoshenko beam model. After establishing the well-posedness of the corresponding system of three equations, they applied finite element discretization in space combined with Euler and Crank–Nicolson time-stepping schemes, derived a priori error estimates for the discrete problem, and presented results from numerical experiments. Peradze \cite{Peradze2004} investigated the homogeneous version of problem \eqref{eq:timosh_spec_nonlinear}-\eqref{eq:timosh_spec_bound_conds} under the assumption that the associated Cauchy data are analytic. The numerical solution was obtained by combining the Galerkin method with a Crank–Nicolson-type difference scheme and a Picard iteration procedure, and convergence and accuracy properties of the resulting algorithm were analyzed. In a subsequent work \cite{Peradze2020}, Peradze considers problem \eqref{eq:timosh_spec_nonlinear}-\eqref{eq:timosh_spec_bound_conds} in a homogeneous setting and proposes a numerical solution based on the finite element method, a modified Crank–Nicolson-type difference scheme, and a Picard-type iteration process. The total error of the proposed method is estimated. In the paper by Hauck, M{\aa}lqvist, and Rupp \cite{Hauck2025}, two-level domain decomposition methods for spatial network models are developed, and an application of this method to the Timoshenko beam network is presented.
	
	In our opinion, viewing the system \eqref{eq:abst_timoshenko1}-\eqref{eq:abst_timoshenko2} as an abstract analogue of the nonlinear Timoshenko model has certain advantages. In this context, we can employ a unified approach for a class of related problems. In particular, the formulation \eqref{eq:abst_timoshenko1}-\eqref{eq:abst_timoshenko3} includes, within the framework of the Timoshenko model, the Dirichlet, Neumann, Robin, and mixed boundary value problems. Furthermore, by using the abstract setting, we can also address the spatially multidimensional case on a Lipschitz domain.
	
	The approach used in the following discussion also covers the case where the coupling coefficients $a_1$ and $a_2$ vanish. When $a_1 = a_2 = 0$, system \eqref{eq:abst_timoshenko1}-\eqref{eq:abst_timoshenko2} decouples into two independent equations. The first equation is an abstract counterpart of the classical Kirchhoff equation and includes both one-dimensional and multidimensional settings; see Lions \cite{Lions1978}. Moreover, our approach can be extended to the case where, in equation \eqref{eq:abst_timoshenko1}, the term ${\norm{A^{\nicefrac{1}{2}} u}}^2$ is replaced by ${\norm{u}}^2$. This formulation includes, as a particular case, the Carrier equation \cite{Carrier1945}.
	
	The second equation provides an abstract formulation of the linear model of wave propagation. Clearly, it covers both one-dimensional and multidimensional cases. In addition, this equation also includes dynamic plane and spatial problems arising in the theory of elasticity. In this setting, $A$ represents the Friedrichs extension of the symmetric positive definite elasticity operator associated with the displacement formulation of the elasticity system under classical boundary conditions; see Mikhlin \citepbf{Section~29}{Mikhlin1964}.
	
	\subsection{Outline of the Problem}
	
	We develop a time-stepping scheme for the nonlinear problem \eqref{eq:abst_timoshenko1}-\eqref{eq:abst_timoshenko3}. To achieve this objective, a uniform temporal grid is introduced, and a symmetric three-layer semi-discrete scheme is proposed, in which the nonlinear term is evaluated at a temporal midpoint. Such a treatment permits an approximate solution of the original nonlinear problem by solving a linear problem in parallel at each temporal layer.
	
	The investigation of the stability and convergence of the proposed scheme is based on the following fact: the sequences of vectors $A^{\nicefrac{s}{2}} \brac{u_k - u_{k - 1}} / \tau$ and $A^{\nicefrac{\brac{s + 1}}{2}} u_k$, with $s = 0, 1$, as well as $\brac{v_k - v_{k - 1}} / \tau$ and $L^{\nicefrac{1}{2}} v_k$, where $L = \delta A + \gamma C$, are uniformly bounded. Here, $u_k$ and $v_k$ denote approximate solutions, and $\tau > 0$ is the time step. These facts allow us to estimate the error of the approximate solution and to show that the order of convergence is $\bigO \brac{\tau^2}$ in the class of smooth solutions. Moreover, the approximation error of the first derivative, obtained by applying the central finite difference formula to the approximate solution, is also of order $\bigO \brac{\tau^2}$.
	
	We subsequently consider a nonlinear, spatially one-dimensional Timoshenko beam system \eqref{eq:timosh_spec_nonlinear}-\eqref{eq:timosh_spec_bound_conds} as a special case of the problem \eqref{eq:abst_timoshenko1}-\eqref{eq:abst_timoshenko3}, and the results obtained within the abstract framework are extended to this setting. In this case, the proposed scheme yields second-order linear ordinary differential equations for the unknowns $u_k \brac{x}$ and $v_k \brac{x}$ at each temporal layer. Note that these differential equations are not coupled and can therefore be solved in parallel. To compute numerical solutions of these equations, we employ the Legendre–Galerkin spectral method. The chosen basis functions ensure that the resulting Galerkin linear system is sparse, which can be decoupled into two linear subsystems. More precisely, the coefficient matrix of the corresponding Galerkin linear system is a symmetric positive-definite tridiagonal matrix with a gap, in the sense that it has only nonzero entries on the main diagonal and the second sub- and superdiagonals, while the first sub- and superdiagonals vanish. By exploiting this structure, the associated linear system can be decomposed into two independent tridiagonal subsystems: one corresponding to the odd-indexed unknowns and the other to the even-indexed unknowns. This property is particularly important for numerical implementation.
	
	An important step in the numerical computation of the Timoshenko model using the proposed scheme is the estimation of the error of the Legendre–Galerkin spectral method. The approximation error of this method is estimated using both the $L^2$-norm and the uniform norm for the ordinary differential equations resulting from time discretization.
	
	The crowning stage of each algorithm's development is numerical experiments. It is well known that the construction of algorithms for the numerical computation of mathematical models describing various processes, together with their computer implementation, enables the use of numerical experiments that replace many expensive real-world experiments spanning various fields of natural science. To assess the practical value of an algorithm, it is important to conduct numerical experiments on model problems that fully demonstrate the role of each component of the algorithm.
	
	The paper discusses four benchmark problems whose solutions exhibit oscillatory behavior. Numerical results obtained using the proposed algorithm indicate that the combined algorithm is stable and achieves high practical accuracy for these problems. The proposed algorithm is implemented in Python, and the corresponding source code is archived in an open-access Zenodo repository \cite{RogVashCode2026}.
	
	\section{Symmetric Three-Layer Semi-Discrete Scheme}\label{sec:semidiscrete}
	We aim to find a solution to the problem \eqref{eq:abst_timoshenko1}-\eqref{eq:abst_timoshenko3} using the following semi-discrete scheme:
	\begin{subequations}
		\begin{gather}
			\frac{\Delta^2 u_{k - 1}}{\tau^2} + \brac{\alpha + \beta {\norm{A^{\nicefrac{1}{2}} u_k}}^2} \frac{A u_{k + 1} + A u_{k - 1}}{2} = f_{1,k} - a_1 B v_k\,,\label{eq:semi_discrete_timosh_1} \\
			\frac{\Delta^2 v_{k - 1}}{\tau^2} + \frac{L v_{k + 1} + L v_{k - 1}}{2} = f_{2,k} - a_2 B u_k\,,\label{eq:semi_discrete_timosh_2}
		\end{gather}
	\end{subequations}
	where $k = 1,2,\ldots,n - 1$ and $\tau = T/n$ (with $n > 1$), $\Delta u_{k - 1} = u_k - u_{k - 1}$ and $\Delta^2 u_{k - 1} = \Delta\brac{\Delta u_{k - 1}}$. Additionally, let $f_{1,k} = f_1\brac{t_k}$ and $f_{2,k} = f_2\brac{t_k}$, where $t_k = k\tau$. The initial conditions are given by $u_0 = \varphi_0$ and $v_0 = \psi_0$. The operator $L$ is defined as $L = \gamma A + \delta C$.
	
	Let $u\brac{t_k}$ and $v\brac{t_k}$ be the values of the solutions to problem \eqref{eq:abst_timoshenko1}-\eqref{eq:abst_timoshenko3} at the discrete time points $t = t_k$. Their corresponding numerical approximations at these points are denoted by $u_k$ and $v_k$, respectively; thus, $u\brac{t_k} \approx u_k$ and $v\brac{t_k} \approx v_k$.
	
	To perform computations using schemes \eqref{eq:semi_discrete_timosh_1} and \eqref{eq:semi_discrete_timosh_2}, it is essential to ascertain the initial vectors $u_0$, $v_0$, $u_1$, and $v_1$. The vectors $u_0$ and $v_0$ are prescribed by $u_0 = \varphi_0$ and $v_0 = \psi_0$, respectively. The vectors $u_1$ and $v_1$ need to be approximated. It is well-established that to approximate $u_1$ and $v_1$, one should expand the exact solutions $u\brac{t}$ and $v\brac{t}$ in a Taylor series around $t = 0$, retaining at least the first two terms. This yields: $u_1 = \varphi_0 + \tau \varphi_1$ and $v_1 = \psi_0 + \tau \psi_1$. To achieve second-order accuracy, it is necessary to include the first three terms in the Taylor expansions of $u\brac{\tau}$ and $v\brac{\tau}$, which involves the second-order derivatives. These second-order derivatives of $u\brac{t}$ and $v\brac{t}$ at $t = 0$ can be determined from equations \eqref{eq:abst_timoshenko1} and \eqref{eq:abst_timoshenko2}, considering the initial conditions specified in \eqref{eq:abst_timoshenko3}. Thus, we obtain:
	\begin{gather}
		u_1 = \varphi_0 + \tau \varphi_1 + \frac{\tau^2}{2} \varphi_2\,,\quad \varphi_2 = f_{1,0} - a_1 B \psi_0 - \brac{\alpha + \beta {\norm{A^{\nicefrac{1}{2}} \varphi_0}}^2} A \varphi_0\,,\label{eq:initial_vect01} \\
		v_1 = \psi_0 + \tau \psi_1 + \frac{\tau^2}{2} \psi_2\,,\quad \psi_2 = f_{2,0} - a_2 B \varphi_0 - L \psi_0\,.\label{eq:initial_vect02}
	\end{gather}
	By substituting the values of the starting vectors $u_0$, $v_0$, $u_1$, and $v_1$ into equations \eqref{eq:semi_discrete_timosh_1} and \eqref{eq:semi_discrete_timosh_2}, we derive the subsequent linear system of equations that defines the vectors $u_2$ and $v_2$:
	\begin{subequations}
		\begin{align}
			\brac{I + \frac{\tau^2}{2} q_1 A} u_2 &= g_1\,,\label{eq:oper_eq_u2} \\
			\brac{I + \frac{\tau^2}{2} L} v_2 &= g_2\,,\label{eq:oper_eq_v2}
		\end{align}
	\end{subequations}
	where the scalar $q_1 = \alpha + \beta {\norm{A^{\nicefrac{1}{2}} u_1}}^2$ and the right-hand sides $g_1$ and $g_2$ are given.
	
	Given that $A$ is a self-adjoint, positively-definite operator and $q_1 > 0$, it follows that the operator $I + 0.5 \tau^2 q_1 A$ is also self-adjoint and positively definite. Consequently, the operator $I + 0.5 \tau^2 q_1 A$ is continuously invertible, ensuring that equation \eqref{eq:oper_eq_u2} has a unique solution for each $g_1$ from $\H$, which depends continuously on the right-hand side. The same reasoning applies to equation \eqref{eq:oper_eq_v2}. Moreover, we consider the following fact: the sum of a self-adjoint operator and a symmetric, bounded operator is also self-adjoint ({\cf} Kato \citepbf{Chapter~IV}{Kato1995}). Therefore, $u_2 = {\brac{I + 0.5 \tau^2 q_1 A}}^{-1} g_1$ and $v_2 = {\brac{I + 0.5 \tau^2 L}}^{-1} g_2$, where ${\brac{I + 0.5 \tau^2 q_1 A}}^{-1}$ and ${\brac{I + 0.5 \tau^2 L}}^{-1}$ are bounded, self-adjoint operators defined over the entire Hilbert space $\H$. In a way analogous to the determination of $u_2$ and $v_2$, the vectors $u_k$ and $v_k$ for $k > 2$ are computed using $u_{k - 1}$, $v_{k - 1}$, $u_{k - 2}$, and $v_{k - 2}$. Consequently, the application of schemes \eqref{eq:semi_discrete_timosh_1} and \eqref{eq:semi_discrete_timosh_2} is reduced to solving linear problems at each temporal layer. The proposed algorithm enables the parallel computation of the vectors $u_k$ and $v_k$.
	
	\section{Preliminaries}\label{sec:prelim}
	In the following, we present two lemmas, stated without proof, which we require to establish the convergence of the scheme \eqref{eq:semi_discrete_timosh_1}-\eqref{eq:semi_discrete_timosh_2}.
	\begin{lemma}[For further details, refer to \textbf{Lemma 3.2} in \citep{RogavaTsiklauri2012}]\label{lemma:rogava-tsiklauri1}
		Consider the sequences of nonnegative numbers ${\left\{ {\alpha}_{k} \right\}}_{k = 0}^{n}$ and ${\left\{ {c}_{k} \right\}}_{k = 0}^{n}$, which satisfy the following inequality:
		\begin{equation*}
			{\alpha}_{k + 1} \leq {\alpha}_{k}\left( 1 + {\tau}{\alpha}_{k}^{s} \right) + {\tau}{c}_{k}\,,
		\end{equation*}
		where ${s} > 0$ and ${\tau} > 0$ hold.
		
		Consequently, the estimate is valid
		\begin{equation*}
			{\alpha}_{k} \leq \frac{\alpha}{{\left( 1 - {s}{\alpha}^{s}{t}_{k}{a}_{k} \right)}^{\nicefrac{1}{s}}}\,,\quad {t}_{k} = {k}{\tau} < \frac{1}{{s}{\alpha}^{s}{a}_{k}}\,,\quad {\alpha} = \max\left( 1,{\alpha}_{0} \right)\,,\quad {a}_{k} = 1 + \max\limits_{{0} \leq {i} \leq {k}}{\left( {c}_{i} \right)}\,.
		\end{equation*}
	\end{lemma}
	The subsequent lemma is stated with an adjustment to the proposed scheme \eqref{eq:semi_discrete_timosh_2}. The lemma is formulated as follows.
	\begin{lemma}[See \textbf{Lemma 3.1} in \citep{RogavaTsiklauri2014} for further details]\label{lemma:rogava-tsiklauri2}
		Given that $L = \delta A + \gamma C$ is a self-adjoint, positive definite operator, the following {\apriori} estimate is valid for scheme \eqref{eq:semi_discrete_timosh_2} for all $s \geq 0$:
		\begin{equation}\label{eq:lemma_rogtsikl2}
			\begin{aligned}
				\norm{L^s v_{k + 1}} &\leq \sqrt{2} \brac{ \norm{L^s v_0} + \norm{L^{s - \nicefrac{1}{2}} \frac{\Delta v_0}{\tau}} } + \tau\norm{L^s\frac{\Delta v_0}{\tau}} \\
				&+ \tau\sum_{i = 1}^{k} \norm{L^{s - \nicefrac{1}{2}} g_{2,i}}\,,\quad L^0 = I\,,
			\end{aligned}
		\end{equation}
		in which $v_0$ and $v_1$ belong to $D\brac{L^s}$, and $g_{2,i} = f_{2,i} - a_2 B u_i$ belongs to $D\brac{L^{s - \nicefrac{1}{2}}}$.
	\end{lemma}
	For our purposes, in the {\apriori} estimate of \eqref{eq:lemma_rogtsikl2}, we need to replace the operator $L$ with the operator $A$. This requires dual assessments of the operator $L$ in terms of the operator $A$. We shall state these estimates in the form of remarks.
	
	\begin{remark}\label{prop:remark1}
		For the operator $L = \delta A + \gamma C$, the following lower and upper bounds hold:
		\begin{equation}\label{eq:remark1_double_ineqt_main}
			\hat{\gamma} \norm{A \varphi} \leq \norm{L \varphi} \leq \nu_{0} \norm{A\varphi}\,,\quad \varphi \in D\brac{A}\,,\quad \hat{\gamma} > 0\,,\quad \nu_{0} > 0\,.
		\end{equation}
	\end{remark}
	\begin{proof}
		Let us now demonstrate the left-hand inequality of \eqref{eq:remark1_double_ineqt_main}. By applying a straightforward transformation, it is obtained
		\begin{equation}\label{eq:remark1_lower_bound}
			{\norm{L\varphi}}^{2} = \gamma^2 {\norm{A\varphi}}^2 + 2 \gamma \delta \brac{A\varphi,C\varphi} + \delta^2 {\norm{C\varphi}}^2 \geq \gamma^2 {\norm{A\varphi}}^2 - 2 \gamma \delta \norm{A\varphi}\norm{C\varphi} + \delta^2 {\norm{C\varphi}}^2\,.
		\end{equation}
		In the derivation, we have applied the Cauchy-Schwarz inequality.
		
		Applying Young's inequality with $\varepsilon$ (valid for every $\varepsilon > 0$) to the term $2\norm{A\varphi}\norm{C\varphi}$ yields the following result for assessment \eqref{eq:remark1_lower_bound}
		\begin{equation}\label{eq:remark1_Lu_ineqt}
			{\norm{L\varphi}}^{2} + \frac{\delta}{\varepsilon} \brac{\gamma - \varepsilon \delta} {\norm{C\varphi}}^2 \geq \gamma \brac{\gamma - \varepsilon \delta} {\norm{A\varphi}}^2\,,\quad 0 < \varepsilon < \frac{\gamma}{\delta}\,.
		\end{equation}
		Given the positive definiteness of the operator $A$ and using the Cauchy-Schwarz inequality, we obtain the following result:
		\begin{equation*}
			\norm{L\varphi} \norm{\varphi} \geq \brac{L\varphi,\varphi} \geq \gamma \brac{A\varphi,\varphi} \geq \gamma \nu {\norm{\varphi}}^2\,,
		\end{equation*}
		which leads to the conclusion that
		\begin{equation*}
			\norm{\varphi} \leq \frac{1}{\gamma \nu} \norm{L\varphi}\,.
		\end{equation*}
		Furthermore, under the assumption that $\norm{C\varphi} \leq c \norm{\varphi}$, with $c = \norm{C}$, we obtain the resulting inequality from \eqref{eq:remark1_Lu_ineqt}
		\begin{equation*}
			\norm{L\varphi} \geq \hat{\gamma} \norm{A\varphi}\,,\quad \hat{\gamma} = \sqrt{\frac{\gamma \brac{\gamma - \varepsilon \delta}}{1 + \frac{c^2 \delta}{\varepsilon \gamma^2 \nu^2}\brac{\gamma - \varepsilon \delta}}}\,.
		\end{equation*}
		
		Proceeding to demonstrate the right-hand inequality of \eqref{eq:remark1_double_ineqt_main}, we account for the fact that $\norm{C\varphi} \leq c \norm{\varphi}$ and thus arrive at the following conclusion:
		\begin{equation}\label{eq:remark1_norm_L}
			\norm{L\varphi} \leq \gamma \norm{A\varphi} + c \delta \norm{\varphi}\,.
		\end{equation}
		By taking an additional step, which involves applying the Cauchy-Schwarz inequality and employing the positive definiteness of the operator $A$, we obtain:
		\begin{equation*}
			\norm{A\varphi}\norm{\varphi} \geq \brac{A\varphi,\varphi} \geq \nu {\norm{\varphi}}^2\,.
		\end{equation*}
		Based on the aforementioned inequality in \eqref{eq:remark1_norm_L}, the desired result follows
		\begin{equation*}
			\norm{L\varphi} \leq \nu_0 \norm{A\varphi}\,,\quad \nu_0 = \gamma + \frac{c\delta}{\nu}\,.
		\end{equation*}
	\end{proof}
	
	\begin{remark}\label{prop:remark2}
		The following bounds are satisfied
		\begin{equation}\label{eq:remark2_ineqt_in_rmk}
			\sqrt{\gamma} \norm{A^{\nicefrac{1}{2}} \varphi} \leq \norm{L^{\nicefrac{1}{2}}\varphi} \leq \nu_0 \norm{A^{\nicefrac{1}{2}}\varphi}\,,\quad \varphi \in D\brac{A^{\nicefrac{1}{2}}}\,, \quad \nu_0 > 0\,.
		\end{equation}
	\end{remark}
	\begin{proof}
		Due to the positive definiteness of the operator $A$, it follows directly that
		\begin{equation}\label{eq:remark2_result_post_def}
			{\norm{A^{\nicefrac{1}{2}} \varphi}}^2 \geq \nu {\norm{\varphi}}^2\,,\quad \varphi \in D\brac{A}\,.
		\end{equation}
		On the other hand, considering inequality \eqref{eq:remark2_result_post_def} along with the fact that $\norm{C\varphi} \leq c \norm{\varphi}$, we obtain
		\begin{align*}
			\brac{L\varphi,\varphi} &= \gamma \brac{A\varphi,\varphi} + \delta\brac{C\varphi,\varphi} \leq \gamma{\norm{A^{\nicefrac{1}{2}} \varphi}}^2 + c\delta{\norm{\varphi}}^2 \\
			&\leq \nu_0 {\norm{A^{\nicefrac{1}{2}} \varphi}}^2\,,\quad \varphi \in D\brac{A}\,,\quad \nu_0 = \gamma + \frac{c\delta}{\nu}\,.
		\end{align*}
		Consequently, we derive that
		\begin{equation}\label{eq:remark2_final_res}
			\norm{L^{\nicefrac{1}{2}}\varphi} \leq \nu_0 \norm{A^{\nicefrac{1}{2}} \varphi}\,,\quad \varphi \in D\brac{A}\,.
		\end{equation}
		Given that the operator $A^{\nicefrac{1}{2}}$ maps $D\brac{A}$ onto $D\brac{A^{\nicefrac{1}{2}}}$ and that the range of $A^{\nicefrac{1}{2}}$, $R\brac{A^{\nicefrac{1}{2}}} = \H$. Under these conditions, it follows that inequality \eqref{eq:remark2_final_res} can be extended to the entire domain $D\brac{A^{\nicefrac{1}{2}}}$. Consequently, the right-hand inequality of \eqref{eq:remark2_ineqt_in_rmk} is satisfied.
		
		The left-hand inequality of \eqref{eq:remark2_ineqt_in_rmk} is derived through analogous reasoning.
	\end{proof}
	
	\begin{remark}\label{prop:remark3}
		Based on Remarks \ref{prop:remark1} and \ref{prop:remark2}, the following {\apriori} estimates can be derived from \eqref{eq:lemma_rogtsikl2}:
		\begin{equation*}
			\begin{aligned}
				\norm{A^{\nicefrac{1}{2}} v_{k + 1}} &\leq \frac{2}{\sqrt{2 \gamma}} \brac{ \nu_0 \norm{A^{\nicefrac{1}{2}} v_0} + \norm{\frac{\Delta v_0}{\tau}} } + \frac{\nu_0\tau}{\sqrt{\gamma}}\norm{A^{\nicefrac{1}{2}} \frac{\Delta v_0}{\tau}} + \frac{\tau}{\sqrt{\gamma}}\sum_{i = 1}^{k} \norm{g_{2,i}}\,, \\
				\norm{A v_{k + 1}} &\leq \frac{\nu_0 \sqrt{2}}{\hat{\gamma}} \brac{ \norm{A v_0} + \norm{A^{\nicefrac{1}{2}} \frac{\Delta v_0}{\tau}} } + \frac{\nu_0 \tau}{\hat{\gamma}} \norm{A\frac{\Delta v_0}{\tau}} + \frac{\nu_0 \tau}{\hat{\gamma}} \sum_{i = 1}^{k} \norm{A^{\nicefrac{1}{2}} g_{2,i}}\,.
			\end{aligned}
		\end{equation*}
	\end{remark}
	
	\begin{remark}\label{prop:remark6}
		From condition \eqref{eq:B_cond}, the following relation is deduced:
		\begin{equation}\label{eq:remark6_final_res}
			\norm{B \varphi} \leq \cstBoper \norm{A^{\nicefrac{1}{2}} \varphi}\,,\quad \forall \varphi \in D\brac{A^{\nicefrac{1}{2}}} \subset D\brac{B}\,.
		\end{equation}
	\end{remark}
	\begin{proof}
		It is evident that from \eqref{eq:B_cond}, the following implication arises:
		\begin{equation}\label{eq:remark6_relat_norm_B_norm_halfA}
			\norm{B \varphi} \leq \cstBoper \norm{A^{\nicefrac{1}{2}} \varphi}\,,\quad \forall \varphi \in D\brac{A} \subset D\brac{B}\,.
		\end{equation}
		It is known that $D\brac{A}$ is a core of $A^{\nicefrac{1}{2}}$ (see, Kato \citepbf{Lemma 3.38}{Kato1995}). This implies that for every $\varphi \in D\brac{A^{\nicefrac{1}{2}}}$, there exists a sequence $\varphi_n \in D\brac{A}$ such that $\varphi_n \to \varphi$ and $A^{\nicefrac{1}{2}} \varphi_n \to A^{\nicefrac{1}{2}} \varphi$. Consequently, by \eqref{eq:remark6_relat_norm_B_norm_halfA}, it follows that $B \varphi_n$ forms a Cauchy sequence. Given that $\H$ is a complete space, it is evident that this sequence converges. Furthermore, since the operator $B$ is closed, we conclude that $\varphi \in D\brac{B}$ and $B \varphi_n \to B \varphi$. As a result of this reasoning and inequality \eqref{eq:remark6_relat_norm_B_norm_halfA}, the relation \eqref{eq:remark6_final_res} is established.
	\end{proof}
	
	In the continuous problem, the coupling of equations \eqref{eq:abst_timoshenko1} and \eqref{eq:abst_timoshenko2} gives rise to the terms $a_1 B v\brac{t}$ and $a_2 B u\brac{t}$. Similarly, in the discrete problem, the terms $a_1 B v_k$ and $a_2 B u_k$ are responsible for the coupling of schemes \eqref{eq:semi_discrete_timosh_1} and \eqref{eq:semi_discrete_timosh_2}. Hence, to naturally connect the {\apriori} estimates of schemes \eqref{eq:semi_discrete_timosh_1} and \eqref{eq:semi_discrete_timosh_2}, it is necessary to fulfil a certain condition for the vector $B \varphi$. The following remark addresses this matter.
	\begin{remark}\label{prop:remark4}
		Let $D_1$ be a subset of $D\brac{A}$ that is dense in the Hilbert space $\H$. Suppose that the operator $B$ maps $D_1$ into $D\brac{A}$, {\ie}, $B:D_1 \to D\brac{A}$. Furthermore, assume that the following inequality is satisfied
		\begin{equation}\label{eq:remark4_inner_prod_ABphi}
			\brac{AB\varphi,B\varphi} \leq \cstrmkthree^2 {\norm{A\varphi}}^2\,,\quad \forall\varphi \in D_1\,,\quad \cstrmkthree > 0\,.
		\end{equation}
		If $R_1 = A D_1$ is dense in $\H$, then the following inequality holds
		\begin{equation}\label{eq:remark4_norm_halfA_B}
			\norm{A^{\nicefrac{1}{2}} B\varphi} \leq \cstrmkthree \norm{A\varphi}\,,\quad \forall \varphi \in D\brac{A}\,.
		\end{equation}
		\begin{proof}
			It is evident that, from equation \eqref{eq:remark4_inner_prod_ABphi}, it follows that
			\begin{equation}\label{eq:remark4_norm_ABphi_D1}
				\norm{A^{\nicefrac{1}{2}} B\varphi} \leq \cstrmkthree \norm{A\varphi}\,,\quad \forall \varphi \in D_1\,.
			\end{equation}
			By introducing the notation $A\varphi = w$, inequality \eqref{eq:remark4_norm_ABphi_D1} can be expressed in the following form
			\begin{equation}\label{eq:remark4_norm_ABphi_D1_w}
				\norm{A^{\nicefrac{1}{2}} B A^{-1} w} \leq \cstrmkthree \norm{w}\,,\quad \forall w \in R_1 = AD_1\,.
			\end{equation}
			Given that $R_1$ is dense in $\H$, therefore for arbitrary $w \in \H$, there exists a sequence $w_n \in R_1$ such that $w_n \to w$. According to \eqref{eq:remark4_norm_ABphi_D1_w}, it then follows that for the sequence $w_n$, we have
			\begin{equation}\label{eq:remark4_norm_ABphi_D1_wn}
				\norm{A^{\nicefrac{1}{2}} B A^{-1} w_n} \leq \cstrmkthree \norm{w_n}\,.
			\end{equation}
			Hence, It follows that $A^{\nicefrac{1}{2}} B A^{-1} w_n$ forms a Cauchy sequence. Consequently, $B A^{-1} w_n$ is also a Cauchy sequence, given that $A^{\nicefrac{1}{2}}$ is bounded below. Furthermore, since $A^{-1} w_n$ is a Cauchy sequence (due to $A^{-1}$ being bounded) and $B$ is a closed operator, we deduce that $B A^{-1} w_n \to B A^{-1} w$. By taking the limit in inequality \eqref{eq:remark4_norm_ABphi_D1_wn} and considering that $A^{\nicefrac{1}{2}}$ is a closed operator, we obtain
			\begin{equation*}
				\norm{A^{\nicefrac{1}{2}} B A^{-1} w} \leq \cstrmkthree \norm{w}\,,\quad \forall w \in \H\,,
			\end{equation*}
			or which is the same
			\begin{equation*}
				\norm{A^{\nicefrac{1}{2}} B\varphi} \leq \cstrmkthree \norm{A\varphi}\,,\quad \forall \varphi \in D\brac{A}\,.
			\end{equation*}
		\end{proof}
	\end{remark}
	
	\begin{remark}\label{prop:remark5}
		The following inequality is satisfied
		\begin{equation*}
			\norm{A v_{k + 1}} \leq \hat{M}_k + \cstrmkfora \tau \sum_{i = 1}^{k} \norm{A u_i}\,,
		\end{equation*}
		where
		\begin{gather*}
			\hat{M}_k = \cstrmkforb \brac{ \sqrt{2} \brac{ \norm{A v_0} + \norm{A^{\nicefrac{1}{2}} \frac{\Delta v_0}{\tau}} } + \tau\norm{A\frac{\Delta v_0}{\tau}} + \tau\sum_{i = 1}^{k} \norm{A^{\nicefrac{1}{2}}f_{2,i}} }\,, \\
			\cstrmkfora = \abs{a_2} \cstrmkthree \cstrmkforb\,,\quad \cstrmkforb = \frac{\nu_0}{\hat{\gamma}}\,.
		\end{gather*}
	\end{remark}
	From the second inequality in \remref{prop:remark3}, it follows immediately that
	\begin{align*}
		\norm{A v_{k + 1}} &\leq \frac{\nu_0}{\hat{\gamma}}\left[ \sqrt{2} \brac{ \norm{A v_0} + \norm{A^{\nicefrac{1}{2}} \frac{\Delta v_0}{\tau}} } + \tau\norm{A\frac{\Delta v_0}{\tau}}\right. \\
		&+ \left.\tau\sum_{i = 1}^{k} \norm{A^{\nicefrac{1}{2}}f_{2,i}} + \abs{a_2}\tau\sum_{i = 1}^{k} \norm{A^{\nicefrac{1}{2}}B u_i} \right]\,.
	\end{align*}
	By considering inequality \eqref{eq:remark4_norm_halfA_B} from \remref{prop:remark4}, we obtain the desired estimate.
	
	\section{Uniform Boundedness of the Solution to a Discrete Problem and Its Corresponding Difference Analogues of the Derivative}\label{sec:lemma1}
	
	It should be noted that, throughout this text and in all sections, the letters $c$ and $M$, indexed with lower subscripts, denote positive constants.
	
	In this \hyperref[sec:lemma1]{section}, we demonstrate that the vectors $\Delta u_{k - 1} / \tau$, $\Delta v_{k - 1} / \tau$, $A^{\nicefrac{1}{2}} u_k$, and $L^{\nicefrac{1}{2}} v_k$ are uniformly bounded, which is essential for proving the convergence of the approximate solution. Additionally, establishing convergence requires us to prove the uniform boundedness of the vectors $A u_k$ and $A^{\nicefrac{1}{2}} \Delta u_{k - 1} / {\tau}$. It is important to note that this issue is nontrivial, as the energy method does not yield a recurrence inequality that would allow the application of the telescoping series cancellation technique. The \hyperref[sec:lemma2]{subsequent section} addresses this by demonstrating that the vectors $A u_k$ and $A^{\nicefrac{1}{2}} \Delta u_{k - 1} / {\tau}$ are locally uniformly bounded.
	\begin{lemma}\label{prop:lemma1}
		Consider the sequences of vectors $\Delta u_{k - 1} / \tau$, $\Delta v_{k - 1} / \tau$, $A^{\nicefrac{1}{2}} u_k$, and $L^{\nicefrac{1}{2}} v_k$ for $k = 1,2,\ldots,n$. These sequences are uniformly bounded, meaning that there exist constants $M_j$ for $j = 1,2,3,4$, independent of $n$, such that the following inequalities are satisfied:
		\begin{equation*}
			\norm{\frac{\Delta u_{k - 1}}{\tau}} \leq \cstmone\,,\quad \norm{\frac{\Delta v_{k - 1}}{\tau}} \leq \cstmtwo\,,\quad \norm{A^{\nicefrac{1}{2}} u_k} \leq \cstmthree\,,\quad \norm{L^{\nicefrac{1}{2}} v_k} \leq \cstmfour\,,\quad k = 1,2,\ldots,n\,.
		\end{equation*}
	\end{lemma}
	\begin{proof}
		By evaluating the inner product of both sides of equation \eqref{eq:semi_discrete_timosh_1} with $u_{k + 1} - u_{k - 1} = \Delta u_k + \Delta u_{k - 1}$ and considering the properties of the operator $A$, which is self-adjoint and positive-definite, we arrive at
		\begin{align}\label{eq:lemma1_main_equality}
			{\norm{\frac{\Delta u_k}{\tau}}}^{2} + \frac{1}{2} \brac{\alpha + \beta {\norm{A^{\nicefrac{1}{2}} u_k}}^{2}} {\norm{A^{\nicefrac{1}{2}} u_{k + 1}}}^{2} &= {\norm{\frac{\Delta u_{k - 1}}{\tau}}}^{2} + \frac{1}{2} \brac{\alpha + \beta {\norm{A^{\nicefrac{1}{2}} u_k}}^{2}} {\norm{A^{\nicefrac{1}{2}} u_{k - 1}}}^{2}\nonumber \\
			&+ \brac{f_{1,k} - a_1 B v_k,\Delta u_k} + \brac{f_{1,k} - a_1 B v_k,\Delta u_{k - 1}}\,,
		\end{align}
		Let us denote
		\begin{equation*}
			\alpha_{1,k} = {\norm{\frac{\Delta u_{k - 1}}{\tau}}}^{2}\,,\quad \gamma_{1,k} = {\norm{A^{\nicefrac{1}{2}} u_k}}^{2}\,.
		\end{equation*}
		Employing these notations, equality \eqref{eq:lemma1_main_equality} can be expressed as follows:
		\begin{align*}
			\alpha_{1,k + 1} + \frac{1}{2} \brac{\alpha + \beta \gamma_{1,k}} \gamma_{1,k + 1} &= \alpha_{1,k} + \frac{1}{2} \brac{\alpha + \beta \gamma_{1,k}} \gamma_{1,k - 1} \\
			&+ \brac{f_{1,k} - a_1 B v_k,\Delta u_k} + \brac{f_{1,k} - a_1 B v_k,\Delta u_{k - 1}}\,.
		\end{align*}
		By applying the Cauchy-Schwarz inequality to the right-hand side of the given equality, one can conclude that
		\begin{align*}
			\alpha_{1,k + 1} + \frac{1}{2} \brac{\alpha + \beta \gamma_{1,k}} \gamma_{1,k + 1} &\leq \alpha_{1,k} + \frac{1}{2} \brac{\alpha + \beta \gamma_{1,k}} \gamma_{1,k - 1} \\
			&+ \tau \brac{\sqrt{\alpha_{1,k}} + \sqrt{\alpha_{1,k + 1}}} \norm{f_{1,k} - a_1 B v_k}\,.
		\end{align*}
		Consequently, we derive the following result
		\begin{equation}\label{eq:lemma1_ineq_lambda1_eps1}
			\lambda_{1,k + 1} \leq \lambda_{1,k} + \varepsilon_{1,k}\,,
		\end{equation}
		where
		\begin{equation*}
			\lambda_{1,k} = \alpha_{1,k} + \frac{1}{2} \brac{\alpha + \beta \gamma_{1,k - 1}} \gamma_{1,k}\,, \quad \varepsilon_{1,k} = \frac{1}{2} \alpha \brac{\gamma_{1,k - 1} - \gamma_{1,k}} + \tau \brac{\sqrt{\alpha_{1,k}} + \sqrt{\alpha_{1,k + 1}}} \norm{f_{1,k} - a_1 B v_k}\,.
		\end{equation*}
		By involving the telescoping series cancellation technique to inequality \eqref{eq:lemma1_ineq_lambda1_eps1}, the following result is obtained
		\begin{align*}
			\lambda_{1,k + 1} &\leq \lambda_{1,1} + \sum_{i = 1}^{k} \varepsilon_{1,i} \\
			&= \lambda_{1,1} + \frac{1}{2} \alpha \sum_{i = 1}^{k} \brac{\gamma_{1,i - 1} - \gamma_{1,i}} + \tau \sum_{i = 1}^{k} \brac{\sqrt{\alpha_{1,i}} + \sqrt{\alpha_{1,i + 1}}} \norm{f_{1,i} - a_1 B v_i} \\
			&= \lambda_{1,1} + \frac{1}{2} \alpha \brac{\gamma_{1,0} - \gamma_{1,k}} + \tau \sum_{i = 1}^{k} \brac{\sqrt{\alpha_{1,i}} + \sqrt{\alpha_{1,i + 1}}} \norm{f_{1,i} - a_1 B v_i}\,.
		\end{align*}
		Subsequently, by rearranging the terms and taking into account that $\alpha_{1,i} \leq \lambda_{1,i}$, we derive the following result
		\begin{align*}
			\lambda_{1,k + 1} + \frac{1}{2} \alpha \gamma_{1,k} \leq \lambda_{1,1} + \frac{1}{2} \alpha \gamma_{1,0} + \tau \sum_{i = 1}^{k} \brac{\sqrt{\lambda_{1,i}} + \sqrt{\lambda_{1,i + 1}}} \norm{f_{1,i} - a_1 B v_i}\,.
		\end{align*}
		Upon introducing the notation
		\begin{equation*}
			\delta_{1,k} = \sqrt{\lambda_{1,k} + \frac{1}{2} \alpha \gamma_{1,k - 1}}\,,
		\end{equation*}
		it follows that
		\begin{equation}\label{eq:lemma1_ineq_delta1}
			\delta_{1,k + 1}^{2} \leq \delta_{1,1}^{2} + \tau \sum_{i = 1}^{k} \brac{\delta_{1,i} + \delta_{1,i + 1}} \norm{f_{1,i} - a_1 B v_i}\,.
		\end{equation}
		By employing the technique for inequality \eqref{eq:lemma1_ineq_delta1} as described in \citepbf{Theorem 2.1}{RogavaTsiklauri2014}, we arrive at the following conclusion
		\begin{equation}\label{eq:lemma1_final_ineq_dlt1}
			\delta_{1,k + 1} \leq \delta_{1,1} + 2 \tau \sum_{i = 1}^{k} \norm{f_{1,i} - a_1 B v_i}\,.
		\end{equation}
		
		By performing the inner product of both sides of equation \eqref{eq:semi_discrete_timosh_2} with $v_{k + 1} - v_{k - 1} = \Delta v_k + \Delta v_{k - 1}$, and considering the self-adjoint and positive-definite nature of the operator $A$, one can immediately deduce that
		\begin{equation*}
			\alpha_{2,k + 1} + \frac{1}{2} \gamma_{2,k + 1} = \alpha_{2,k} + \frac{1}{2} \gamma_{2,k - 1} + \brac{f_{2,k} - a_2 B u_k,\Delta v_k} + \brac{f_{2,k} - a_2 B u_k,\Delta v_{k - 1}}\,,
		\end{equation*}
		where
		\begin{equation*}
			\alpha_{2,k} = {\norm{\frac{\Delta v_{k - 1}}{\tau}}}^{2}\,,\quad \gamma_{2,k} = {\norm{L^{\nicefrac{1}{2}} v_k}}^{2}\,,\quad L = \gamma A + \delta C\,.
		\end{equation*}
		If we adopt the approach employed to establish inequality \eqref{eq:lemma1_final_ineq_dlt1}, the ensuing result follows
		\begin{equation}\label{eq:lemma1_final_ineq_dlt2}
			\delta_{2,k + 1} \leq \delta_{2,1} + 2 \tau \sum_{i = 1}^{k} \norm{f_{2,i} - a_2 B u_i}\,.
		\end{equation}
		Here,
		\begin{equation*}
			\delta_{2,k} = \sqrt{\lambda_{2,k} + \frac{1}{2} \gamma_{2,k - 1}}\,,\quad \lambda_{2,k} = \alpha_{2,k} + \frac{1}{2} \gamma_{2,k}\,.
		\end{equation*}
		By summing inequalities \eqref{eq:lemma1_final_ineq_dlt1} and \eqref{eq:lemma1_final_ineq_dlt2}, taking into account condition \eqref{eq:B_cond}, and introducing the notation $\delta_k = \delta_{1,k} + \delta_{2,k}$, it consequently follows that
		\begin{align}\label{eq:lemma1_ineq_delta}
			\delta_{k + 1} &\leq \delta_1 + 2 \tau \sum_{i = 1}^{k} \brac{\norm{f_{1,i}} + \norm{f_{2,i}}} + 2 \tau \sum_{i = 1}^{k}\brac{\abs{a_1} \norm{B v_i} + \abs{a_2} \norm{B u_i}}\nonumber \\
			&\leq \delta_1 + 2 \tau \sum_{i = 1}^{k} \brac{\norm{f_{1,i}} + \norm{f_{2,i}}} + \cstDeltaIneq \tau \sum_{i = 1}^{k} \brac{\sqrt{\gamma_{1,i}} + \sqrt{\gamma_{2,i}}}\,,\quad \cstDeltaIneq = 2 \cstBoper \max\brac{\abs{a_1},\abs{a_2}}\,.
		\end{align}
		Observe that the following straightforward inequality holds
		\begin{align*}
			\delta_k = \delta_{1,k} + \delta_{2,k} \geq \frac{1}{\sqrt{2}} \brac{\sqrt{\alpha}\sqrt{\gamma_{1,k}} + \sqrt{\gamma_{2,k}}} \geq \frac{1}{\sqrt{2}} \min\brac{1,\sqrt{\alpha}}\brac{\sqrt{\gamma_{1,k}} + \sqrt{\gamma_{2,k}}}\,,
		\end{align*}
		Hence, it follows that
		\begin{equation}\label{eq:lemma1_sqrt_gamma}
			\sqrt{\gamma_{1,k}} + \sqrt{\gamma_{2,k}} \leq \cstGammaIneq \delta_k\,,\quad \cstGammaIneq = \sqrt{2} \max\brac{1,\frac{1}{\sqrt{\alpha}}}\,.
		\end{equation}
		Considering inequality \eqref{eq:lemma1_sqrt_gamma} in the estimation of \eqref{eq:lemma1_ineq_delta} yields the finding that
		\begin{equation*}
			\delta_{k + 1} \leq \delta_1 + \cstFinalGammaIneq \tau \sum_{i = 1}^{k} \delta_i + 2 \tau \sum_{i = 1}^{k} \brac{\norm{f_{1,i}} + \norm{f_{2,i}}}\,.
		\end{equation*}
		Hence, through employing the discrete Gr\"{o}nwall-type inequality ({\cf}, {\eg}, \citepbf{Lemma 3.1}{RogavaVashakidze2024}), it can be established that
		\begin{equation}\label{eq:lemma1_gronwall_ineqt}
			\delta_{k + 1} \leq e^{\cstFinalGammaIneq t_k} \brac{\delta_1 + 2 \tau \sum_{i = 1}^{k} \brac{\norm{f_{1,i}} + \norm{f_{2,i}}}}\,.
		\end{equation}
		Upon consideration of the estimate,
		\begin{equation*}
			\sum_{i = 1}^{k} \brac{\norm{f_{1,i}} + \norm{f_{2,i}}} \leq k \max_{1 \leq i \leq k} \brac{\norm{f_{1,i}} + \norm{f_{2,i}}}\,,
		\end{equation*}
		it is evident that inequality \eqref{eq:lemma1_gronwall_ineqt} can be formulated as
		\begin{equation*}
			\delta_{k + 1} \leq e^{\cstFinalGammaIneq T} \brac{\delta_1 + 2 T \max_{1 \leq i \leq n} \brac{\norm{f_{1,i}} + \norm{f_{2,i}}}}\,.
		\end{equation*}
		Thus, it can be inferred that $\alpha_{1,k}$, $\alpha_{2,k}$, $\gamma_{1,k}$, and $\gamma_{2,k}$ are uniformly bounded.
	\end{proof}
	
	\section{Uniform Boundedness of High-Order Terms Corresponding to the Solution of the Discrete Problem}\label{sec:lemma2}
	
	It should be emphasized that the proof of the convergence of the approximate solution depends, among other considerations, on the uniform boundedness of the vectors $A^{\nicefrac{1}{2}} \Delta u_{k - 1} / {\tau}$ and $A u_k$. This is quite natural, as the behavior of the solution to the discrete problem is predominantly governed by the first equation \eqref{eq:semi_discrete_timosh_1} of the system, which represents a difference analogue of the nonlinear Kirchhoff equation (see \citep{RogavaVashakidze2024}), excluding the term $a_1 B v_k$.
	
	In this \hyperref[sec:lemma2]{section}, we establish that the vectors $A^{\nicefrac{1}{2}} \Delta u_{k - 1} / {\tau}$ and $A u_k$ are locally uniformly bounded. The proof of this result relies on the nonlinear inequality outlined in \lemref{lemma:rogava-tsiklauri1}.
	\begin{lemma}\label{prop:lemma2}
		The sequences of vectors $A^{\nicefrac{1}{2}} \Delta u_{k - 1} / {\tau}$ and $A u_k$ are locally uniformly bounded. Specifically, there exists $\overline{T} > 0$ such that
		\begin{equation*}
			\norm{A^{\nicefrac{1}{2}} \frac{\Delta u_{k - 1}}{\tau}} \leq \cstmfive\,,\quad \norm{A u_k} \leq \cstmsix\,,\quad k = 1, 2, \ldots, \qbrac{\frac{\overline{T}}{\tau}}\,,
		\end{equation*}
		where $\cstmfive$ and $\cstmsix$ are positive constants, each dependent on the parameter $\overline{T}$.
	\end{lemma}
	\begin{proof}
		Consider taking the inner product of both sides of equation \eqref{eq:semi_discrete_timosh_1} with $A\brac{u_{k + 1} - u_{k - 1}} = A\brac{\Delta u_k} + A\brac{\Delta u_{k - 1}}$. By leveraging the properties of the operator $A$, which is self-adjoint and positive-definite, we obtain:
		\begin{align}\label{eq:lemma2_first_norm_eq}
			{\norm{\frac{1}{\tau} A^{\nicefrac{1}{2}} \brac{\Delta u_k}}}^2 &+ \frac{1}{2} \brac{\alpha + \beta {\norm{A^{\nicefrac{1}{2}} u_k}}^{2}} {\norm{A u_{k + 1}}}^2\nonumber \\
			&= {\norm{\frac{1}{\tau} A^{\nicefrac{1}{2}} \brac{\Delta u_{k - 1}}}}^2 + \frac{1}{2} \brac{\alpha + \beta {\norm{A^{\nicefrac{1}{2}} u_k}}^{2}} {\norm{A u_{k - 1}}}^2\nonumber \\
			&+ \brac{A^{\nicefrac{1}{2}} \brac{f_{1,k} - a_1 B v_k},A^{\nicefrac{1}{2}}\brac{\Delta u_k}} + \brac{A^{\nicefrac{1}{2}} \brac{f_{1,k} - a_1 B v_k},A^{\nicefrac{1}{2}}\brac{\Delta u_{k - 1}}}\,.
		\end{align}
		In this context, we assume that $f_{1,k} - a_1 B v_k$ belongs to $D\brac{A^{\nicefrac{1}{2}}}$.
		
		By employing the Cauchy-Schwarz inequality along with inequality \eqref{eq:remark4_norm_halfA_B} from \remref{prop:remark4}, the following result can be derived
		\begin{align}\label{eq:lemma2_f_k_cauchy_schwarz}
			&\abs{\brac{A^{\nicefrac{1}{2}} \brac{f_{1,k} - a_1 B v_k},A^{\nicefrac{1}{2}}\brac{\Delta u_k}} + \brac{A^{\nicefrac{1}{2}} \brac{f_{1,k} - a_1 B v_k},A^{\nicefrac{1}{2}}\brac{\Delta u_{k - 1}}}}\nonumber \\
			\leq& \brac{ \norm{A^{\nicefrac{1}{2}} f_{1,k}} + \abs{a_1} \cstrmkthree \norm{A v_k} } \tau \brac{ \norm{\frac{1}{\tau} A^{\nicefrac{1}{2}} \brac{\Delta u_k}} + \norm{\frac{1}{\tau} A^{\nicefrac{1}{2}} \brac{\Delta u_{k - 1}}} }\,.
		\end{align}
		Let us introduce the following denotations:
		\begin{equation*}
			\widetilde{\alpha}_{1,k} = {\norm{\frac{1}{\tau} A^{\nicefrac{1}{2}} \brac{\Delta u_{k - 1}}}}^2\,,\quad \beta_{1,k} = {\norm{A u_k}}^2\,,\quad \gamma_{1,k} = {\norm{A^{\nicefrac{1}{2}} u_k}}^{2}\,,\quad \sigma_{1,k} = \max_{1 \leq i \leq k} \norm{A^{\nicefrac{1}{2}} f_{1,i}}\,.
		\end{equation*}
		By substituting the notations introduced in the preceding step and employing inequality \eqref{eq:lemma2_f_k_cauchy_schwarz}, we can rewrite equality \eqref{eq:lemma2_first_norm_eq} as follows
		\begin{align}\label{eq:lemma2_notat_ineq_main}
			\widetilde{\alpha}_{1,k + 1} + \frac{1}{2} \brac{\alpha + \beta \gamma_{1,k}} \beta_{1,k + 1} &\leq \widetilde{\alpha}_{1,k} + \frac{1}{2} \brac{\alpha + \beta \gamma_{1,k}} \beta_{1,k - 1}\nonumber \\
			&+ \brac{ \sigma_{1,k} + \abs{a_1} \cstrmkthree \norm{A v_k} } \tau \brac{ \sqrt{\widetilde{\alpha}_{1,k + 1}} + \sqrt{\widetilde{\alpha}_{1,k}} }\,.
		\end{align}
		By incorporating \remref{prop:remark5} into inequality \eqref{eq:lemma2_notat_ineq_main}, we obtain the subsequent result
		\begin{align*}
			\widetilde{\alpha}_{1,k + 1} &+ \frac{1}{2} \brac{\alpha + \beta \gamma_{1,k}} \beta_{1,k + 1} \leq \widetilde{\alpha}_{1,k} + \frac{1}{2} \brac{\alpha + \beta \gamma_{1,k}} \beta_{1,k - 1} \\
			&+ \widetilde{M}_{k - 1} \tau \brac{ \sqrt{\widetilde{\alpha}_{1,k + 1}} + \sqrt{\widetilde{\alpha}_{1,k}} } + \cstmainalph \tau^2 \brac{ \sqrt{\widetilde{\alpha}_{1,k + 1}} + \sqrt{\widetilde{\alpha}_{1,k}} } \sum_{i = 1}^{k - 1} \sqrt{\beta_{1,i}}\,,
		\end{align*}
		where $\widetilde{M}_{k - 1} = \sigma_{1,k} + \abs{a_1} \cstrmkthree \hat{M}_{k - 1}$ and $\cstmainalph = \abs{a_1} \cstrmkthree \cstrmkfora = \abs{a_1 a_2} \cstrmkthree^2 \cstrmkforb$.
		
		Continuing from the previous step, by adding the term $\frac{1}{2} \brac{\alpha + \beta \gamma_{1,k}} \beta_{1,k}$ to both sides of the previous inequality, and introducing the notation $\widetilde{\beta}_{1,k} = \frac{1}{2} \brac{\beta_{1,k} + \beta_{1,k - 1}}$, we obtain
		\begin{align}\label{eq:lemma2_tilde_M}
			\widetilde{\alpha}_{1,k + 1} &+ \brac{\alpha + \beta \gamma_{1,k}} \widetilde{\beta}_{1,k + 1} \leq \widetilde{\alpha}_{1,k} + \brac{\alpha + \beta \gamma_{1,k - 1}} \widetilde{\beta}_{1,k} + \beta\brac{ \gamma_{1,k} - \gamma_{1,k - 1} }\widetilde{\beta}_{1,k}\nonumber \\
			&+ \widetilde{M}_{k - 1} \tau \brac{ \sqrt{\widetilde{\alpha}_{1,k + 1}} + \sqrt{\widetilde{\alpha}_{1,k}} } + \cstmainalph \tau^2 \brac{ \sqrt{\widetilde{\alpha}_{1,k + 1}} + \sqrt{\widetilde{\alpha}_{1,k}} } \sum_{i = 1}^{k - 1} \sqrt{\beta_{1,i}}\,.
		\end{align}
		To evaluate the absolute value of the difference $\abs{\gamma_{1,k} - \gamma_{1,k - 1}}$, we shall employ \lemref{prop:lemma1}
		\begin{equation}\label{eq:lemma2_diff_gamma}
			\begin{aligned}
				\abs{\gamma_{1,k} - \gamma_{1,k - 1}} &\leq 2 \cstmthree \norm{A^{\nicefrac{1}{2}} \brac{\Delta u_{k - 1}}} = 2 \cstmthree \tau \sqrt{\widetilde{\alpha}_{1,k}}\,.
			\end{aligned}
		\end{equation}
		Through the use of inequalities \eqref{eq:lemma2_diff_gamma} and $\beta_{1,i} \leq 2 \widetilde{\beta}_{1,i}$, we can reformulate inequality \eqref{eq:lemma2_tilde_M} as follows
		\begin{align}\label{eq:lemma2_tilde_lam_ineq}
			\widetilde{\lambda}_{1,k + 1} &\leq \widetilde{\lambda}_{1,k} + 2 \cstmthree \beta \tau \sqrt{\widetilde{\alpha}_{1,k}} \widetilde{\beta}_{1,k} +
			\widetilde{M}_{k - 1} \tau \brac{ \sqrt{\widetilde{\alpha}_{1,k + 1}} + \sqrt{\widetilde{\alpha}_{1,k}} }\nonumber \\
			&+ \sqrt{2} \cstmainalph \brac{ \sqrt{\widetilde{\alpha}_{1,k + 1}} + \sqrt{\widetilde{\alpha}_{1,k}} } \tau^2 \sum_{i = 1}^{k - 1} \sqrt{\widetilde{\beta}_{1,i}}\,,
		\end{align}
		where $\widetilde{\lambda}_{1,k} = \widetilde{\alpha}_{1,k} + \brac{\alpha + \beta \gamma_{1,k - 1}} \widetilde{\beta}_{1,k}$.
		
		Given the inequalities $\widetilde{\alpha}_{1,k} \leq \widetilde{\lambda}_{1,k}$ and $\widetilde{\beta}_{1,k} \leq \frac{1}{\alpha} \widetilde{\lambda}_{1,k}$, inequality \eqref{eq:lemma2_tilde_lam_ineq} should be rewritten as
		\begin{align*}
			\widetilde{\lambda}_{1,k + 1} &\leq \widetilde{\lambda}_{1,k} + \csttildelama \tau \sqrt{\widetilde{\lambda}_{1,k}} \widetilde{\lambda}_{1,k} + \widetilde{M}_{k} \tau \brac{ \sqrt{\widetilde{\lambda}_{1,k + 1}} + \sqrt{\widetilde{\lambda}_{1,k}} } \\
			&+ \csttildelamb \tau^2 \brac{ \sqrt{\widetilde{\lambda}_{1,k + 1}} + \sqrt{\widetilde{\lambda}_{1,k}} } \sum_{i = 1}^{k} \sqrt{\widetilde{\lambda}_{1,i}}\,,\quad \csttildelama = \frac{2 \cstmthree \beta}{\alpha}\,,\quad \csttildelamb = \frac{2 \cstmainalph}{\sqrt{2 \alpha}}\,.
		\end{align*}
		
		Let us denote the maximum of $\widetilde{\lambda}_{1,i}$ for $1 \leq i \leq k$ by $\hat{\lambda}_{1,k}$, {\ie}, $\hat{\lambda}_{1,k} = \max_{1 \leq i \leq k} \widetilde{\lambda}_{1,i}$. Let $\widetilde{\lambda}_{1,i + 1}$ (for $i \leq k$) attain its maximum at $i = j$. Thus, we have $\hat{\lambda}_{1,k + 1} = \widetilde{\lambda}_{1,j + 1}$. Clearly, from the aforementioned inequality, we find that
		\begin{align*}
			\widetilde{\lambda}_{1,j + 1} &\leq \hat{\lambda}_{1,j} + \csttildelama \tau \sqrt{\hat{\lambda}_{1,j}} \hat{\lambda}_{1,j} + \widetilde{M}_{j} \tau \brac{ \sqrt{\hat{\lambda}_{1,j + 1}} + \sqrt{\hat{\lambda}_{1,j}} } \\
			&+ \csttildelamb \tau^2 \brac{ \sqrt{\hat{\lambda}_{1,j + 1}} + \sqrt{\hat{\lambda}_{1,j}} } \brac{ k \sqrt{\hat{\lambda}_{1,j}} }\,.
		\end{align*}
		Hence, it follows that
		\begin{align*}
			\hat{\lambda}_{1,k + 1} &\leq \hat{\lambda}_{1,k} + \csttildelama \tau \sqrt{\hat{\lambda}_{1,k}} \hat{\lambda}_{1,k} + \widetilde{M}_{k} \tau \brac{ \sqrt{\hat{\lambda}_{1,k + 1}} + \sqrt{\hat{\lambda}_{1,k}} } \\
			&+ \csthatlam \tau \brac{ \sqrt{\hat{\lambda}_{1,k + 1}} + \sqrt{\hat{\lambda}_{1,k}} } \sqrt{\hat{\lambda}_{1,k}}\,,\quad \csthatlam = T \csttildelamb\,.
		\end{align*}
		If we rearrange the terms on the right-hand side of the previously stated inequality, it takes the form
		\begin{equation}\label{eq:lemma2_hat_lam_rearranged}
			\hat{\lambda}_{1,k + 1} \leq \brac{ 1 + \csthatlam \tau +  \csttildelama \tau \sqrt{\hat{\lambda}_{1,k}} } \hat{\lambda}_{1,k} + \widetilde{M}_{k} \tau \brac{ \sqrt{\hat{\lambda}_{1,k + 1}} + \sqrt{\hat{\lambda}_{1,k}} } + \csthatlam \tau \sqrt{\hat{\lambda}_{1,k} \hat{\lambda}_{1,k + 1}}\,.
		\end{equation}
		By employing the well-known Young's inequality for products, one can deduce that:
		\begin{align}
			\sqrt{\hat{\lambda}_{1,k} \hat{\lambda}_{1,k + 1}} &\leq \frac{1}{3} \brac{ \hat{\lambda}_{1,k}^{\nicefrac{3}{2}} + 2 \hat{\lambda}_{1,k + 1}^{\nicefrac{3}{4}} }\,,\label{eq:lemma2_young_ineq_hat_lam} \\
			\sqrt{\hat{\lambda}_{1,k + 1}} + \sqrt{\hat{\lambda}_{1,k}} &\leq \frac{5}{6} + \frac{1}{2} \hat{\lambda}_{1,k} + \frac{2}{3} \hat{\lambda}_{1,k + 1}^{\nicefrac{3}{4}}\,.\label{eq:lemma2_young_ineq_sum_hat_lam}
		\end{align}
		By incorporating inequalities \eqref{eq:lemma2_young_ineq_hat_lam} and \eqref{eq:lemma2_young_ineq_sum_hat_lam} into bound \eqref{eq:lemma2_hat_lam_rearranged}, one can immediately deduce the following result
		\begin{equation}\label{eq:lemma2_hat_lam_final}
			\hat{\lambda}_{1,k + 1} \leq \brac{ 1 + \cstlemtwoa \tau +  \cstlemtwob \tau \sqrt{\hat{\lambda}_{1,k}} } \hat{\lambda}_{1,k} + \cstlemtwoc \tau \hat{\lambda}_{1,k + 1}^{\nicefrac{3}{4}} + \cstlemtwod \tau\,.
		\end{equation}
		Here,
		\begin{equation*}
			\cstmaxMk = \max\limits_{1 \leq k \leq n - 1} \widetilde{M}_{k}\,,\quad \cstlemtwoa = \csthatlam + \frac{\cstmaxMk}{2}\,,\quad \cstlemtwob = \csttildelama + \frac{\csthatlam}{3}\,,\quad \cstlemtwoc = \frac{2}{3}\brac{ \csthatlam + \cstmaxMk }\,,\quad \cstlemtwod = \frac{5 \cstmaxMk}{6}\,.
		\end{equation*}
		By introducing the notations
		\begin{equation*}
			w_k = \brac{ 1 + \cstlemtwoa \tau +  \cstlemtwob \tau \sqrt{\hat{\lambda}_{1,k}} } \hat{\lambda}_{1,k} + \cstlemtwod \tau
		\end{equation*}
		and $y_{k + 1} = \sqrt[4]{\hat{\lambda}_{1,k + 1}}$, the estimate \eqref{eq:lemma2_hat_lam_final} should be rewritten as follows
		\begin{equation*}
			y_{k + 1}^4 - \cstlemtwoc \tau y_{k + 1}^3 - w_k \leq 0\,.
		\end{equation*}
		From here, it follows
		\begin{equation}\label{eq:lemma2_fourth_ord_ineq}
			\brac{\frac{y_{k + 1}}{\cstlemtwoc \tau}}^4 - \brac{\frac{y_{k + 1}}{\cstlemtwoc \tau}}^3 - \frac{w_k}{\brac{\cstlemtwoc \tau}^4} \leq 0\,.
		\end{equation}
		
		Consider the polynomial $P\brac{ \xi } = \xi^4 - \xi^3 - b$, where $b = w_k / \brac{\cstlemtwoc \tau}^4$. This polynomial corresponds to the left-hand side of inequality \eqref{eq:lemma2_fourth_ord_ineq}. It is evident that $P\brac{ \xi }$ has exactly two real roots, one positive and one negative. Since $P\brac{ 0 } = -b < 0$ and $P\brac{ 1 + \sqrt[4]{b} } > 0$, it follows that the positive root of $P\brac{ \xi }$ must be less than $1 + \sqrt[4]{b}$. Considering these facts, from inequality \eqref{eq:lemma2_fourth_ord_ineq}, we have
		\begin{equation*}
			\frac{y_{k + 1}}{\cstlemtwoc \tau} \leq 1 + \sqrt[4]{b}\,,
		\end{equation*}
		or which is the same
		\begin{equation*}
			y_{k + 1} \leq \cstlemtwoc \tau + \sqrt[4]{w_k}\,.
		\end{equation*}
		By squaring both sides of the given inequality and substituting the product term with the sum of squares, we obtain the following
		\begin{equation*}
			y_{k + 1}^2 \leq \brac{1 + \cstlemtwoc \tau} \brac{\cstlemtwoc \tau + \sqrt{w_k}}\,.
		\end{equation*}
		If we apply the same transformation as in the previous case, we have
		\begin{equation}\label{eq:lemma2_fourth_y}
			y_{k + 1}^4 \leq \brac{1 + \cstlemtwoc \tau}^3 \brac{\cstlemtwoc \tau + w_k}\,.
		\end{equation}
		We may now revert to the previous notation. In this case, inequality \eqref{eq:lemma2_fourth_y} takes the form
		\begin{equation}\label{eq:lemma2_hat_lam_bef_div}
			\hat{\lambda}_{1,k + 1} \leq \brac{1 + \cstmaxab \tau}^4 \brac{ 1 +  \frac{\cstlemtwob \tau}{1 + \cstlemtwoa \tau} \sqrt{\hat{\lambda}_{1,k}} } \hat{\lambda}_{1,k} + \brac{1 + \cstmaxab \tau}^3 \cstykplusonefin \tau\,,
		\end{equation}
		where $\cstykplusonefin = \cstlemtwoc + \cstlemtwod$ and $\cstmaxab = \max\brac{ \cstlemtwoa,\cstlemtwoc }$.
		
		If we replace $\brac{1 + \cstmaxab \tau}^4$ with $1 + \cstfourthpoly \tau$ in inequality \eqref{eq:lemma2_hat_lam_bef_div}, and subsequently divide both sides of the resulting inequality by $\brac{1 + \cstfourthpoly \tau}^{k + 1}$, performing straightforward transformations, we obtain the following result
		\begin{equation*}
			\xi_{k + 1} \leq \xi_k \brac{1 + \cstsqrtfracinxi \tau \sqrt{\xi_k}} + \cstxikfin \tau\,,
		\end{equation*}
		where
		\begin{equation*}
			\xi_k = \frac{\hat{\lambda}_{1,k}}{\brac{1 + \cstfourthpoly \tau}^k}\,.
		\end{equation*}
		Consider the transformation where $\overline{\tau} = \cstoverlinetau \tau$ with $\cstoverlinetau = \max\brac{\cstsqrtfracinxi,\cstxikfin}$. Then we have
		\begin{equation*}
			\xi_{k + 1} \leq \xi_k \brac{1 + \overline{\tau} \sqrt{\xi_k}} + \overline{\tau}\,.
		\end{equation*}
		From \lemref{lemma:rogava-tsiklauri1}, it follows that
		\begin{equation}\label{eq:lemma2_xi_and_hat_gamma}
			\xi_k \leq \frac{\xi}{\brac{1 - \overline{t}_k \sqrt{\xi}}^2} \leq \frac{\hat{\lambda}}{\brac{1 - \cstoverlinetau \sqrt{\hat{\lambda}} t_k}^2}\,,\quad k = 1,2,\ldots,m\,,
		\end{equation}
		where
		\begin{equation*}
			\xi = \max\brac{1,\xi_1} \leq \max\brac{1,\hat{\lambda}_{1,1}} = \hat{\lambda}\,,\quad \overline{t}_k = k \overline{\tau} = \cstoverlinetau t_k < \frac{1}{\sqrt{\hat{\lambda}}} \leq \frac{1}{\sqrt{\xi}}\,.
		\end{equation*}
		Following the established notation, we have
		\begin{equation}\label{eq:lemma2_xi_with_exp_denom}
			\xi_k = \frac{\hat{\lambda}_{1,k}}{\brac{1 + \cstfourthpoly \tau}^k} \geq \frac{\hat{\lambda}_{1,k}}{e^{\cstfourthpoly t_k}}\,.
		\end{equation}
		Using inequality \eqref{eq:lemma2_xi_and_hat_gamma} along with \eqref{eq:lemma2_xi_with_exp_denom}, we obtain the following estimate
		\begin{equation}\label{eq:lemma2_final_ineq}
			\hat{\lambda}_{1,k} \leq \frac{\hat{\lambda}}{\brac{1 - \cstoverlinetau \sqrt{\hat{\lambda}} t_k}^2} e^{\cstfourthpoly t_k}\,,\quad k = 1,2,\ldots,m\,.
		\end{equation}
		
		It should be noted that the value of $m$ in the inequality \eqref{eq:lemma2_final_ineq} is influenced both by the coefficient of $t_k$ (appearing in the denominator of the fraction) and by the number of subdivisions of the time interval, denoted by $n$. The coefficient of $t_k$ can be explicitly determined based on the data provided in problem \eqref{eq:abst_timoshenko1}-\eqref{eq:abst_timoshenko3}, while also taking into account the value of $T$. Furthermore, the inequality ${\hat \lambda} \leq {\widetilde{M}}$ is satisfied, where ${\widetilde{M}}$ represents a positive constant that is dependent on the original data from the stated problem \eqref{eq:abst_timoshenko1}-\eqref{eq:abst_timoshenko3}, as well as the value of $T$.
		
		The following inequality is derived from \eqref{eq:lemma2_final_ineq}
		\begin{equation}\label{eq:lemma2_final_result_loc_bound}
			\hat{\lambda}_{1,k} \leq \frac{\widetilde{M}}{\brac{1 - \overline{M}{\,}\overline{T}}^2} e^{\cstfourthpoly \overline{T}}\,,\quad k = 1, 2, \ldots, \qbrac{\frac{\overline{T}}{\tau}}\,,
		\end{equation}
		where $\overline{M} = \cstoverlinetau \sqrt{\widetilde{M}}$ and $\displaystyle \overline{T} = \frac{q}{\overline{M}}$, with $0 < q < 1$.
		
		The inequality given in \eqref{eq:lemma2_final_result_loc_bound} implies that the vectors $A u_k$ and $A^{\nicefrac{1}{2}} \Delta u_{k - 1} / {\tau}$ are uniformly bounded over the local interval $\qbrac{0,\overline{T}}$.
	\end{proof}
	One should observe that the local uniform boundedness of the vectors $A v_k$ follows directly from \remref{prop:remark5}.
	
	\section{Error Estimates for Approximate Solutions of the System}\label{sec:convergence}
	
	Before the theorem concerning the convergence of the scheme \eqref{eq:semi_discrete_timosh_1}-\eqref{eq:semi_discrete_timosh_2} is presented, a remark regarding the smoothness of the solutions to the problem \eqref{eq:abst_timoshenko1}-\eqref{eq:abst_timoshenko3} is made to clarify the order of convergence of the proposed symmetric three-layer semi-discrete scheme \eqref{eq:semi_discrete_timosh_1}-\eqref{eq:semi_discrete_timosh_2}. A minimum degree of smoothness in the solutions is required to ensure the well-posedness of the problem. It should be noted that this condition guarantees convergence but is insufficient for determining the order of convergence. When the smoothness of the solutions is increased by one degree, an order of convergence equal to one is achieved. Nevertheless, in this case, as well as in the previous case, the following initial conditions should be taken: $u_1 = \varphi_0 + \tau \varphi_1$ and $v_1 = \psi_0 + \tau \psi_1$. Furthermore, if the smoothness is increased by two degrees and the initial functions are specified according to formulas \eqref{eq:initial_vect01} and \eqref{eq:initial_vect02}, an additional degree of convergence is attained, resulting in a total order of two. However, any further increase in smoothness would be regarded as superfluous, since the approximation order of the scheme \eqref{eq:semi_discrete_timosh_1}-\eqref{eq:semi_discrete_timosh_2} does not exceed two.
	
	The following theorem is formulated to address the convergence of the scheme \eqref{eq:semi_discrete_timosh_1}-\eqref{eq:semi_discrete_timosh_2}.
	
	\begin{theorem}\label{prop:main_thm_convrg}
		Suppose the problem \eqref{eq:abst_timoshenko1}-\eqref{eq:abst_timoshenko3} is well-posed, and the following conditions are fulfilled:
		\begin{enumerate}[label=(\alph*)]
			\item\label{itm_theorem_a} The vectors $\varphi_0, \psi_0 \in D\brac{A}$, while the vectors $\varphi_1, \psi_1, \varphi_2$, and $\psi_2$ belong to $D\brac{A^{\nicefrac{1}{2}}}$. Additionally, the right-hand sides of equations \eqref{eq:abst_timoshenko1} and \eqref{eq:abst_timoshenko2}, specifically $f_1\brac{t}$ and $f_2\brac{t}$, are continuous functions. Moreover, $f_1\brac{t}, f_2\brac{t} \in D\brac{A^{\nicefrac{1}{2}}}$ for all $t \in \qbrac{0,T}$, and $A^{\nicefrac{1}{2}} f_1\brac{t}$ and $A^{\nicefrac{1}{2}} f_2\brac{t}$ are also continuous functions.
			\item\label{itm_theorem_b} The solutions $u\brac{t}$ and $v\brac{t}$ to the problem \eqref{eq:abst_timoshenko1}-\eqref{eq:abst_timoshenko3} are continuously differentiable up to and including third order, and the functions $u^{\prime\prime\prime}\brac{t}$ and $v^{\prime\prime\prime}\brac{t}$  satisfy the Lipschitz condition.
			\item\label{itm_theorem_c} The functions $A u\brac{t}$ and $A v\brac{t}$ are continuously differentiable. Moreover, $A u^{\prime}\brac{t}$ and $A v^{\prime}\brac{t}$ are Lipschitz continuous functions.
		\end{enumerate}
		Then there exists $\overline{T}$ $\brac{0 < \overline{T} \leq T}$ such that for the errors of the approximate solutions $z_{1,k} = u\brac{t_k} - u_k$ and $z_{2,k} = v\brac{t_k} - v_k$, the following estimates hold:
		\begin{equation*}
			\max_{1 \leq k \leq m} \norm{A^{\nicefrac{1}{2}} z_{\detectstaticarg,k}} \leq \cststatthmone \tau^2\,,\quad \max_{1 \leq k \leq m} \norm{\frac{\Delta z_{\detectstaticarg,k - 1}}{\tau}} \leq \cststatthmtwo \tau^2\,,\quad \detectstaticarg = 1,2\,,
		\end{equation*}
		where $m = \qbrac{\dfrac{\overline{T}}{\tau}}$, and $\Delta z_{\detectstaticarg,k} = z_{\detectstaticarg,k + 1} - z_{\detectstaticarg,k}$.
	\end{theorem}
	\begin{proof}
		The initial step is a standard procedure. We shall derive the system corresponding to the scheme \eqref{eq:semi_discrete_timosh_1}-\eqref{eq:semi_discrete_timosh_2} for the errors associated with the approximate solutions $z_{1,k}$ and $z_{2,k}$, and proceed to evaluate the remainder terms. Following this, equations \eqref{eq:abst_timoshenko1} and \eqref{eq:abst_timoshenko2} are reformulated at the discrete time points $t = t_k$, where $k = 1,2,\ldots,n - 1$, as follows:
		\begin{subequations}
			\begin{align}\label{eq:main_thm_convrg_scheme_u}
				\frac{\Delta^2 u \brac{t_{k - 1}}}{\tau^2} &+ \brac{\alpha + \beta {\norm{A^{\nicefrac{1}{2}} u \brac{t_k}}}^2} \frac{A u \brac{t_{k + 1}} + A u \brac{t_{k - 1}}}{2}\nonumber \\
				&= \brac{f_1\brac{t_k} - a_1 B v\brac{t_k}} + R_{1,k}\brac{\tau} + R_{2,k}\brac{\tau}\,,
			\end{align}
			\begin{equation}\label{eq:main_thm_convrg_scheme_v}
				\frac{\Delta^2 v \brac{t_{k - 1}}}{\tau^2} + \frac{L v\brac{t_{k + 1}} + L v\brac{t_{k - 1}}}{2} = \brac{f_2\brac{t_k} - a_2 B u\brac{t_k}} + R_{3,k}\brac{\tau} + R_{4,k}\brac{\tau}\,.
			\end{equation}
		\end{subequations}
		Here,
		\begin{align*}
			R_{1,k}\brac{\tau} &= \frac{\Delta^2 u \brac{t_{k - 1}}}{\tau^2} - \frac{\d^2 u\brac{t_k}}{\d t^2}\,,\quad R_{2,k}\brac{\tau} = \frac{1}{2} \brac{\alpha + \beta {\norm{A^{\nicefrac{1}{2}} u \brac{t_k}}}^2} A\brac{\Delta^2 u \brac{t_{k - 1}}}\,, \\
			R_{3,k}\brac{\tau} &= \frac{\Delta^2 v \brac{t_{k - 1}}}{\tau^2} - \frac{\d^2 v\brac{t_k}}{\d t^2}\,,\quad R_{4,k}\brac{\tau} = \frac{1}{2} L\brac{\Delta^2 v \brac{t_{k - 1}}}\,.
		\end{align*}
		
		Based on conditions \ref{itm_theorem_b} and \ref{itm_theorem_c} of \thmref{prop:main_thm_convrg}, the following estimates can be derived for the remainder terms:
		\begin{equation}\label{eq:main_thm_convrg_est_rems}
			\norm{R_{\detectvar{boundremainders},k}\brac{\tau}} \leq \boundremainders \tau^2\,,\quad \detectvar{boundremainders} = 1,2,3,4\,.
		\end{equation}
		
		By subtracting equations \eqref{eq:main_thm_convrg_scheme_u} and \eqref{eq:main_thm_convrg_scheme_v} from equations \eqref{eq:semi_discrete_timosh_1} and \eqref{eq:semi_discrete_timosh_2}, respectively, we obtain the following system of equations:
		\begin{subequations}
			\begin{gather}
				\frac{\Delta^2 z_{1,k - 1}}{\tau^2} + \brac{\alpha + \beta {\norm{A^{\nicefrac{1}{2}} u_k}}^2} \frac{A z_{1,k + 1} + A z_{1,k - 1}}{2} = \widetilde{g}_k - a_1 B z_{2,k}\,,\label{eq:main_thm_convrg_error_eqt_z_one} \\
				\frac{\Delta^2 z_{2,k - 1}}{\tau^2} + \frac{L z_{2,k + 1} + L z_{2,k - 1}}{2} = R_k\brac{\tau} - a_2 B z_{1,k}\,,\label{eq:main_thm_convrg_error_eqt_z_two}
			\end{gather}
		\end{subequations}
		where
		\begin{gather*}
			\widetilde{g}_k = \beta \brac{{\norm{A^{\nicefrac{1}{2}} u_k}}^2 - {\norm{A^{\nicefrac{1}{2}} u \brac{t_k}}}^2} \frac{A u \brac{t_{k + 1}} + A u \brac{t_{k - 1}}}{2} + R_{1,k}\brac{\tau} + R_{2,k}\brac{\tau}\,, \\
			R_k\brac{\tau} = R_{3,k}\brac{\tau} + R_{4,k}\brac{\tau}\,.
		\end{gather*}
		
		Taking the inner product of both sides of equality \eqref{eq:main_thm_convrg_error_eqt_z_one} with $z_{1,k + 1} - z_{1,k - 1} = \Delta z_{1,k} + \Delta z_{1,k - 1}$, and considering the self-adjointness and positive definiteness of the operator $A$, we arrive at the following result
		\begin{equation}\label{eq:main_thm_convrg_inner_prod_eqt}
			\overline{\lambda}_{1,k + 1} = \overline{\lambda}_{1,k} + \brac{\overline{\varepsilon}_{1,k} + \eta_{1,k} + \eta_{2,k}}\,,
		\end{equation}
		where:
		\begin{gather*}
			\overline{\lambda}_{1,k} = \overline{\alpha}_{1,k}^{2} + \frac{1}{2}\brac{\alpha + \beta \gamma_{1,k - 1}} \overline{\gamma}_{1,k}^{2}\,,\quad \overline{\alpha}_{1,k} = \norm{\frac{\Delta z_{1,k - 1}}{\tau}}\,,\quad \overline{\gamma}_{1,k} = \norm{A^{\nicefrac{1}{2}} z_{1,k}}\,, \\
			\gamma_{1,k} = {\norm{A^{\nicefrac{1}{2}} u_k}}^{2}\,,\quad \overline{\varepsilon}_{1,k} = \frac{1}{2} \alpha \brac{\overline{\gamma}_{1,k - 1}^{2} - \overline{\gamma}_{1,k}^{2}} + \frac{1}{2} \beta \brac{\gamma_{1,k} \overline{\gamma}_{1,k - 1}^{2} - \gamma_{1,k - 1} \overline{\gamma}_{1,k}^{2}}\,, \\
			\eta_{1,k} = \brac{\widetilde{g}_k,\Delta z_{1,k} + \Delta z_{1,k - 1}}\,,\quad \eta_{2,k} = -a_1 \brac{B z_{2,k},\Delta z_{1,k} + \Delta z_{1,k - 1}}\,.
		\end{gather*}
		
		The resulting recurrence relation \eqref{eq:main_thm_convrg_inner_prod_eqt} provides a ``descent'' from $k$ to $1$. Therefore, it follows from \eqref{eq:main_thm_convrg_inner_prod_eqt} that
		\begin{equation}\label{eq:main_thm_convrg_sum_inner_prod}
			\begin{aligned}
				\overline{\lambda}_{1,k + 1} &= \overline{\lambda}_{1,1} + \sum_{i = 1}^{k} \brac{\overline{\varepsilon}_{1,i} + \eta_{1,i} + \eta_{2,i}} \\
				&= \overline{\lambda}_{1,1} + \frac{1}{2} \alpha \brac{\overline{\gamma}_{1,0}^{2} - \overline{\gamma}_{1,k}^{2}} + \frac{1}{2} \beta \sum_{i = 1}^{k} \brac{\gamma_{1,i} \overline{\gamma}_{1,i - 1}^{2} - \gamma_{1,i - 1} \overline{\gamma}_{1,i}^{2}} + \sum_{i = 1}^{k} \brac{\eta_{1,i} + \eta_{2,i}}\,.
			\end{aligned}
		\end{equation}
		We observe that for the following sum, the subsequent representation is satisfied
		\begin{equation}\label{eq:main_thm_convrg_sum_term_repr}
			\sum_{i = 1}^{k} \brac{\gamma_{1,i} \overline{\gamma}_{1,i - 1}^{2} - \gamma_{1,i - 1} \overline{\gamma}_{1,i}^{2}} = \gamma_{1,1} \overline{\gamma}_{1,0}^{2} + \sum_{i = 1}^{k - 1} \overline{\gamma}_{1,i}^{2}\brac{\gamma_{1,i + 1} - \gamma_{1,i - 1}} - \gamma_{1,k - 1} \overline{\gamma}_{1,k}^{2}\,.
		\end{equation}
		By substituting the corresponding term in the equality \eqref{eq:main_thm_convrg_sum_inner_prod} with the sum representation given in expression \eqref{eq:main_thm_convrg_sum_term_repr}, the following relation is established
		\begin{equation}\label{eq:main_thm_convrg_overl_lamb_eq}
			\begin{aligned}
				\overline{\lambda}_{1,k + 1} + \frac{1}{2} \brac{\alpha + \beta \gamma_{1,k - 1}} \overline{\gamma}_{1,k}^{2} &= \overline{\lambda}_{1,1} + \frac{1}{2} \brac{\alpha + \beta \gamma_{1,1}} \overline{\gamma}_{1,0}^{2} \\
				&+ \frac{1}{2} \beta \sum_{i = 1}^{k - 1} \overline{\gamma}_{1,i}^{2}\brac{\gamma_{1,i + 1} - \gamma_{1,i - 1}} + \sum_{i = 1}^{k} \brac{\eta_{1,i} + \eta_{2,i}}\,.
			\end{aligned}
		\end{equation}
		We shall introduce the following notation
		\begin{equation*}
			\overline{\delta}_{1,k} = \sqrt{\overline{\lambda}_{1,k} + \frac{1}{2} \brac{\alpha + \beta \gamma_{1,k - 2}} \overline{\gamma}_{1,k - 1}^{2}}\,.
		\end{equation*}
		Under this denotation, assuming $\gamma_{1,-1} = \gamma_{1,1}$, equality \eqref{eq:main_thm_convrg_overl_lamb_eq} can be expressed as follows
		\begin{equation}\label{eq:main_thm_convrg_tild_delta_one}
			\overline{\delta}_{1,k + 1}^{2} = \overline{\delta}_{1,1}^{2} + \frac{1}{2} \beta \sum_{i = 1}^{k - 1} \overline{\gamma}_{1,i}^{2}\brac{\gamma_{1,i + 1} - \gamma_{1,i - 1}} + \sum_{i = 1}^{k} \brac{\eta_{1,i} + \eta_{2,i}}\,.
		\end{equation}
		
		Following \lemref{prop:lemma1} and \lemref{prop:lemma2}, an estimate for the difference $\gamma_{1,i + 1} - \gamma_{1,i - 1}$ can be derived
		\begin{align}\label{eq:main_thm_convrg_abs_diff_gamma}
			\abs{\gamma_{1,i + 1} - \gamma_{1,i - 1}} &\leq \brac{\abs{\sqrt{\gamma_{1,i + 1}} - \sqrt{\gamma_{1,i}}} + \abs{\sqrt{\gamma_{1,i}} - \sqrt{\gamma_{1,i - 1}}}} \brac{\sqrt{\gamma_{1,i + 1}} + \sqrt{\gamma_{1,i - 1}}}\nonumber \\
			&\leq \tau \brac{\norm{A^{\nicefrac{1}{2}} \frac{\Delta u_{i}}{\tau}} + \norm{A^{\nicefrac{1}{2}} \frac{\Delta u_{i - 1}}{\tau}}} \brac{\norm{A^{\nicefrac{1}{2}} u_{i + 1}} + \norm{A^{\nicefrac{1}{2}} u_{i - 1}}} \leq 4 \cstmthree \cstmfive \tau\,.
		\end{align}
		On the other hand, it is easy to see that
		\begin{equation}\label{eq:main_thm_convrg_overl_gamma_est}
			\overline{\gamma}_{1,i}^{2} \leq \frac{2}{\alpha} \overline{\delta}_{1,i}^{2}\,.
		\end{equation}
		
		By employing the Cauchy-Schwarz inequality, we can ascertain that
		\begin{equation}\label{eq:main_thm_convrg_etaone}
			\abs{\eta_{1,i}} \leq \norm{\widetilde{g}_i} \brac{\norm{\Delta z_{1,i}} + \norm{\Delta z_{1,i - 1}}} = \tau \norm{\widetilde{g}_i} \brac{\overline{\alpha}_{1,i + 1} + \overline{\alpha}_{1,i}}\,.
		\end{equation}
		We shall proceed to estimate $\norm{\widetilde{g}_i}$. To achieve this, it is necessary to assess certain quantities. By virtue of inequality \eqref{eq:main_thm_convrg_est_rems}, we conclude that $\norm{R_{1,k}\brac{\tau}} + \norm{R_{2,k}\brac{\tau}} \leq \inthmsumonetwo \tau^2$. Taking into account this fact, together with estimate \eqref{eq:main_thm_convrg_overl_gamma_est}, \lemref{prop:lemma1}, and condition \ref{itm_theorem_c} of \thmref{prop:main_thm_convrg}, we establish that the following inequality is valid
		\begin{align}\label{eq:main_thm_convrg_norm_tild_g}
			\norm{\widetilde{g}_i} &\leq \frac{\beta}{2} \abs{{\norm{A^{\nicefrac{1}{2}} u_i}}^2 - {\norm{A^{\nicefrac{1}{2}} u \brac{t_i}}}^2} \brac{\norm{A u \brac{t_{i + 1}}} + \norm{A u \brac{t_{i - 1}}}} + \norm{R_{1,i}\brac{\tau}} + \norm{R_{2,i}\brac{\tau}}\nonumber \\
			&\leq \beta \overline{\gamma}_{1,i} \brac{\cstmthree + \inthmmaxhalfa} \inthmmaxa + \inthmsumonetwo \tau^2 \leq \beta \sqrt{2 {\alpha}^{-1}} \overline{\delta}_{1,i} \brac{\cstmthree + \inthmmaxhalfa} \inthmmaxa + \inthmsumonetwo \tau^2 = \inthmtilgfin \overline{\delta}_{1,i} + \inthmsumonetwo \tau^2\,,
		\end{align}
		where $\inthmmaxhalfa = \nu^{-\nicefrac{1}{2}} \inthmmaxa$ and $\inthmmaxa = \max_{0 \leq t \leq T} \norm{A u \brac{t}}$.
		
		Given that $\overline{\alpha}_{1,i}^{2} \leq \overline{\lambda}_{1,i} \leq \overline{\delta}_{1,i}^{2}$, and in view of the derived inequality \eqref{eq:main_thm_convrg_norm_tild_g}, the following estimate for \eqref{eq:main_thm_convrg_etaone} can be established
		\begin{equation}\label{eq:main_thm_convrg_abs_eta_one}
			\abs{\eta_{1,i}} \leq \tau \brac{\inthmtilgfin \overline{\delta}_{1,i} + \inthmsumonetwo \tau^2} \brac{\overline{\delta}_{1,i + 1} + \overline{\delta}_{1,i}}\,.
		\end{equation}
		By applying the Cauchy-Schwarz inequality along with the estimate \eqref{eq:remark6_final_res}, we easily find that
		\begin{align*}
			\abs{\eta_{2,i}} &\leq \abs{a_1} \norm{B z_{2,i}} \brac{\norm{\Delta z_{1,i}} + \norm{\Delta z_{1,i - 1}}} = \tau \abs{a_1} \norm{B z_{2,i}} \brac{\overline{\alpha}_{1,i + 1} + \overline{\alpha}_{1,i}} \\
			&\leq \tau \abs{a_1} \norm{B z_{2,i}} \brac{\overline{\delta}_{1,i + 1} + \overline{\delta}_{1,i}} \leq \tau \abs{a_1} \cstBoper \norm{A^{\nicefrac{1}{2}} z_{2,i}} \brac{\overline{\delta}_{1,i + 1} + \overline{\delta}_{1,i}}\,.
		\end{align*}
		Hence, from \remref{prop:remark2}, the following inequality follows:
		\begin{equation}\label{eq:main_thm_convrg_abs_eta_two}
			\abs{\eta_{2,i}} \leq \inthmnutwo \tau \overline{\gamma}_{2,i} \brac{\overline{\delta}_{1,i + 1} + \overline{\delta}_{1,i}}\,,\quad \overline{\gamma}_{2,i} = \norm{L^{\nicefrac{1}{2}} z_{2,i}}\,,\quad \inthmnutwo = \frac{\abs{a_1} \cstBoper}{\sqrt{\gamma}}\,.
		\end{equation}
		By incorporating \eqref{eq:main_thm_convrg_abs_diff_gamma}, \eqref{eq:main_thm_convrg_overl_gamma_est}, \eqref{eq:main_thm_convrg_abs_eta_one}, and \eqref{eq:main_thm_convrg_abs_eta_two} into expression \eqref{eq:main_thm_convrg_tild_delta_one}, we derive the following conclusion
		\begin{align*}
			\overline{\delta}_{1,k + 1}^{2} \leq \overline{\delta}_{1,1}^{2} &+ \inthmsumovrldlt \tau \sum_{i = 1}^{k - 1} \overline{\delta}_{1,i}^{2} \\
			&+ \tau \sum_{i = 1}^{k} \brac{\inthmtilgfin \overline{\delta}_{1,i} +\inthmnutwo \overline{\gamma}_{2,i} + \inthmsumonetwo \tau^2} \brac{\overline{\delta}_{1,i + 1} + \overline{\delta}_{1,i}}\,,\quad \inthmsumovrldlt = \frac{4 \cstmthree \cstmfive \beta}{\alpha}\,.
		\end{align*}
		By implementing the approach for the previously established inequality, as indicated in \citepbf{Theorem 2.1}{RogavaTsiklauri2014}, we arrive at the following result:
		\begin{align}\label{eq:main_thm_convrg_overl_deltone_final}
			\overline{\delta}_{1,k + 1} &\leq \overline{\delta}_{1,1} + \inthmsumovrldlt \tau \sum_{i = 1}^{k} \overline{\delta}_{1,i} + 2 \tau \sum_{i = 1}^{k} \brac{\inthmtilgfin \overline{\delta}_{1,i} +\inthmnutwo \overline{\gamma}_{2,i} + \inthmsumonetwo \tau^2} \nonumber \\
			&= \overline{\delta}_{1,1} + \brac{2 \inthmtilgfin + \inthmsumovrldlt} \tau \sum_{i = 1}^{k} \overline{\delta}_{1,i} + 2 \inthmnutwo \tau \sum_{i = 1}^{k} \overline{\gamma}_{2,i} + 2 \inthmsumonetwo t_k \tau^2\,.
		\end{align}
		
		Let us now move forward to derive an estimate of the type \eqref{eq:main_thm_convrg_overl_deltone_final} for the equation \eqref{eq:main_thm_convrg_error_eqt_z_two}. To do this, we adopt a similar approach as was used for the previous equation \eqref{eq:main_thm_convrg_error_eqt_z_one}. Specifically, we apply the inner product to both sides of equality \eqref{eq:main_thm_convrg_error_eqt_z_two} with $z_{2,k + 1} - z_{2,k - 1} = \Delta z_{2,k} + \Delta z_{2,k - 1}$, employing the properties of the operator $L = \gamma A + \delta C$, which is self-adjoint and positive definite. Furthermore, reasoning analogous to that used to establish the bound \eqref{eq:lemma1_final_ineq_dlt2} in \lemref{prop:lemma1} leads us to the following relation:
		\begin{equation}\label{eq:main_thm_convrg_overl_delt_z_two}
			\overline{\delta}_{2,k + 1} \leq \overline{\delta}_{2,1} + 2 \tau \sum_{i = 1}^{k} \norm{R_i\brac{\tau}} + 2 \abs{a_2} \tau \sum_{i = 1}^{k} \norm{B z_{1,i}}\,,
		\end{equation}
		for which
		\begin{equation*}
			\overline{\delta}_{2,k} = \sqrt{\overline{\lambda}_{2,k} + \frac{1}{2} \overline{\gamma}_{2,k - 1}^{2}}\,,\quad \overline{\lambda}_{2,k} = \overline{\alpha}_{2,k}^{2} + \frac{1}{2} \overline{\gamma}_{2,k}^{2}\,,\quad \overline{\alpha}_{2,k} = \norm{\frac{\Delta z_{2,k - 1}}{\tau}}\,,\quad \overline{\gamma}_{2,k} = \norm{L^{\nicefrac{1}{2}} z_{2,k}}\,.
		\end{equation*}
		
		It can be observed that by employing the estimates \eqref{eq:main_thm_convrg_est_rems} and \eqref{eq:main_thm_convrg_overl_gamma_est} alongside \remref{prop:remark6}, from inequality \eqref{eq:main_thm_convrg_overl_delt_z_two} we can immediately deduce the following result:
		\begin{equation}\label{eq:main_thm_convrg_overl_delttwo_final}
			\overline{\delta}_{2,k + 1} \leq \overline{\delta}_{2,1} + \inthmfrsumoverldlt \tau \sum_{i = 1}^{k} \overline{\delta}_{1,i}  + 2 \inthmsumthreefour t_k \tau^2\,,\quad \inthmfrsumoverldlt = \frac{2 \sqrt{2} \abs{a_2} \cstBoper}{\sqrt{\alpha}}\,.
		\end{equation}
		
		It is evident that a similar inequality holds for $\overline{\gamma}_{2,i}$, as given in \eqref{eq:main_thm_convrg_overl_gamma_est}, specifically $\overline{\gamma}_{2,i}^{2} \leq 2 \overline{\delta}_{2,i}^{2}$. Taking this inequality into consideration in the bound \eqref{eq:main_thm_convrg_overl_deltone_final}, we arrive at the following result:
		\begin{equation}\label{eq:main_thm_convrg_overl_deltone_latest}
			\overline{\delta}_{1,k + 1} \leq \overline{\delta}_{1,1} + \brac{2 \inthmtilgfin + \inthmsumovrldlt} \tau \sum_{i = 1}^{k} \overline{\delta}_{1,i} + 2 \sqrt{2} \inthmnutwo \tau \sum_{i = 1}^{k} \overline{\delta}_{2,i} + 2 \inthmsumonetwo t_k \tau^2\,.
		\end{equation}
		By summing inequalities \eqref{eq:main_thm_convrg_overl_delttwo_final} and \eqref{eq:main_thm_convrg_overl_deltone_latest} and introducing the notation $\overline{\delta}_{k} = \overline{\delta}_{1,k} + \overline{\delta}_{2,k}$, we can infer that:
		\begin{equation*}
			\overline{\delta}_{k + 1} \leq \overline{\delta}_{1} + \inthmsumovrlndlt \tau \sum_{i = 1}^{k} \overline{\delta}_{i} + \inthmassoctau t_k \tau^2\,,
		\end{equation*}
		in which
		\begin{equation*}
			\inthmsumovrlndlt = \max\brac{2 \inthmtilgfin + \inthmsumovrldlt + \inthmfrsumoverldlt, 2 \sqrt{2} \inthmnutwo}\, \text{and}\,\, \inthmassoctau = 2\brac{\inthmsumonetwo + \inthmsumthreefour}\,.
		\end{equation*}
		Thus, by employing the discrete Gr\"{o}nwall-type inequality ({\cf} \citepbf{Lemma 3.1}{RogavaVashakidze2024}), it can be concluded that:
		\begin{equation}\label{eq:main_thm_convrg_reslt_gronwl}
			\overline{\delta}_{k + 1} \leq {e}^{\inthmsumovrlndlt t_k} \brac{\overline{\delta}_{1} + \inthmassoctau t_k \tau^2} \leq {e}^{\inthmsumovrlndlt T} \brac{\overline{\delta}_{1} + \inthmassoctau T \tau^2}\,.
		\end{equation}
		
		In order to derive the desired estimates, it becomes essential to take an additional step. Specifically, this entails estimating the term $\overline{\delta}_{1}$ in inequality \eqref{eq:main_thm_convrg_reslt_gronwl}. On the other hand, the term $\overline{\delta}_{1}$ is expressed as the sum of $\overline{\delta}_{1,1}$ and $\overline{\delta}_{2,1}$. The first summand involves the quantities $\overline{\alpha}_{1,1}$, $\overline{\gamma}_{1,0}$, and $\overline{\gamma}_{1,1}$, whereas the second summand contains the quantities $\overline{\alpha}_{2,1}$, $\overline{\gamma}_{2,0}$, and $\overline{\gamma}_{2,1}$.
		
		In view of assumptions \ref{itm_theorem_a}–\ref{itm_theorem_c} of \thmref{prop:main_thm_convrg}, the following estimates hold:
		\begin{equation}\label{eq:main_thm_convrg_overln_alphbet}
			\overline{\alpha}_{1,1} \leq \inthmalphoneone \tau^2\,,\quad \overline{\gamma}_{1,1} \leq \inthmagammoneone \tau^2\,,\quad \overline{\alpha}_{2,1} \leq \inthmalphtwoone \tau^2\,,\quad \overline{\gamma}_{2,1} \leq \inthmagammtwoone \tau^2\,.
		\end{equation}
		
		By employing these estimates given in \eqref{eq:main_thm_convrg_overln_alphbet}, the following bound is established:
		\begin{equation}\label{eq:main_thm_convrg_overl_delta_one}
			\overline{\delta}_{1} = \overline{\delta}_{1,1} + \overline{\delta}_{2,1} \leq \inthmoverldltone \tau^2\,.
		\end{equation}
		
		From \eqref{eq:main_thm_convrg_reslt_gronwl}, and considering the bound provided in \eqref{eq:main_thm_convrg_overl_delta_one}, the estimates for \thmref{prop:main_thm_convrg} are derived.
	\end{proof}
	
	\begin{corollary}\label{prop:corol_follow_main_thm_convrg}
		The errors $z_{1,k} = u\brac{t_k} - u_k$ and $z_{2,k} = v\brac{t_k} - v_k$ associated with the approximate solutions to the problem defined by equations \eqref{eq:abst_timoshenko1}-\eqref{eq:abst_timoshenko3} are constrained by the following inequality:
		\begin{equation}\label{eq:corol_follow_main_thm_convrg_bound}
			\max_{1 \leq k \leq m} \norm{z_{\detectstaticarg,k}} \leq \cstremfollthm \tau^2\,,\quad \detectstaticarg = 1,2\,,
		\end{equation}
		where $m = \qbrac{\dfrac{\overline{T}}{\tau}}$ and $\overline{T}$ satisfies $0 < \overline{T} \leq T$.
	\end{corollary}
	\begin{corollary}\label{prop:corol_error_deriv}
		For the approximate solution of system \eqref{eq:abst_timoshenko1}-\eqref{eq:abst_timoshenko3}, the following estimate holds for the error corresponding to the finite-difference approximation (more precisely, the second-order central finite difference) of the first-order derivative:
		\begin{equation}\label{eq:corol_error_deriv_ineq}
			\max_{1 \leq k \leq m} \norm{u^{\prime}\brac{t_k} - \frac{u_{k + 1} - u_{k - 1}}{2 \tau}} \leq \cstcorolderivatone \tau^2\,,\quad \max_{1 \leq k \leq m} \norm{v^{\prime}\brac{t_k} - \frac{v_{k + 1} - v_{k - 1}}{2 \tau}} \leq \cstcorolderivattwo \tau^2\,,
		\end{equation}
		where $m = \qbrac{\dfrac{\overline{T}}{\tau}}$ and $\overline{T}$ lies within the interval $0 < \overline{T} \leq T$.
	\end{corollary}
	\begin{remark}\label{prop:rem_main_thm_holder}
		Suppose that, in \thmref{prop:main_thm_convrg}, the functions $u^{\prime\prime\prime}\brac{t}$, $v^{\prime\prime\prime}\brac{t}$, $A u^{\prime}\brac{t}$, and $A v^{\prime}\brac{t}$ satisfy a H\"{o}lder condition with exponent $\lambda$ $\brac{0 < \lambda \leq 1}$. Then, the following estimates are valid:
		\begin{equation*}
			\max_{1 \leq k \leq m} \norm{A^{\nicefrac{1}{2}} z_{\detectstaticarg,k}} \leq \cstremholderone \tau^{1 + \lambda}\,,\quad \max_{1 \leq k \leq m} \norm{\frac{\Delta z_{\detectstaticarg,k - 1}}{\tau}} \leq \cstremholdertwo \tau^{1 + \lambda}\,,\quad \detectstaticarg = 1,2\,.
		\end{equation*}
	\end{remark}
	\begin{remark}
		The developed approach in this work allows us to extend the results obtained for problems \eqref{eq:abst_timoshenko1}-\eqref{eq:abst_timoshenko3} to the following modified system:
		\begin{subequations}
			\begin{gather*}
				\frac{\d^2 u\brac{t}}{\d t^2} + \brac{\alpha + \beta {\norm{A_1^{\nicefrac{1}{2}} u}}^2} A_1 u\brac{t} + a_1 B_2 v\brac{t} = f_1\brac{t}\,, \\
				\frac{\d^2 v\brac{t}}{\d t^2} + \gamma A_2 v\brac{t} + \delta C v\brac{t} + a_2 B_1 u\brac{t} = f_2\brac{t}\,.
			\end{gather*}
		\end{subequations}
		Here, $A_1$ and $A_2$ are self-adjoint, positive definite operators, while $B_1$ and $B_2$ are closed linear operators satisfying the subordination conditions $D\brac{A_i} \subset D\brac{B_{3 - i}}$, $i = 1,2$, and
		\begin{equation*}
			{\norm{B_{3 - i} \varphi}}^2 \leq b_i^2 \brac{A_i \varphi,\varphi}\,,
		\end{equation*}
		for all $\varphi \in D\brac{A_i}$ and for some constants $b_i > 0$. Furthermore, it is assumed that the intersection of the domains $D\brac{A_1} \cap D\brac{A_2}$ is dense in the Hilbert space $\H$.
	\end{remark}
	
	\section{Spatial One-Dimensional Nonlinear Dynamic Timoshenko Model}\label{sec:timoshenko_specific}
	Observe that the system \eqref{eq:abst_timoshenko1}-\eqref{eq:abst_timoshenko2} is an abstract analogue of the spatial one-dimensional nonlinear dynamic Timoshenko model, which takes the following form:
	\begin{gather}
		\frac{\partial^2 u\brac{x,t}}{\partial t^2} - \brac{\alpha + \beta \int\limits_{0}^{\ell} \qbrac{\frac{\partial u\brac{x,t}}{\partial x}}^2 \d x} \frac{\partial^2 u\brac{x,t}}{\partial x^2} + a_1 \frac{\partial v\brac{x,t}}{\partial x} = f_1\brac{x,t}\,,\label{eq:timosh_spec_nonlinear} \\
		\frac{\partial^2 v\brac{x,t}}{\partial t^2} - \gamma \frac{\partial^2 v\brac{x,t}}{\partial x^2} + \delta v\brac{x,t} - a_2 \frac{\partial u\brac{x,t}}{\partial x} = f_2\brac{x,t}\,.\label{eq:timosh_spec_linear}
	\end{gather}
	In this setting, $\brac{x,t} \in \brac{0,\ell} \times \left( 0,T \right]$; the constants $\alpha$, $\beta$, $\gamma$, $\delta$, $a_1$, and $a_2$ are positive; the functions $f_1\brac{x,t}$ and $f_2\brac{x,t}$ are continuous over the prescribed domain.
	
	For the system \eqref{eq:timosh_spec_nonlinear}-\eqref{eq:timosh_spec_linear}, we pose the following initial-boundary value problem:
	\begin{gather}
		u\brac{x,0} = \varphi_0\brac{x}\,,\quad u_{t}^{\prime}\brac{x,0} = \varphi_1\brac{x}\,,\quad v\brac{x,0} = \psi_0\brac{x}\,,\quad v_{t}^{\prime}\brac{x,0} = \psi_1\brac{x}\,,\label{eq:timosh_spec_init_conds} \\
		u\brac{0,t} = 0\,,\quad u\brac{\ell,t} = 0\,,\quad v\brac{0,t} = 0\,,\quad v\brac{\ell,t} = 0\,.\label{eq:timosh_spec_bound_conds}
	\end{gather}
	Here, the functions $\varphi_0\brac{x}$, $\varphi_1\brac{x}$, $\psi_0\brac{x}$, and $\psi_1\brac{x}$ are continuous. Moreover, compatibility conditions are assumed to hold for $\varphi_0\brac{x}$ and $\psi_0\brac{x}$, specifically that $\varphi_0\brac{0} = \varphi_0\brac{\ell} = 0$ and $\psi_0\brac{0} = \psi_0\brac{\ell} = 0$; the functions $u\brac{x,t}$ and $v\brac{x,t}$ are unknown.
	
	\subsection{Semi-Discrete Scheme}\label{subsec:semidiscrete}
	We shall establish a uniform grid for the time domain $\qbrac{0,T}$ with a step size of $\tau$, defined as follows:
	\begin{equation*}
		0 < t_0 < t_1 < \cdots < t_{n - 1} < t_n = T\,,\quad t_k = k \tau\,,\quad k = 0,1,\ldots,n\,,\quad \tau = \frac{T}{n}\,.
	\end{equation*}
	Let us write a semi-discrete scheme for the problem \eqref{eq:timosh_spec_nonlinear}-\eqref{eq:timosh_spec_bound_conds} based on the proposed scheme \eqref{eq:semi_discrete_timosh_1}-\eqref{eq:semi_discrete_timosh_2} for the abstract coupled system. It has the following form:
	\begin{subequations}
		\begin{gather}
			\begin{aligned}\label{eq:tim_spec_semidiscr_u}
				\frac{{\Delta}^2 u_{k - 1}\brac{x}}{\tau^2} - \frac{1}{2} q_k \brac{\frac{\d^2 u_{k + 1}\brac{x}}{\d x^2} + \frac{\d^2 u_{k - 1}\brac{x}}{\d x^2}} = f_{1,k}\brac{x} - a_1 \frac{\d v_k\brac{x}}{\d x}\,,
			\end{aligned}
		\end{gather}
		\begin{align}\label{eq:tim_spec_semidiscr_v}
			\frac{{\Delta}^2 v_{k - 1}\brac{x}}{\tau^2} - \frac{1}{2} \gamma \brac{\frac{\d^2 v_{k + 1}\brac{x}}{\d x^2} + \frac{\d^2 v_{k - 1}\brac{x}}{\d x^2}} + \frac{1}{2} \delta \brac{v_{k + 1}\brac{x} + v_{k - 1}\brac{x}}\nonumber \\
			= f_{2,k}\brac{x} + a_2 \frac{\d u_k\brac{x}}{\d x}\,,
		\end{align}
	\end{subequations}
	where $k = 1,2,\ldots,n - 1$, $f_1\brac{x,t_k} = f_{1,k}\brac{x}$, and $f_2\brac{x,t_k} = f_{2,k}\brac{x}$.
	
	The nonlinear term in equation \eqref{eq:tim_spec_semidiscr_u} is evaluated at the middle node point and denoted by $q_k$, specifically:
	\begin{equation}\label{eq:tim_spec_q_k_nonlin_term}
		q_k = \alpha + \beta \int\limits_{0}^{\ell} \brac{\frac{\d u_k \brac{x}}{\d x}}^2 \d x\,.
	\end{equation}
	
	For each discrete time layer, the boundary conditions prescribed in \eqref{eq:timosh_spec_bound_conds} are rewritten as follows:
	\begin{equation}\label{eq:tim_spec_bound_conds}
		u_{k + 1}\brac{0} = 0\,,\quad u_{k + 1}\brac{\ell} = 0\,,\quad v_{k + 1}\brac{0} = 0\,,\quad v_{k + 1}\brac{\ell} = 0\,.
	\end{equation}
	
	The values of the unknown functions at the zeroth and first temporal layers are determined by the initial conditions specified in \eqref{eq:timosh_spec_init_conds} and the system given by \eqref{eq:timosh_spec_nonlinear}-\eqref{eq:timosh_spec_linear} as follows:
	\begin{gather*}
		u_0\brac{x} = \varphi_0\brac{x}\,, \\
		u_1\brac{x} = \varphi_0\brac{x} + \tau \varphi_1\brac{x} + \frac{\tau^2}{2} \varphi_2\brac{x}\,,\quad \varphi_2\brac{x} = f_{1,0}\brac{x} - a_1 \frac{\d \psi_0\brac{x}}{\d x} + q_0 \frac{\d^2 \varphi_0\brac{x}}{\d x^2}\,,
	\end{gather*}
	and
	\begin{gather*}
		v_0\brac{x} = \psi_0\brac{x}\,, \\
		v_1\brac{x} = \psi_0\brac{x} + \tau \psi_1\brac{x} + \frac{\tau^2}{2} \psi_2\brac{x}\,,\quad \psi_2\brac{x} = f_{2,0}\brac{x} + a_2 \frac{\d \varphi_0\brac{x}}{\d x} + \gamma \frac{\d^2 \psi_0\brac{x}}{\d x^2} - \delta \psi_0\brac{x}\,.
	\end{gather*}
	
	In the subsequent discussion, let $u_k\brac{x}$ and $v_k\brac{x}$ denote the solutions to the system of differential-difference equations \eqref{eq:tim_spec_semidiscr_u}-\eqref{eq:tim_spec_semidiscr_v}, subject to the boundary conditions \eqref{eq:tim_spec_bound_conds}. Consequently, these solutions are declared as approximate values of the exact solution $u\brac{x,t}$ and $v\brac{x,t}$ of the problem \eqref{eq:timosh_spec_nonlinear}-\eqref{eq:timosh_spec_bound_conds} at the discrete time instants $t = t_k$, respectively, so that $u\brac{x,t_k} \approx u_k\brac{x}$ and $v\brac{x,t_k} \approx v_k\brac{x}$.
	
	Let us now define the notation for the first-order and second-order differential operators as follows:
	\begin{gather}
		{\B}_0 = \frac{\d}{\d x}\,,\quad D\brac{{\B}_0} = C^1\brac{\qbrac{0,\ell}}\,,\label{eq:tim_spec_first_order_deriv} \\
		{\A}_0 = - \frac{\d^2}{\d x^2}\,,\quad D\brac{{\A}_0} = \left\{ u\brac{x} \in C^2\brac{\qbrac{0,\ell}} \mid u\brac{0} = u\brac{\ell} = 0 \right\}\,.\label{eq:tim_spec_second_order_deriv}
	\end{gather}
	Employing the notations defined in \eqref{eq:tim_spec_first_order_deriv} and \eqref{eq:tim_spec_second_order_deriv}, the system \eqref{eq:tim_spec_semidiscr_u}-\eqref{eq:tim_spec_semidiscr_v} can be reformulated as follows:
	\begin{subequations}
		\begin{gather}
			\begin{aligned}\label{eq:tim_spec_semidiscr_oper_u}
				\frac{{\Delta}^2 u_{k - 1}\brac{x}}{\tau^2} + q_k \frac{{\A}_0 u_{k + 1} + {\A}_0 u_{k - 1}}{2} = f_{1,k}\brac{x} - a_1 {\B}_0 v_k\brac{x}\,,
			\end{aligned}
		\end{gather}
		\begin{gather}
			\begin{aligned}\label{eq:tim_spec_semidiscr_oper_v}
				\frac{{\Delta}^2 v_{k - 1}\brac{x}}{\tau^2} + \frac{{\L}_0 v_{k + 1} + {\L}_0 v_{k - 1}}{2} = f_{2,k}\brac{x} + a_2 {\B}_0 u_k\brac{x}\,,
			\end{aligned}
		\end{gather}
	\end{subequations}
	for $k = 1,2,\ldots,n - 1$, and ${\L}_0 = \gamma {\A}_0 + \delta {\I}$, where ${\I}$ represents the identity operator.
	
	It follows easily from the application of integration by parts that the quantity $q_k$ can be represented in the subsequent form, using the notation introduced in \eqref{eq:tim_spec_second_order_deriv}, specifically:
	\begin{equation*}
		q_k = \alpha + \beta \brac{{\A}_0 u_k,u_k}\,.
	\end{equation*}
	In the following discussions, the notation $\brac{\cdot,\cdot}$ is used to denote the inner product in the space $L^2\brac{0,\ell}$, whereas the corresponding norm is represented by $\norm{\cdot}$.
	
	Since ${\A}_{0}$ is a symmetric and positive definite operator (see, {\eg}, \textbf{Chap. 8} in Rektorys \citep{Rektorys1980}), it admits an extension to a self-adjoint and positive definite operator ({\cf} \textbf{Chap. V}, \textbf{Sec. 4} in Maurin \citep{Moren1965}). We denote this operator by $\A$ (${\A}_{0} \subset {\A}$). It is evident that the operator ${\B}_{0}$ is closable and can be extended to a closed operator. The closure of ${\B}_{0}$ is denoted by $\B$ (${\B}_{0} \subset {\B}$). We shall rewrite equations \eqref{eq:tim_spec_semidiscr_oper_u} and \eqref{eq:tim_spec_semidiscr_oper_v} in terms of the operators $\A$ and $\B$ as follows:
	\begin{subequations}
		\begin{gather}
			\begin{aligned}\label{eq:tim_spec_semidiscr_extend_oper_u}
				\frac{{\Delta}^2 u_{k - 1}\brac{x}}{\tau^2} + q_k \frac{{\A} u_{k + 1} + {\A} u_{k - 1}}{2} = f_{1,k}\brac{x} - a_1 {\B} v_k\brac{x}\,,
			\end{aligned}
		\end{gather}
		\begin{gather}
			\begin{aligned}\label{eq:tim_spec_semidiscr_extend_oper_v}
				\frac{{\Delta}^2 v_{k - 1}\brac{x}}{\tau^2} + \frac{{\L} v_{k + 1} + {\L} v_{k - 1}}{2} = f_{2,k}\brac{x} + a_2 {\B} u_k\brac{x}\,,\quad {\L} = \gamma {\A} + \delta {\I}\,,
			\end{aligned}
		\end{gather}
	\end{subequations}
	where
	\begin{equation*}
		q_k = \alpha + \beta {\norm{{\A}^{\nicefrac{1}{2}} u_k}}^2\,.
	\end{equation*}
	
	Our objective is to extend a theorem analogous to \thmref{prop:main_thm_convrg} to the scheme \eqref{eq:tim_spec_semidiscr_extend_oper_u}-\eqref{eq:tim_spec_semidiscr_extend_oper_v}. To achieve this, it requires demonstrating that inequalities analogous to \eqref{eq:B_cond} and \eqref{eq:remark4_inner_prod_ABphi} are retained for the operators $\A$ and $\B$.
	
	\begin{remark}\label{prop:rmk_spec_semidiscrete}
		The following relation is satisfied:
		\begin{equation}\label{eq:rmk_spec_semidiscrete_main_bound}
			{\norm{{\B} u}}^2 = \brac{{\A} u,u}\,,\quad \forall u \in D\brac{{\A}} \subset D\brac{{\B}}\,.
		\end{equation}
	\end{remark}
	\begin{proof}
		It is easy to obtain
		\begin{equation*}
			{\norm{{\B}_{0} u}}^2 = \int\limits_{0}^{\ell} {\qbrac{u^{\prime}\brac{x}}}^2 \d x = - \int\limits_{0}^{\ell} u^{\prime\prime}\brac{x} u\brac{x} \d x = \brac{{\A}_{0} u,u}\,,\quad \forall u \in D\brac{{\A}_{0}} \subset D\brac{{\B}_{0}}\,.
		\end{equation*}
		Hence, it follows that
		\begin{equation}\label{eq:rmk_spec_semidiscrete_square_norm}
			{\norm{{\B} u}}^2 = \brac{{\A} u,u}\,,\quad \forall u \in D\brac{{\A}_{0}} \subset D\brac{{\B}_{0}}\,.
		\end{equation}
		Define ${\A} u = v$. In this notation, \eqref{eq:rmk_spec_semidiscrete_square_norm} takes the following form:
		\begin{equation}\label{eq:rmk_spec_semidiscrete_inv_oper}
			{\norm{{\B} {\A}^{-1} v}}^2 = \brac{v,{\A}^{-1} v}\,,\quad \forall v \in {R}_{0} = {\A} D\brac{{\A}_{0}}\,.
		\end{equation}
		
		It is evident that $R_0$ is dense in $L^2$. Indeed, $R_0$ represents the set of functions $f$ for which the differential equation $-u^{\prime \prime}\brac{x} = f\brac{x}$, subject to homogeneous boundary conditions, admits a unique classical solution. It is well established that, for such a simple differential equation, the solution $u$ lies in $D\brac{{\A}_0}$ whenever $f \in C\brac{\qbrac{0,\ell}}$ ({\cf} \textbf{Chap. II}, \textbf{Sec. 10} in book by Fu\v{c}\'{\i}k and Kufner \citep{FuchikKufner1980}). Moreover, the density of $C\brac{\qbrac{0,\ell}}$ in $L^2\brac{0,\ell}$ is well-known.
		
		The relation \eqref{eq:rmk_spec_semidiscrete_inv_oper} may now be extended over the entire space $L^2$. Given that $R_0$ is dense in $L^2$, for each $v \in L^2$, there exists a sequence $v_n \in R_0$ such that $v_n \to v$. Accordingly, by \eqref{eq:rmk_spec_semidiscrete_inv_oper}, for the sequence $v_n$ we have:
		\begin{equation}\label{eq:rmk_spec_semidiscrete_inv_oper_vn}
			{\norm{{\B} {\A}^{-1} v_n}}^2 = \brac{v_n,{\A}^{-1} v_n}\,.
		\end{equation}
		Applying the Cauchy–Schwarz inequality to \eqref{eq:rmk_spec_semidiscrete_inv_oper_vn}, we obtain:
		\begin{equation*}
			{\norm{{\B} {\A}^{-1} v_n}}^2 \leq \norm{v_n} \norm{{\A}^{-1} v_n} \leq \norm{{\A}^{-1}} {\norm{v_n}}^2\,.
		\end{equation*}
		From this, it follows that ${\B} {\A}^{-1} v_n$ forms the Cauchy sequence. Taking into account that ${\B}$ is a closed operator and that ${\A}^{-1} v_n \to {\A}^{-1} v$, we conclude that ${\A}^{-1} v \in D\brac{\B}$ and that ${\B} {\A}^{-1} v_n \to {\B} {\A}^{-1} v$. Considering this fact, it follows from \eqref{eq:rmk_spec_semidiscrete_inv_oper_vn} that:
		\begin{equation*}
			{\norm{{\B} {\A}^{-1} v}}^2 = \brac{v,{\A}^{-1} v}\,,\quad \forall v \in L^2\,,
		\end{equation*}
		or, which is the same
		\begin{equation*}
			{\norm{{\B} u}}^2 = \brac{{\A} u,u}\,,\quad \forall u \in D\brac{{\A}}\,.
		\end{equation*}
	\end{proof}
	\begin{remark}\label{prop:rmk_operat_ab}
		Suppose that the set $D_0 \subset D\brac{{\A}_{0}}$ is such that the operator ${\B}_{0}$ maps $D_0$ into $D\brac{{\A}_{0}}$, that is, ${\B}_{0} : D_0 \to D\brac{{\A}_{0}}$. Moreover, the sets $D_0$ and $B_0 D_0$ are dense in $L^2$. Then, the following relation holds:
		\begin{equation}\label{eq:rmk_operat_ab_main}
			\brac{\A \B u, \B u} = {\norm{\A u}}^2\,,\quad \forall u \in D_0 \subset D\brac{\A}\,.
		\end{equation}
	\end{remark}
	Indeed, by an application of integration by parts, we obtain:
	\begin{equation*}
		\brac{{\A}_0 {\B}_0 u, {\B}_0 u} = -\int_{0}^{\ell} u^{\prime \prime \prime} \brac{x} u^{\prime} \brac{x} \d x = \int_{0}^{\ell} {\qbrac{u^{\prime \prime} \brac{x}}}^2 \d x = {\norm{{\A}_0 u}}^2\,,\quad \forall u \in D_0\,.
	\end{equation*}
	Thus, it is evident that \eqref{eq:rmk_operat_ab_main} follows since the operators $\A$ and $\B$ are extensions of ${\A}_0$ and ${\B}_0$, respectively.
	
	Due to \remref{prop:remark4}, the following relation follows from \remref{prop:rmk_operat_ab}:
	\begin{equation*}
		\norm{{\A}^{\nicefrac{1}{2}} {\B} u} = \norm{{\A} u}\,,\quad \forall u \in D\brac{{\A}}\,.
	\end{equation*}
	
	Based on \thmref{prop:main_thm_convrg}, we now formulate a theorem concerning the convergence of the scheme \eqref{eq:tim_spec_semidiscr_extend_oper_u}-\eqref{eq:tim_spec_semidiscr_extend_oper_v}.
	
	\begin{theorem}\label{prop:main_thm_specific_timosh}
		Let the initial-boundary value problem \eqref{eq:timosh_spec_nonlinear}-\eqref{eq:timosh_spec_bound_conds} be well-posed. Furthermore, the following conditions are satisfied:
		\begin{enumerate}[label=(\alph*)]
			\item\label{itm_main_thm_specific_timosh_a} Let $\varphi_0\brac{x}, \psi_0\brac{x} \in D\brac{{\A}_{0}}$; $\varphi_1\brac{x}, \psi_1\brac{x}, \varphi_2\brac{x}$, and $\psi_2\brac{x} \in C^{1} \brac{\qbrac{0,\ell}}$. The functions $u_x \brac{x,t}$ and $v_x \brac{x,t}$ possess a second-order continuous derivative with respect to the temporal variable. Furthermore, the right-hand sides of equations \eqref{eq:timosh_spec_nonlinear} and \eqref{eq:timosh_spec_linear}, namely $f_1\brac{x,t}$ and $f_2\brac{x,t}$, are continuous functions that vanish at the endpoints of the interval $\qbrac{0,\ell}$. In addition, $f_1\brac{x,t}$ and $f_2\brac{x,t}$ are continuously differentiable with respect to the spatial variable.
			\item\label{itm_main_thm_specific_timosh_b} The solutions $u\brac{x,t}$ and $v\brac{x,t}$ to the initial-boundary value problem \eqref{eq:timosh_spec_nonlinear}-\eqref{eq:timosh_spec_bound_conds} are continuously differentiable functions up to and including the third order with respect to the temporal variable. Furthermore, the third derivatives $u^{\prime \prime \prime}\brac{x,t}$ and $v^{\prime \prime \prime}\brac{x,t}$ are Lipschitz continuous with respect to the temporal variable.
			\item\label{itm_main_thm_specific_timosh_c} The functions $u_{xx}\brac{x,t}$ and $v_{xx}\brac{x,t}$ are continuously differentiable with respect to the temporal variable. Moreover, the mixed partial derivatives $u_{xxt}\brac{x,t}$ and $v_{xxt}\brac{x,t}$ satisfy the Lipschitz condition with respect to the temporal variable.
		\end{enumerate}
		Then, there exists $\overline{T}$ such that $0 < \overline{T} \leq T$ and the following estimates hold for the errors of the approximate solutions $z_{1,k}\brac{x} = u\brac{x,t_k} - u_k\brac{x}$ and $z_{2,k}\brac{x} = v\brac{x,t_k} - v_k\brac{x}$:
		\begin{equation*}
			\max_{1 \leq k \leq m} \norm{\frac{\d z_{\detectstaticarg,k}}{\d x}} \leq \cstthmconvspecopone \tau^2\,,\quad \max_{1 \leq k \leq m} \norm{\frac{\Delta z_{\detectstaticarg,k - 1}}{\tau}} \leq \cstthmconvspecoptwo \tau^2\,,\quad \detectstaticarg = 1,2\,,
		\end{equation*}
		where $m = \qbrac{\dfrac{\overline{T}}{\tau}}$, and $\Delta z_{\detectstaticarg,k}\brac{x} = z_{\detectstaticarg,k + 1}\brac{x} - z_{\detectstaticarg,k}\brac{x}$.
	\end{theorem}
	
	\subsection{Legendre–Galerkin Spectral Method}\label{subsec:galerkin}
	
	It should be observed that solving the system \eqref{eq:tim_spec_semidiscr_u}-\eqref{eq:tim_spec_semidiscr_v} reduces to solving the following intermediate system:
	\begin{subequations}
		\begin{align}
			w_{1,k} \brac{x} - \frac{\tau^2}{2} q_k \frac{\d^2 w_{1,k}}{\d x^2} &= \tau^2 f_{1,k}\brac{x} + 2 u_k\brac{x} - a_1 \tau^2 \frac{\d v_k}{\d x}\,,\label{eq:tim_spec_expanded_eq_u} \\
			\brac{1 + \frac{\tau^2}{2} \delta} w_{2,k} \brac{x} - \frac{\tau^2}{2} \gamma \frac{\d^2 w_{2,k}}{\d x^2} &= \tau^2 f_{2,k}\brac{x} + 2 v_k\brac{x} + a_2 \tau^2 \frac{\d u_k}{\d x}\,,\label{eq:tim_spec_expanded_eq_v} \\
			w_{1,k} \brac{0} = w_{1,k} \brac{\ell} &= 0\,,\quad w_{2,k} \brac{0} = w_{2,k} \brac{\ell} = 0\,,\label{eq:tim_spec_expanded_bound_cond}
		\end{align}
	\end{subequations}
	for which $k = 1,2,\ldots,n - 1$.
	
	By solving the system \eqref{eq:tim_spec_expanded_eq_u}-\eqref{eq:tim_spec_expanded_eq_v}, the unknowns $u_k\brac{x}$ and $v_k\brac{x}$ are determined by the following formulas:
	\begin{equation}\label{eq:tim_spec_inv_operat_ref_uv}
		u_{k + 1}\brac{x} = w_{1,k} \brac{x} - u_{k - 1}\brac{x}\,,\quad v_{k + 1}\brac{x} = w_{2,k} \brac{x} - v_{k - 1}\brac{x}\,.
	\end{equation}
	
	Before presenting the method in detail, we shall first introduce the shifted Legendre polynomials. These polynomials are obtained by an affine transformation that involves ``shifting'' and ``scaling'' the argument of the standard Legendre polynomials, mapping the interval $\qbrac{0,\ell}$ onto $\qbrac{-1,1}$. Specifically, the shifted Legendre polynomial of degree $m$ is defined as
	\begin{equation*}
		\widetilde{P}_m \brac{x} = P_m \brac{\frac{2}{\ell} x - 1}\,,\quad x \in \qbrac{0,\ell}\,,
	\end{equation*}
	where $P_m \brac{x}$ represents the $m$-th standard Legendre polynomial. In the following discussion, the shifted Legendre polynomial of degree $m$ is denoted by $\widetilde{P}_m \brac{x}$.
	
	Recall that the shifted Legendre polynomials retain an orthogonality property, similar to their unshifted counterparts, but adjusted to the interval $\qbrac{0,\ell}$, yielding a dependence on the interval length $\ell$, namely:
	\begin{equation}\label{eq:galerkin_orthog_prop}
		\brac{\widetilde{P}_i,\widetilde{P}_m} = \ell A_i A_m \delta_{im}\,,\quad A_m = \frac{1}{\sqrt{2m + 1}}\,,
	\end{equation}
	where $\delta_{im}$ denotes the Kronecker delta, taking the value $1$ when $i = m$ and $0$ otherwise.
	
	In order to approximate the solutions $w_{1,k} \brac{x}$ and $w_{2,k} \brac{x}$ of the system \eqref{eq:tim_spec_expanded_eq_u}-\eqref{eq:tim_spec_expanded_bound_cond} at each temporal layer, we seek them in the form of the following linear combinations:
	\begin{equation}\label{eq:galerkin_intermed_syst_ansatzes}
		{w}_{k,N}^{\brac{j}} \brac{x} = \sum_{m = 1}^{N} {w}_{j,m}^{k} \phi_m \brac{x}\,,\quad j = 1,2\,.
	\end{equation}
	Here, we choose the differences of the shifted Legendre polynomials as \ansatz{} (trial) functions:
	\begin{equation*}
		\phi_m \brac{x} = \frac{1}{A_m \sqrt{\ell}} \int\limits_{0}^{x} \widetilde{P}_m \brac{s} \d s = \frac{\sqrt{\ell}}{2} A_m \brac{\widetilde{P}_{m + 1} \brac{x} - \widetilde{P}_{m - 1} \brac{x}}\,.
	\end{equation*}
	
	Observe that the functions $\phi_m \brac{x}$ inherently vanish at the endpoints of the interval $\qbrac{0,\ell}$. Furthermore, it is obvious that:
	\begin{equation*}
		\phi_{m}^{\prime} \brac{x} = \widehat{P}_m \brac{x}\,,\quad \widehat{P}_m \brac{x} = \frac{1}{A_m \sqrt{\ell}} \widetilde{P}_m \brac{x}\,.
	\end{equation*}
	Evidently, $\widehat{P}_m \brac{x}$ denotes the orthonormal shifted Legendre polynomial on the interval $\qbrac{0,\ell}$.
	
	To approximate the solutions $u_k\brac{x}$ and $v_k\brac{x}$ of the system \eqref{eq:tim_spec_semidiscr_u}-\eqref{eq:tim_spec_semidiscr_v} at each temporal layer, we employ the same linear combination as in \eqref{eq:galerkin_intermed_syst_ansatzes}, that is,
	\begin{equation}\label{eq:galerkin_ansatzes}
		\widetilde{u}_{k,N} \brac{x} = \sum_{m = 1}^{N} \tilde{u}_{m}^{k} \phi_m \brac{x}\,,\quad \widetilde{v}_{k,N} \brac{x} = \sum_{m = 1}^{N} \tilde{v}_{m}^{k} \phi_m \brac{x}\,,\quad k = 2,3,\ldots,n\,.
	\end{equation}
	
	Substituting the Galerkin approximations defined by \eqref{eq:galerkin_intermed_syst_ansatzes} and \eqref{eq:galerkin_ansatzes} into the relations provided in \eqref{eq:tim_spec_inv_operat_ref_uv}, we derive the following expressions for the expansion coefficients:
	\begin{equation*}
		\tilde{u}_{m}^{k + 1} = {w}_{1,m}^{k} - \tilde{u}_{m}^{k - 1} \quad \text{and} \quad \tilde{v}_{m}^{k + 1} = {w}_{2,m}^{k} - \tilde{v}_{m}^{k - 1}\,.
	\end{equation*}
	
	Let the test functions be chosen to coincide with the trial functions $\phi_m \brac{x}$. The derivation of the Galerkin system involves substituting the Galerkin approximations from \eqref{eq:galerkin_intermed_syst_ansatzes} into equations \eqref{eq:tim_spec_expanded_eq_u} and \eqref{eq:tim_spec_expanded_eq_v}, followed by taking the inner product of both sides of the resulting equations with the test functions $\phi_m \brac{x}$ for $m = 1,2,\ldots,N$. Finally, the resulting Galerkin subsystems can be reformulated in the following matrix-vector form:
	\begin{subequations}
		\begin{align}
			\vm{\T}_{N}^{k} \vm{w}_{1}^{k} &= \frac{4 \tau^2}{\ell^2} \vm{I}_{k}^{\brac{1}} + 2 \vm{\H}_N \vm{\tilde{u}}^k - \frac{2 a_1 \tau^2}{\ell} \vm{\B}_N \vm{\tilde{v}}^k\,,\label{eq:galerkin_system_matrx_form_u} \\
			\vm{\T}_{N} \vm{w}_{2}^{k} &= \frac{2 a_0 \tau^2}{\ell^2} \vm{I}_{k}^{\brac{2}} + a_0 \vm{\H}_N \vm{\tilde{v}}^k + \frac{a_0 a_2 \tau^2}{\ell} \vm{\B}_N \vm{\tilde{u}}^k\,,\quad a_0 = \frac{4}{2 + \delta \tau^2}\,.\label{eq:galerkin_system_matrx_form_v}
		\end{align}
	\end{subequations}
	
	The coefficient matrices corresponding to the subsystems \eqref{eq:galerkin_system_matrx_form_u} and \eqref{eq:galerkin_system_matrx_form_v} are defined by
	\begin{equation*}
		\vm{\T}_{N}^{k} = \vm{\H}_N + \frac{2 \tau^2}{\ell^2} q_{k,N} \vm{\I}_N \quad \text{and} \quad \vm{\T}_{N} = \vm{\H}_N + \frac{a_0 \tau^2}{\ell^2} \gamma \vm{\I}_N\,,
	\end{equation*}
	for which
	\begin{equation*}
		q_{k,N} = \alpha + \beta \sum_{m = 1}^{N} \brac{\tilde{u}_{m}^{k}}^2\,.
	\end{equation*}
	
	The matrix $\vm{\H}_N = \brac{h_{i,j}}_{1 \leq i,j \leq N}$ is symmetric and sparse. Its entries are defined as follows: the main diagonal entries are given by $h_{i,i} = 2 A_{i - 1}^{2} A_{i + 1}^{2}$, for $i = 1,2,\ldots,N$. The nonzero entries on the second sub- and super-diagonals are $h_{i,i + 2} = h_{i + 2,i} = - A_{i} A_{i + 1}^{2} A_{i + 2}$, for $i = 1,2,\ldots,N - 2$. All other entries of this matrix are zero. The matrix $\vm{\B}_N = \brac{b_{i,j}}_{1 \leq i,j \leq N}$ is sparse and skew-symmetric, with nonzero entries confined to the first sub- and super-diagonals. Its off-diagonal entries satisfy $b_{i,i + 1} = A_i A_{i + 1}$ and $b_{i + 1,i} = -A_i A_{i + 1}$, for $i = 1,2,\ldots,N - 1$, while all other entries, including those on the main diagonal, are equal to zero. Moreover, $\vm{\I}_N$ denotes the identity matrix of order $N$.
	
	The column vectors $\vm{\tilde{u}}^k$, $\vm{\tilde{v}}^k$, $\vm{w}_{1}^{k}$, and $\vm{w}_{2}^{k}$, appearing in the subsystems of the Galerkin linear equations \eqref{eq:galerkin_system_matrx_form_u} and \eqref{eq:galerkin_system_matrx_form_v}, consist of the expansion coefficients given in \eqref{eq:galerkin_ansatzes} and \eqref{eq:galerkin_intermed_syst_ansatzes}, respectively.
	
	For each $j = 1,2$, the vector
	\begin{equation*}
		\vm{I}_{k}^{\brac{j}} = \tps{\brac{I_{k,1}^{\brac{j}},I_{k,2}^{\brac{j}},\ldots,I_{k,N}^{\brac{j}}}}
	\end{equation*}
	is defined by $I_{k,m}^{\brac{j}} = \brac{f_{j,k},\phi_m}$. Thus, the components $I_{k,m}^{\brac{j}}$ are given by the inner products of the functions $f_{j,k}\brac{x}$ with the basis functions $\phi_m\brac{x}$.
	
	\begin{theorem}\label{prop:thm_galerkin_positive_def_syst}
		The coefficient matrices $\vm{\T}_{N}^{k}$ and $\vm{\T}_{N}$ arising from the subsystems \eqref{eq:galerkin_system_matrx_form_u} and \eqref{eq:galerkin_system_matrx_form_v} are positive definite.
	\end{theorem}
	
	This \hyperref[prop:thm_galerkin_positive_def_syst]{Theorem} follows from the subsequent \hyperref[prop:lem_galerkin_positive_def_gen]{Lemma}.
	
	\begin{lemma}\label{prop:lem_galerkin_positive_def_gen}
		Consider a general operator equation in a Hilbert space $\H$,
		\begin{equation}\label{eq:galerkin_lemma_pos_def}
			Au = f\,,\quad f \in \H\,,
		\end{equation}
		where the operator $A$ is symmetric and satisfies the condition:
		\begin{equation}\label{eq:galerkin_lemma_main_condition}
			\brac{Au,u}_{\H} \geq \alpha \brac{Bu,u}_{\H} + \nu \norm{u}_{\H}^2\,,\quad \forall u \in D\brac{A} \subset D\brac{B}\,.
		\end{equation}
		The operator $B$ is also symmetric, and $D\brac{A} \subset D\brac{B}$, with $\alpha$ and $\nu$ being positive constants.
		
		Let the basis functions $\left\{ \phi_m \right\}_{m = 1}^{\infty}$ be $B$-orthogonal, which is understood in the following sense:
		\begin{equation}\label{eq:galerkin_lemma_b_orthog_condition}
			\brac{B \phi_i,\phi_m}_{\H} = b_m \delta_{im}\,,\quad b_m > 0\,,
		\end{equation}
		where $\delta_{im}$ denotes the Kronecker delta.
		
		The coefficient matrix $\vm{\S}_N = \brac{\brac{A \phi_i,\phi_m}_{\H}}_{1 \leq i,m \leq N}$, corresponding to the system of Galerkin linear equations derived using the Galerkin spectral method for equation \eqref{eq:galerkin_lemma_pos_def}, is positive definite.
	\end{lemma}
	\begin{proof}
		Let us introduce the vector:
		\begin{equation*}
			\vm{v} = \tps{\brac{c_1,c_2,\ldots,c_N}}\,.
		\end{equation*}
		Here, $c_1,c_2,\ldots,c_N$ are the coefficients of the Galerkin approximation, corresponding to the basis functions $\phi_1,\phi_2,\ldots,\phi_N$. This vector constitutes the solution to the Galerkin system of linear equations within the finite-dimensional subspace spanned by the chosen $N$ basis functions.
		
		It follows directly that:
		\begin{equation*}
			\vm{\S}_N \vm{v} = \tps{\brac{\brac{A u_N,\phi_1}_{\H},\brac{A u_N,\phi_2}_{\H},\ldots,\brac{A u_N,\phi_N}_{\H}}}\,,
		\end{equation*}
		where
		\begin{equation}\label{eq:galerkin_lemma_gal_approx}
			u_N = \sum_{i = 1}^{N} c_i \phi_i\,.
		\end{equation}
		In fact, we have:
		\begin{equation}\label{eq:galerkin_lemma_inner_prod_oper}
			\brac{A u_N,\phi_m}_{\H} = \sum_{i = 1}^{N} c_i \brac{A \phi_i,\phi_m}_{\H}\,,\quad m = 1,2,\ldots,N\,.
		\end{equation}
		As a result of \eqref{eq:galerkin_lemma_inner_prod_oper}, we obtain:
		\begin{equation}\label{eq:galerkin_lemma_equality_inn_prod}
			\brac{\vm{\S}_N \vm{v},\vm{v}} = \sum_{i = 1}^{N} c_i \brac{A u_N,\phi_i}_{\H} = \brac{A u_N,\sum_{i = 1}^{N} c_i \phi_i}_{\H} = \brac{A u_N,u_N}_{\H}\,.
		\end{equation}
		
		By combining relations \eqref{eq:galerkin_lemma_main_condition} and \eqref{eq:galerkin_lemma_equality_inn_prod}, it follows that:
		\begin{equation}\label{eq:galerkin_lemma_finitedim_condition}
			\brac{\vm{\S}_N \vm{v},\vm{v}} \geq \alpha \brac{B u_N,u_N}_{\H} + \nu \norm{u_N}_{\H}^2\,.
		\end{equation}
		Substituting \eqref{eq:galerkin_lemma_gal_approx} into inequality \eqref{eq:galerkin_lemma_finitedim_condition} and leveraging the $B$-orthogonality property \eqref{eq:galerkin_lemma_b_orthog_condition}, we arrive at:
		\begin{align*}
			\brac{\vm{\S}_N \vm{v},\vm{v}} &\geq \alpha \brac{\sum_{i = 1}^{N} c_i B \phi_i,\sum_{m = 1}^{N} c_m \phi_m}_{\H} + \nu \norm{u_N}_{\H}^2 \\
			&\geq \alpha \sum_{i = 1}^{N} \sum_{m = 1}^{N} c_i c_m \brac{B \phi_i,\phi_m}_{\H} = \alpha \sum_{m = 1}^{N} b_m c_m^2 \geq \widehat{\alpha} \norm{\vm{v}}_{2}^2\,,\quad \widehat{\alpha} = \alpha \min_{1 \leq m \leq N} b_m\,.
		\end{align*}
		Finally, from the inequality stated above, it follows that the matrix $\vm{\S}_N$ is positive definite.
	\end{proof}
	
	\begin{remark}\label{prop:lemma_leading_oper}
		The operators appearing on the left-hand sides of each equation in the system \eqref{eq:tim_spec_expanded_eq_u}-\eqref{eq:tim_spec_expanded_eq_v}, subject to the boundary conditions \eqref{eq:tim_spec_expanded_bound_cond}, have the following form:
		\begin{equation*}
			{\T}_0 = a \I + b {\A}_0\,,\quad D\brac{{\T}_0} = D\brac{{\A}_0} = \left\{ u\brac{x} \in C^2\brac{\qbrac{0,\ell}} \mid u\brac{0} = u\brac{\ell} = 0 \right\}\,,
		\end{equation*}
		where $a$ and $b$ are positive constants.
		
		For the operator ${\T}_0$, the condition \eqref{eq:galerkin_lemma_main_condition} is interpreted as follows:
		\begin{equation}\label{eq:galerkin_rem_spec_condition}
			\brac{{\T}_0 u,u} = b \brac{{\A}_0 u,u} + a \norm{u}^2\,,
		\end{equation}
		and, as defined, ${\A}_0 = - {\d^2} / {\d x^2}$, it is well-known that this operator is positive definite ({\cf} \textbf{Chap. 8} in \citep{Rektorys1980}). Furthermore, by relation \eqref{eq:galerkin_rem_spec_condition}, the positive definiteness of ${\A}_0$ implies that the operator ${\T}_0$ is also positive definite.
		
		The orthogonality relation \eqref{eq:galerkin_orthog_prop} satisfied by the shifted Legendre polynomials yields
		\begin{equation}\label{eq:galerkin_rem_b_orth_prop}
			\brac{{\A}_0 \phi_i,\phi_m} = \brac{\widehat{P}_i,\widehat{P}_m} = \delta_{im}\,.
		\end{equation}
		
		Finally, by \lemref{prop:lem_galerkin_positive_def_gen}, and under the conditions \eqref{eq:galerkin_rem_spec_condition} and \eqref{eq:galerkin_rem_b_orth_prop}, it follows that the coefficient matrices of the subsystems of the Galerkin linear equations, resulting from the application of the Legendre–Galerkin spectral method, are positive definite, as stated in \thmref{prop:thm_galerkin_positive_def_syst}.
	\end{remark}
	
	\subsection{Some Notes on Solving the Derived Subsystems of the Galerkin Linear Equations}\label{subsec:cholesky_alg}
	
	In this section, our objective is to establish an efficient approach for addressing the subsystems of the linear equations \eqref{eq:galerkin_system_matrx_form_u} and \eqref{eq:galerkin_system_matrx_form_v} that originate from the Legendre–Galerkin spectral approximation. Specifically, we focus on solving the subsystems of Galerkin linear equations at each temporal step. For simplicity, the temporal layer index $k$ is suppressed. Each subsystem is subsequently represented in matrix-vector form as follows:
	\begin{equation}\label{eq:galerkin_eff_alg_unif_system}
		\vm{\A}_{N} \vm{w} = \vm{f}\,,
	\end{equation}
	where the coefficient matrix of the system is expressed as:
	\begin{equation*}
		\vm{\A}_{N} = \vm{\H}_N + \frac{4}{a \ell^2} b \vm{\I}_N\,,
	\end{equation*}
	and the unknown vector $\vm{w}$, which has $N$ components, is defined as:
	\begin{equation*}
		\vm{w} = \tps{\brac{w_1,w_2,\ldots,w_N}}\,.
	\end{equation*}
	Furthermore, the vector $\vm{f} = \tps{\brac{f_1,f_2,\ldots,f_N}}$ assumes values as specified in \eqref{eq:galerkin_system_matrx_form_u} and \eqref{eq:galerkin_system_matrx_form_v}, depending on the subsystem being solved. For subsystem \eqref{eq:galerkin_system_matrx_form_u}, the coefficients $\brac{a,b}$ are assigned the values $\brac{1,\tau^2 q_{k,N} / 2}$, whereas for subsystem \eqref{eq:galerkin_system_matrx_form_v}, the coefficients are $\brac{1 + \tau^2 \delta / 2,\tau^2 \gamma / 2}$.
	
	The matrix $\vm{\A}_{N} \in \real^{N \times N}$ is distinguished by a sparsity pattern in which nonzero entries appear exclusively on the main diagonal, the second sub-diagonal, and the second super-diagonal ({\cf} \refitem{Subsec.}{subsec:galerkin}). Moreover, $\vm{\A}_{N}$ is symmetric, and as shown in \thmref{prop:thm_galerkin_positive_def_syst}, it is positive-definite. We shall occasionally refer to this matrix as a tridiagonal matrix with a gap, where the ``gap'' refers to the two skipped diagonals: one below and one above the main diagonal.
	
	We propose the following method to solve the system of Galerkin linear equations \eqref{eq:galerkin_eff_alg_unif_system}. Leveraging the structure of the matrix $\vm{\A}_{N}$, the system can be decomposed into two independent subsystems, as described below.
	
	If the number of basis functions is even, {\ie}, $N = 2s$, where $s \in \nat$, the unknowns are enumerated in the following manner:
	\begin{equation*}
		w_{2j - 1} = \widetilde{w}_j \quad \text{and}\quad w_{2j} = \widehat{w}_j\,,\quad \text{for}\,\, j = 1,2,\ldots,s\,.
	\end{equation*}
	
	Similarly, when the number of basis functions is odd, that is, $N = 2s - 1$ with $s \in \nat$, the unknowns are reordered as follows:
	\begin{align*}
		w_{2j - 1} &= \widetilde{w}_j\,,\quad \text{for}\,\, j = 1,2,\ldots,s\,, \\
		w_{2j} &= \widehat{w}_j\,,\quad \text{for}\,\, j = 1,2,\ldots,s - 1\,.
	\end{align*}
	
	The matrices associated with the two independent subsystems, resulting from the considered decomposition, are tridiagonal. These systems can be solved in parallel for the unknowns $\widetilde{w}_j$ and $\widehat{w}_j$ using the Thomas algorithm (tridiagonal matrix algorithm). Although this algorithm is generally not stable, stability is assured when the matrix is diagonally dominant (either row- or column-wise) or symmetric positive definite (see \textbf{Theorem 9.12} in Higham \citep{Higham2002}). It is straightforward to establish that, in this case, the matrices corresponding to both decomposed tridiagonal subsystems are positive definite. In order to derive closed-form expressions for the unknowns $\widetilde{w}_j$ and $\widehat{w}_j$, one may apply the Cholesky decomposition to the coefficient matrix associated with the linear system \eqref{eq:galerkin_eff_alg_unif_system}.
	
	Consider a variant of the classical Cholesky decomposition, referred to as the square-root-free Cholesky decomposition, applied to the real symmetric positive-definite matrix $\vm{\A}_{N}$, which corresponds to the coefficient matrix of the system of linear equations \eqref{eq:galerkin_eff_alg_unif_system}, and is formulated as follows:
	\begin{equation*}
		\vm{\A}_{N} = \vm{L} \vm{D} \tps{\vm{L}}\,.
	\end{equation*}
	In the case of the square-root-free Cholesky decomposition of a tridiagonal matrix with a gap, the matrix $\vm{L}$ is defined such that all entries along its main diagonal are equal to unity, with nonzero off-diagonal entries appearing only on the second sub-diagonal. The matrix $\vm{D} = \diag{d_1,\ldots,d_N}$ is diagonal, and $\tps{\vm{L}}$ denotes the transpose of $\vm{L}$.
	
	It should be observed that this factorization not only provides closed-form solutions to the system \eqref{eq:galerkin_eff_alg_unif_system}, but also permits the parallel computation of solutions corresponding to the unknowns with odd and even indices.
	
	Under this decomposition, the system \eqref{eq:galerkin_eff_alg_unif_system} takes the following form:
	\begin{equation*}
		\vm{L} \vm{D} \tps{\vm{L}} \vm{w} = \vm{f}\,.
	\end{equation*}
	This system evidently allows for decomposition into the ensuing subsystems by introducing auxiliary unknowns $\vm{y} = \tps{\brac{y_1,y_2,\ldots,y_N}}$ and $\vm{z} = \tps{\brac{z_1,z_2,\ldots,z_N}}$, as follows:
	\begin{equation}\label{eq:galerkin_cholesky_decom}
		\begin{cases}
			\vm{L} \vm{z}       &= \vm{f}\,, \\
			\vm{D} \vm{y}       &= \vm{z}\,, \\
			\tps{\vm{L}} \vm{w} &= \vm{y}\,.
		\end{cases}
	\end{equation}
	As in the preceding case, two distinct scenarios should be considered, depending on whether $N$, the number of basis functions, is even or odd, which in turn determines the dimension of the associated matrix. Thus, the problem of solving system \eqref{eq:galerkin_eff_alg_unif_system} is reduced to that of solving system \eqref{eq:galerkin_cholesky_decom}, for which the corresponding closed-form solutions are subsequently discussed.
	
	Consider the case where $N$ is even, {\ie}, $N = 2s$ with $s \in \nat$. In this case, the solution to the system \eqref{eq:galerkin_cholesky_decom} can be expressed as follows:
	\begin{align*}
		z_1 = f_1\,, \qquad z_2 &= f_2\,, \\
		d_1 = C_1 + \frac{4}{a \ell^2} b\,,\qquad d_2 &= C_2 + \frac{4}{a \ell^2} b\,.
	\end{align*}
	For $j = 2,3,\ldots,s$, the following holds:
	\begin{align*}
		z_{2j - 1} = f_{2j - 1} + \frac{B_{2 \brac{j - 1}}}{d_{2j - 3}} z_{2j - 3}\,,\qquad z_{2j} &= f_{2j} + \frac{B_{2j - 1}}{d_{2 \brac{j - 1}}} z_{2 \brac{j - 1}}\,, \\
		d_{2j - 1} = \brac{C_{2j - 1} + \frac{4}{a \ell^2} b} - \frac{B_{2 \brac{j - 1}}^2}{d_{2j - 3}}\,,\qquad d_{2j} &= \brac{C_{2j} + \frac{4}{a \ell^2} b} - \frac{B_{2j - 1}^2}{d_{2 \brac{j - 1}}}\,,
	\end{align*}
	where $C_m = 2 A_{m - 1}^{2} A_{m + 1}^{2}$ for $m = 1,2,\ldots,N$, and $B_m = A_{m - 1} A_{m}^{2} A_{m + 1}$ for $m = 2,3,\ldots,N - 1$. Recall that $A_m = 1 / \sqrt{2m + 1}$.
	
	The vector $\vm{y}$ is computed component-wise by multiplying each entry of the vector $\vm{z}$ by the reciprocal of the corresponding diagonal element of the matrix $\vm{D}$. Therefore, for $j = 1,2,\ldots,s$, we have:
	\begin{equation*}
		y_{2j - 1} = \frac{z_{2j - 1}}{d_{2j - 1}}\,,\qquad y_{2j} = \frac{z_{2j}}{d_{2j}}\,.
	\end{equation*}
	Observing that $w_{2s} = y_{2s}$ and $w_{2s - 1} = y_{2s - 1}$, the solution to the system \eqref{eq:galerkin_eff_alg_unif_system} is subsequently determined for $j = s - 1,s - 2,\ldots,1$ and is represented by:
	\begin{equation*}
		w_{2j - 1} = y_{2j - 1} + \frac{B_{2j}}{d_{2j - 1}} w_{2j + 1}\,,\qquad w_{2j} = y_{2j} + \frac{B_{2j + 1}}{d_{2j}} w_{2 \brac{j + 1}}\,.
	\end{equation*}
	
	Let $N$ be an odd number, so that $N = 2s - 1$, where $s \in \nat$. Even in this setting, the same formulas for determining the components of the unknown vectors $\vm{z}$, $\vm{y}$ and $\vm{w}$, as well as the diagonal matrix $\vm{D}$, continue to be valid. The only distinction occurs in the computation of even-indexed components within a for-loop. Specifically, for $j = 2,3,\ldots,s - 1$, the components $z_{2j}$ and $d_{2j}$ are determined using the closed-form expressions derived previously. Likewise, for $j = 1,2,\ldots,s - 1$, the components $y_{2j}$ are defined as in the prior case. Finally, noting that $w_{2 \brac{s - 1}} = d_{2 \brac{s - 1}}$, the remaining components $w_{2j}$ for $j = s - 2, s - 3,\ldots,1$ are computed recursively by applying the formula established in the preceding case.
	
	Thus, as shown above, the Galerkin system is solved explicitly, and the corresponding computational cost is optimal, namely of order $\bigO \left( N \right)$.
	
	If sine functions are chosen as both trial and test functions, then the matrix of the resulting Galerkin system is diagonal. However, in order to form the right-hand side of this system, one must multiply the matrix associated with the coupling terms by the relevant coefficient vectors. Since this matrix is almost dense, this requires $\bigO \left( N^2 \right)$ operations. Hence, the computation of the solution at each time layer requires $\bigO \left( N^2 \right)$ arithmetic operations.
	
	Consequently, the Galerkin system can be solved at each time layer with $\bigO \left( N \right)$ operations for the Legendre–Galerkin spectral method, whereas the corresponding sine-based formulation requires $\bigO \left( N^2 \right)$ operations. Since we tackle a dynamic problem, the total number of arithmetic operations is $\bigO \left( n \times N \right)$ and $\bigO \left( n \times N^2 \right)$, respectively, where $n$ denotes the number of subintervals in the partition of the interval $\qbrac{0,T}$. Therefore, for relatively large values of $n$, the difference between the computational costs becomes noticeable. Thus, in this sense, the Legendre–Galerkin spectral method has a computational advantage in the present setting.
	
	In addition to the above, in order to implement the three-layer semi-discrete scheme \eqref{eq:tim_spec_semidiscr_u}–\eqref{eq:tim_spec_semidiscr_v} efficiently, it is necessary to obtain a simple recurrence relation between the Galerkin expansion coefficients of the corresponding approximate solutions. This is achieved by the proposed method (see \refitem{Subsection}{subsec:galerkin}). If, instead of employing the Legendre–Galerkin spectral method, one represents the solution of problem \eqref{eq:tim_spec_semidiscr_u}–\eqref{eq:tim_spec_semidiscr_v}, subject to the homogeneous boundary conditions \eqref{eq:tim_spec_bound_conds}, by means of Green's formula, then the derivation of a recurrence relation similar to the one mentioned above becomes substantially more involved. Consequently, efficient computation is no longer feasible.
	
	At the end of this subsection, we would like to draw attention to one nuance that illustrates the advantages of the proposed method. Suppose that the data of the discrete problem are continuous functions. Clearly, these functions can be approximated by Bernstein polynomials. In this case, the solution of the discrete problem will be a polynomial, since our scheme is locally linear. Moreover, this solution can be constructed exactly by means of the Legendre–Galerkin spectral method. Since the discrete problem is well posed, the polynomial solution provides an approximation to the exact solution of the same order of accuracy as the approximation of the prescribed data.
	
	\subsection{Estimate of the Error in the Legendre–Galerkin Spectral Method}\label{subsec:ritz_gal}
	
	In this subsection, we shall estimate the error incurred by the Legendre–Galerkin spectral method when applied to equations \eqref{eq:tim_spec_expanded_eq_u} and \eqref{eq:tim_spec_expanded_eq_v}. For this purpose, we reformulate these equations in the following manner:
	\begin{equation}\label{eq:ritz_gal_main}
		a w \brac{x} - b \frac{\d^2 w \brac{x}}{\d x^2} = f \brac{x}\,,\quad x \in \qbrac{0,\ell}\,.
	\end{equation}
	Here, the pair of coefficients $\brac{a,b}$ is defined as $\brac{a,b} = \brac{1,\tau^2 q_{k,N} / 2}$ for equation \eqref{eq:tim_spec_expanded_eq_u} and $\brac{a,b} = \brac{1 + \tau^2 \delta / 2,\tau^2 \gamma / 2}$ for equation \eqref{eq:tim_spec_expanded_eq_v}. In equation \eqref{eq:ritz_gal_main}, the right-hand side $f \brac{x}$ is a continuous function, coinciding with the right-hand side of either equation \eqref{eq:tim_spec_expanded_eq_u} or equation \eqref{eq:tim_spec_expanded_eq_v}, depending on the context under consideration.
	
	Equation \eqref{eq:ritz_gal_main} is considered subject to homogeneous boundary conditions, which are specified as follows:
	\begin{equation}\label{eq:ritz_gal_bound_cond}
		w \brac{0} = w \brac{\ell} = 0\,.
	\end{equation}
	
	The variational formulation of the problem \eqref{eq:ritz_gal_main}-\eqref{eq:ritz_gal_bound_cond} is well-established in the literature ({\cf}, {\eg}, Kantorovich and Krylov \citep{Kantorovich1958}). The problem \eqref{eq:ritz_gal_main}-\eqref{eq:ritz_gal_bound_cond} is thus equivalent to the problem of minimizing the following functional:
	\begin{equation}\label{eq:ritz_gal_variat_form}
		I \brac{w} = \int\limits_{0}^{\ell} \brac{b w^{\prime 2} + a w^2 - 2 f w} \d x\,,\quad w \brac{0} = w \brac{\ell} = 0\,.
	\end{equation}
	
	The subsequent discussion follows the proof technique outlined in the aforementioned book (see \textbf{Chap. IV}, \textbf{Sec. 4} in \citep{Kantorovich1958}), which concerns error estimation in the variational method for ordinary differential equations, with sines taken as the trial (basis) functions.
	
	Let $\widetilde{w} \brac{x}$ be any function that satisfies the homogeneous boundary conditions $\widetilde{w} \brac{0} = \widetilde{w} \brac{\ell} = 0$. Define the function $\eta \brac{x} = \widetilde{w} \brac{x} - w \brac{x}$. It follows immediately that $\eta \brac{0} = \eta \brac{\ell} = 0$.
	
	From \eqref{eq:ritz_gal_variat_form}, a straightforward transformation yields:
	\begin{align*}
		I \brac{\widetilde{w}} - I \brac{w} &= I \brac{w + \eta} - I \brac{w} \\
		&= 2 \int\limits_{0}^{\ell} \brac{b w^{\prime} \eta^{\prime} + a w \eta - f \eta} \d x + \int\limits_{0}^{\ell} \brac{b \eta^{\prime 2} + a {\eta}^2} \d x\,.
	\end{align*}
	In the resulting expression, the first summand is simply the variation of the integral $I$, which vanishes, {\ie}, $\delta I = 0$. Consequently, we obtain:
	\begin{equation}\label{eq:ritz_gal_diff_var}
		I \brac{\widetilde{w}} - I \brac{w} = \int\limits_{0}^{\ell} \brac{b \brac{\widetilde{w}^{\prime} - w^{\prime}}^{2} + a \brac{\widetilde{w} - w}^2} \d x\,.
	\end{equation}
	
	Recall that we have chosen the following system of functions as the basis functions:
	\begin{equation*}
		\phi_m \brac{x} = \int\limits_{0}^{x} \widehat{P}_m \brac{s} \d s\,,\quad m = 1,2,\ldots\,,
	\end{equation*}
	where $\widehat{P}_m \brac{x}$ represents the orthonormal shifted Legendre polynomial defined on the interval $\qbrac{0,\ell}$.
	
	It is known that if $u \brac{x} \in C^p \brac{\qbrac{0,\ell}}$, with $p \geq 1$, then the following bound holds ({\cf} \textbf{Theorem 4.10} in Suetin \citep{Suetin1979}):
	\begin{equation}\label{eq:ritz_gal_suetin}
		\abs{u \brac{x} - \sum_{i = 0}^{N} \hat{c}_i \widehat{P}_i \brac{x}} \leq \frac{\lgsuetin}{N^{p - \nicefrac{1}{2}}}\,,\quad x \in \qbrac{0,\ell}\,.
	\end{equation}
	Here, $\hat{c}_i$ denotes the coefficient in the Fourier–Legendre series expansion.
	
	Assume that the solution to the problem \eqref{eq:ritz_gal_main}-\eqref{eq:ritz_gal_bound_cond} is such that $w \brac{x} \in C^2 \brac{\qbrac{0,\ell}}$. From estimate \eqref{eq:ritz_gal_suetin}, we have:
	\begin{equation}\label{eq:ritz_gal_w_diff}
		\abs{w^{\prime} \brac{x} - S_N \brac{x}} \leq \frac{\lgdiffw}{\sqrt{N}}\,,\quad x \in \qbrac{0,\ell}\,,
	\end{equation}
	where
	\begin{equation*}
		S_N \brac{x} = \sum_{i = 0}^{N} \hat{a}_i \widehat{P}_i \brac{x}\,,\quad \hat{a}_i = \int\limits_{0}^{\ell} w^{\prime} \brac{x} \widehat{P}_i \brac{x} \d x\,.
	\end{equation*}
	Observe that
	\begin{equation*}
		\hat{a}_0 = \int\limits_{0}^{\ell} w^{\prime} \brac{x} \widehat{P}_0 \brac{x} \d x = \frac{1}{\sqrt{\ell}} \int\limits_{0}^{\ell} w^{\prime} \brac{x} \d x = \frac{1}{\sqrt{\ell}} \brac{w \brac{\ell} - w \brac{0}} = 0\,.
	\end{equation*}
	
	Consider the following combination:
	\begin{equation*}
		\Phi_N \brac{x} = \sum_{m = 1}^{N} \hat{a}_m \phi_m \brac{x} = \sum_{m = 1}^{N} \hat{a}_m \int\limits_{0}^{x} \widehat{P}_m \brac{s} \d s\,.
	\end{equation*}
	Taking into account that $\Phi_{N}^{\prime} \brac{x} = S_N \brac{x}$, it follows from inequality \eqref{eq:ritz_gal_w_diff} that:
	\begin{align*}
		\abs{w \brac{x} - \Phi_N \brac{x}} &= \abs{\int\limits_{0}^{x} \brac{w^{\prime} \brac{s} - \Phi_{N}^{\prime} \brac{s}} \d s} = \abs{\int\limits_{0}^{x} \brac{w^{\prime} \brac{s} - S_N \brac{s}} \d s} \\
		&\leq \int\limits_{0}^{\ell} \abs{w^{\prime} \brac{x} - S_N \brac{x}} \d x \leq \frac{\lgdiffw \ell}{\sqrt{N}}\,.
	\end{align*}
	Thus, we have established the completeness of the system of functions $\left\{ \phi_m \brac{x} \right\}_{m = 1}^{\infty}$.
	
	Consider the following approximation:
	\begin{equation*}
		w_N \brac{x} = \sum_{m = 1}^{N} a_m \phi_m \brac{x}\,,
	\end{equation*}
	where $a_m$ ($m = 1,2,\ldots,N$) represents the solution to the Ritz–Galerkin system.
	
	It is important to note that, in our case, the systems of linear equations emerging from the Ritz–Galerkin and Legendre–Galerkin spectral methods coincide.
	
	On the linear hull spanned by the system of functions $\phi_1 \brac{x}, \phi_2 \brac{x},\ldots,\phi_N \brac{x}$, the functional $I \brac{w}$ attains its minimum at $w \brac{x} = w_N \brac{x}$. From this, the following inequality holds:
	\begin{equation*}
		I \brac{w_N} - I \brac{w} \leq I \brac{\Phi_N} - I \brac{w}\,.
	\end{equation*}
	By virtue of \eqref{eq:ritz_gal_diff_var}, this inequality takes the following form:
	\begin{equation}\label{eq:ritz_gal_unif_integ_ineq}
		\int\limits_{0}^{\ell} \qbrac{b \brac{{w}_{N}^{\prime} - w^{\prime}}^{2} + a \brac{{w}_{N} - w}^2} \d x \leq \int\limits_{0}^{\ell} \qbrac{b \brac{{\Phi}_{N}^{\prime} - w^{\prime}}^{2} + a \brac{{\Phi}_{N} - w}^2} \d x\,.
	\end{equation}
	It follows immediately from \eqref{eq:ritz_gal_unif_integ_ineq} that the following result holds:
	\begin{equation*}
		a \int\limits_{0}^{\ell} \brac{{w}_{N} - w}^2 \d x \leq \int\limits_{0}^{\ell} \qbrac{b \brac{{\Phi}_{N}^{\prime} - w^{\prime}}^{2} + a \brac{{\Phi}_{N} - w}^2} \d x\,.
	\end{equation*}
	or, equivalently:
	\begin{equation}\label{eq:ritz_gal_norm_error}
		a \norm{{w}_{N} - w}^2 \leq b \norm{{\Phi}_{N}^{\prime} - w^{\prime}}^2 + a \norm{{\Phi}_{N} - w}^2\,.
	\end{equation}
	
	We shall now apply the classical Steklov inequality, which states:
	\begin{equation}\label{eq:ritz_gal_func_bound_diff}
		\norm{w} \leq \frac{\ell}{\pi} \norm{w^{\prime}}\,.
	\end{equation}
	Hence, in view of inequality \eqref{eq:ritz_gal_func_bound_diff}, it follows that
	\begin{equation}\label{eq:ritz_gal_diff_ineq_der}
		\norm{\Phi_N - w} \leq \frac{\ell}{\pi} \norm{\Phi_{N}^{\prime} - w^{\prime}}\,.
	\end{equation}
	Substituting inequality \eqref{eq:ritz_gal_diff_ineq_der} into \eqref{eq:ritz_gal_norm_error}, we find that:
	\begin{equation}\label{eq:ritz_gal_norm_error_refined}
		\norm{{w}_{N} - w} \leq b_1 \norm{{\Phi}_{N}^{\prime} - w^{\prime}}\,,\quad b_1 = \sqrt{\frac{b}{a} + \frac{\ell^2}{\pi^2}}\,.
	\end{equation}
	Taking into account that $\Phi_{N}^{\prime} \brac{x} = S_N \brac{x}$, the following inequality is obtained from \eqref{eq:ritz_gal_norm_error_refined}:
	\begin{equation}\label{eq:ritz_gal_norm_error_s}
		\norm{{w}_{N} - w} \leq b_1 \norm{S_N - w^{\prime}}\,.
	\end{equation}
	
	Let us assume that the solution to problem \eqref{eq:ritz_gal_main}-\eqref{eq:ritz_gal_bound_cond}, $w \brac{x} \in C^{p + 1} \brac{\qbrac{0,\ell}}$, with $p \geq 1$. Then, by virtue of bound \eqref{eq:ritz_gal_suetin}, inequality \eqref{eq:ritz_gal_norm_error_s} leads to the following estimate:
	\begin{equation}\label{eq:ritz_gal_estimate_final}
		\norm{{w}_{N} - w} \leq \frac{\lgerrorgal}{N^{p - \nicefrac{1}{2}}}\,.
	\end{equation}
	
	Thus, the following theorem is valid.
	\begin{theorem}\label{prop:thm_error_galerkin}
		Let the solution to the boundary value problem \eqref{eq:ritz_gal_main}-\eqref{eq:ritz_gal_bound_cond} be such that $w \brac{x} \in C^{p + 1} \brac{\qbrac{0,\ell}}$ with $p \geq 1$. Then, for the approximation $w_N \brac{x} = a_1 \phi_1 \brac{x} + a_2 \phi_2 \brac{x} + \ldots + a_N \phi_N \brac{x}$, where the coefficients $a_m$ $\brac{m = 1,2,\ldots,N}$ are determined by solving the Ritz–Galerkin system, the error bound given by inequality \eqref{eq:ritz_gal_estimate_final} holds.
	\end{theorem}
	
	\begin{remark}\label{prep:rem_error_general_eq}
		It is easy to observe that the validity of the error estimate stated in \eqref{eq:ritz_gal_estimate_final} naturally extends to a more general equation. Specifically, consider the equation
		\begin{equation}\label{eq:ritz_gal_rem_gen_eq}
			\frac{\d}{\d x}\brac{p \brac{x} \frac{\d w}{\d x}} - q \brac{x} w \brac{x} = f \brac{x}\,,\quad x \in \qbrac{0,\ell}\,,
		\end{equation}
		subject to homogeneous Dirichlet boundary conditions. Assume that the functions $p \brac{x}$, $q \brac{x}$, and $f \brac{x}$ satisfy the standard regularity conditions. More precisely, $p \brac{x}$ is continuously differentiable, while $q \brac{x}$ and $f \brac{x}$ are continuous functions. Furthermore, the coefficient functions $p \brac{x}$ and $q \brac{x}$ are assumed to satisfy the uniform positivity conditions
		\begin{equation*}
			p \brac{x} \geq p_0 > 0\quad \text{and}\quad q \brac{x} \geq q_0 > 0\,.
		\end{equation*}
	\end{remark}
	
	One should note that the error $w \brac{x} - w_N \brac{x}$ can also be estimated using the uniform norm. To achieve this, additional transformations are required.
	
	It is evident from inequality \eqref{eq:ritz_gal_unif_integ_ineq} that the following holds:
	\begin{equation*}
		b \norm{{w}_{N}^{\prime} - w^{\prime}}^2 \leq b \norm{{\Phi}_{N}^{\prime} - w^{\prime}}^2 + a \norm{{\Phi}_{N} - w}^2\,.
	\end{equation*}
	From this, taking into account \eqref{eq:ritz_gal_diff_ineq_der}, we deduce:
	\begin{equation*}
		\norm{{w}_{N}^{\prime} - w^{\prime}}^2 \leq \brac{1 + \frac{a {\ell}^2}{b \pi^2}} \norm{{\Phi}_{N}^{\prime} - w^{\prime}}^2\,,
	\end{equation*}
	or, equivalently:
	\begin{equation}\label{eq:ritz_gal_unif_diff_error}
		\norm{{w}_{N}^{\prime} - w^{\prime}} \leq b_2 \norm{S_N - w^{\prime}}\,,\quad b_2 = \sqrt{1 + \frac{a {\ell}^2}{b \pi^2}}\,.
	\end{equation}
	
	From the representation
	\begin{equation*}
		w \brac{x} = \int\limits_{0}^{x} w^{\prime} \brac{s} \d s\,,
	\end{equation*}
	and by standard reasoning, one obtains ({\cf} \textbf{Chap. IV}, \textbf{Sec. 4} in Kantorovich and Krylov \citep{Kantorovich1958}):
	\begin{equation*}
		\max_{0 \leq x \leq \ell} \abs{w \brac{x}} \leq \sqrt{\frac{\ell}{2}} \norm{w^{\prime}}\,.
	\end{equation*}
	By this inequality, we have:
	\begin{equation*}
		\max_{0 \leq x \leq \ell} \abs{w_N \brac{x} - w \brac{x}} \leq \sqrt{\frac{\ell}{2}} \norm{{w}_{N}^{\prime} - w^{\prime}}\,.
	\end{equation*}
	From this, taking \eqref{eq:ritz_gal_unif_diff_error} into consideration, it follows that:
	\begin{equation}\label{eq:ritz_gal_error_unif_norm}
		\max_{0 \leq x \leq \ell} \abs{w_N \brac{x} - w \brac{x}} \leq b_2 \sqrt{\frac{\ell}{2}} \norm{S_N - w^{\prime}}\,.
	\end{equation}
	
	If $w \brac{x} \in C^{p + 1} \brac{\qbrac{0,\ell}}$ with $p \geq 1$, then from \eqref{eq:ritz_gal_error_unif_norm}, and in accordance with \eqref{eq:ritz_gal_suetin}, the following inequality is obtained:
	\begin{equation}\label{eq:ritz_gal_unif_error_final}
		\max_{0 \leq x \leq \ell} \abs{w_N \brac{x} - w \brac{x}} \leq \sqrt{\frac{\ell}{2}}\frac{b_2 \lgerrunifnorm}{N^{p - \nicefrac{1}{2}}}\,.
	\end{equation}
	
	It is observed that by substituting the values of $a$ and $b$ corresponding to equations \eqref{eq:tim_spec_expanded_eq_u} and \eqref{eq:tim_spec_expanded_eq_v} into the expression for $b_2$, the inequality $b_2 \leq \lgboundbtwo / \tau$, with $\tau = T / n$, follows immediately. If the number of divisions $n$ is chosen as the integer part of $N^s$ ($n = \qbrac{N^s}$), where $p - 1/2 - s >0$, then bound \eqref{eq:ritz_gal_unif_error_final} takes the following form:
	\begin{equation*}
		\max_{0 \leq x \leq \ell} \abs{w_N \brac{x} - w \brac{x}} \leq \frac{\lgfinalest}{N^{p - \nicefrac{1}{2} - s}}\,,
	\end{equation*}
	where $p > 1 / 2 + s$, $p \geq 1$ is a positive integer, and $s > 0$.
	
	\begin{remark}\label{prep:rem_unif_error_general_eq}
		The estimate \eqref{eq:ritz_gal_unif_error_final} continues to hold even for the general equation \eqref{eq:ritz_gal_rem_gen_eq}. In this setting, the constant $b_2$ is precisely defined by the following relation:
		\begin{equation}\label{eq:ritz_gal_rem_unif_btwo}
			b_2 = \sqrt{p_0^{-1} \brac{\max_{0 \leq x \leq \ell} p \brac{x} + \frac{\ell^2}{\pi^2} \max_{0 \leq x \leq \ell} q \brac{x}}}\,.
		\end{equation}
	\end{remark}
	
	In our opinion, the reader will be interested in an explicit error estimate with a uniform norm in the natural class of solutions for equation \eqref{eq:ritz_gal_rem_gen_eq}.
	\begin{remark}\label{prep:rem_unif_explic_error_general_eq}
		Let equation \eqref{eq:ritz_gal_rem_gen_eq} be considered with homogeneous boundary conditions. If the standard regularity conditions are satisfied (see \remref{prep:rem_error_general_eq}), then the following estimate is valid:
		\begin{equation}\label{eq:rem_gener_ritz_gal_unif_norm_desired}
			\max_{0 \leq x \leq \ell} \abs{w_N \brac{x} - w \brac{x}} \leq \frac{\hat{b}_2}{\sqrt{2 \brac{2N + 1} \brac{2N + 5}}} \norm{f}\,,
		\end{equation}
		where $\hat{b}_2 = \hat{b}_1 b_2 \ell \sqrt{\ell}$, $b_2$ is specified by equality \eqref{eq:ritz_gal_rem_unif_btwo}, and
		\begin{equation*}
			\hat{b}_1 = p_0^{-2} \brac{\frac{\ell}{\pi} \max_{0 \leq x \leq \ell} \abs{p^{\prime} \brac{x}} + \frac{{\ell}^2}{\pi^2} \max_{0 \leq x \leq \ell} q \brac{x} + p_0}\,.
		\end{equation*}
	\end{remark}
	\begin{proof}
		Let $\hat{c}_m$ denote the coefficients of the Legendre–Fourier series expansion of the function $w^{\prime \prime} \brac{x}$, given by
		\begin{equation*}
			\hat{c}_m = \int\limits_{0}^{\ell} w^{\prime \prime} \brac{x} \widehat{P}_m \brac{x} \d x \quad \brac{m = 0,1,2,\ldots}\,.
		\end{equation*}
		
		We now establish the relationship between the coefficients $\hat{c}_m$ and $\hat{a}_m$, where $\hat{a}_m$ denotes the coefficients of the Legendre–Fourier series expansion of the function $w^{\prime} \brac{x}$. For the orthonormal shifted Legendre polynomials, the following relation holds:
		\begin{equation}\label{eq:gener_ritz_gal_orthnorm_leg_diff}
			\widehat{P}_m \brac{x} = \frac{\ell}{2} A_{m} \brac{A_{m + 1} \widehat{P}_{m + 1}^{\prime} \brac{x} - A_{m - 1} \widehat{P}_{m - 1}^{\prime} \brac{x}}\,.
		\end{equation}
		Upon considering \eqref{eq:gener_ritz_gal_orthnorm_leg_diff}, the following result is derived:
		\begin{equation}\label{eq:gener_ritz_gal_hat_a_integr}
			\hat{a}_m = \int\limits_{0}^{\ell} w^{\prime} \brac{x} \widehat{P}_m \brac{x} \d x = \frac{\ell}{2} A_{m} \brac{A_{m + 1} \int\limits_{0}^{\ell} w^{\prime} \brac{x} \widehat{P}_{m + 1}^{\prime} \brac{x} \d x - A_{m - 1} \int\limits_{0}^{\ell} w^{\prime} \brac{x} \widehat{P}_{m - 1}^{\prime} \brac{x} \d x}\,.
		\end{equation}
		By applying integration by parts to the first summand of equality \eqref{eq:gener_ritz_gal_hat_a_integr}, we obtain:
		\begin{align*}
			A_{m + 1} \int\limits_{0}^{\ell} w^{\prime} \brac{x} \widehat{P}_{m + 1}^{\prime} \brac{x} \d x &= A_{m + 1} \brac{\qbrac{w^{\prime} \brac{x} \widehat{P}_{m + 1} \brac{x}}_{0}^{\ell} - \int\limits_{0}^{\ell} w^{\prime \prime} \brac{x} \widehat{P}_{m + 1} \brac{x} \d x} \\
			&= \frac{1}{\sqrt{\ell}} \brac{w^{\prime} \brac{\ell} + \brac{-1}^m w^{\prime} \brac{0}} - A_{m + 1} \hat{c}_{m + 1}\,.
		\end{align*}
		Using the same reasoning, for the second summand of equation \eqref{eq:gener_ritz_gal_hat_a_integr}, we arrive at the following result:
		\begin{equation*}
			A_{m - 1} \int\limits_{0}^{\ell} w^{\prime} \brac{x} \widehat{P}_{m - 1}^{\prime} \brac{x} \d x = \frac{1}{\sqrt{\ell}} \brac{w^{\prime} \brac{\ell} + \brac{-1}^m w^{\prime} \brac{0}} - A_{m - 1} \hat{c}_{m - 1}\,.
		\end{equation*}
		By substituting the last two resulting equalities into \eqref{eq:gener_ritz_gal_hat_a_integr}, one yields:
		\begin{equation*}
			\hat{a}_m = \frac{\ell}{2} A_{m} \brac{A_{m - 1} \hat{c}_{m - 1} - A_{m + 1} \hat{c}_{m + 1}}\,,
		\end{equation*}
		From this, it follows that:
		\begin{equation*}
			\hat{a}_{m}^{2} \leq \frac{\ell^2}{2} A_{m}^{2} \brac{A_{m - 1}^{2} \hat{c}_{m - 1}^{2} + A_{m + 1}^{2} \hat{c}_{m + 1}^{2}}\,.
		\end{equation*}
		Summing both sides of the preceding inequality from $N + 1$ to $\infty$ and invoking Bessel's inequality leads to the following result:
		\begin{align}\label{eq:gener_ritz_gal_bess_ineq}
			\norm{S_N - w^{\prime}}^2 = \sum_{m = N + 1}^{\infty} \hat{a}_{m}^{2} &\leq \frac{\ell^2}{2} \sum_{m = N + 1}^{\infty} A_{m}^{2} \brac{A_{m - 1}^{2} \hat{c}_{m - 1}^{2} + A_{m + 1}^{2} \hat{c}_{m + 1}^{2}}\nonumber \\
			& \leq \frac{\ell^2}{2} A_{N + 1}^{2} \brac{A_{N}^{2} \sum_{m = N + 1}^{\infty} \hat{c}_{m - 1}^{2} + A_{N + 2}^{2} \sum_{m = N + 1}^{\infty} \hat{c}_{m + 1}^{2}}\nonumber \\
			& \leq \frac{\ell^2}{2} A_{N + 1}^{2} \brac{A_{N}^{2} + A_{N + 2}^{2}} \sum_{m = N}^{\infty} \hat{c}_{m}^{2} \leq \frac{\ell^2}{\brac{2N + 1} \brac{2N + 5}} \norm{w^{\prime \prime}}^2\,.
		\end{align}
		
		It is observed that, analogous to equation \eqref{eq:ritz_gal_main}, inequality \eqref{eq:ritz_gal_error_unif_norm} also holds for equation \eqref{eq:ritz_gal_rem_gen_eq}, where the constant $b_2$ is determined by equality \eqref{eq:ritz_gal_rem_unif_btwo}. Substituting inequality \eqref{eq:gener_ritz_gal_bess_ineq} into \eqref{eq:ritz_gal_error_unif_norm} yields:
		\begin{equation}\label{eq:gener_ritz_gal_unif_norm_sec_derr_w}
			\max_{0 \leq x \leq \ell} \abs{w_N \brac{x} - w \brac{x}} \leq \frac{b_2 \ell \sqrt{\ell}}{\sqrt{2\brac{2N + 1} \brac{2N + 5}}} \norm{w^{\prime \prime}}\,.
		\end{equation}
		
		For equation \eqref{eq:ritz_gal_rem_gen_eq}, from the condition that the variation of the energy functional vanishes, it follows that:
		\begin{equation}\label{eq:gener_ritz_gal_variat_energ_func}
			\int\limits_{0}^{\ell} \brac{p w^{\prime 2} + q w^2 + f w} \d x = 0\,.
		\end{equation}
		By straightforward reasoning, from equation \eqref{eq:gener_ritz_gal_variat_energ_func}, it can be deduced that:
		\begin{equation*}
			\norm{w^{\prime}}^2 \leq \frac{1}{p_0} \norm{f} \norm{w}\,.
		\end{equation*}
		From this, and considering \eqref{eq:ritz_gal_func_bound_diff}, the following is obtained:
		\begin{equation}\label{eq:gener_ritz_gal_inqts_w_diff_w}
			\norm{w^{\prime}} \leq \frac{\ell}{\pi p_0} \norm{f}\,,\quad \norm{w} \leq \frac{\ell^2}{\pi^2 p_0} \norm{f}\,.
		\end{equation}
		
		We shall rewrite equation \eqref{eq:ritz_gal_rem_gen_eq} in the following form:
		\begin{equation*}
			p w^{\prime \prime} = -p^{\prime} w^{\prime} + q w + f\,.
		\end{equation*}
		Whence, based on the estimates provided in \eqref{eq:gener_ritz_gal_inqts_w_diff_w}, the following conclusion follows:
		\begin{equation}\label{eq:gener_ritz_gal_norm_second_w}
			\norm{w^{\prime \prime}} \leq \frac{1}{p_0} \brac{\max_{0 \leq x \leq \ell} \abs{p^{\prime} \brac{x}} \norm{w^{\prime}} + \max_{0 \leq x \leq \ell} q \brac{x} \norm{w} + \norm{f}} \leq \hat{b}_1 \norm{f}\,.
		\end{equation}
		Finally, comparing inequalities \eqref{eq:gener_ritz_gal_unif_norm_sec_derr_w} and \eqref{eq:gener_ritz_gal_norm_second_w} gives the estimate \eqref{eq:rem_gener_ritz_gal_unif_norm_desired}.
	\end{proof}
	
	\section{Illustrative Numerical Results}\label{sec:num_res}
	
	In this section, we discuss four benchmark problems corresponding to the initial–boundary value problem \eqref{eq:timosh_spec_nonlinear}-\eqref{eq:timosh_spec_bound_conds} in order to verify consistency and to validate the performance of the proposed combined numerical schemes against the theoretical findings obtained. For these cases, exact analytical solutions are available. All coefficients appearing in equations \eqref{eq:timosh_spec_nonlinear}-\eqref{eq:timosh_spec_linear} are set equal to one; {\ie}, $\alpha = \beta = \gamma = \delta = a_1 = a_2 = 1$.
	
	In \refitem{Test}{bm:test4}, the initial data are prescribed in the form of a wave packet. In this case, an exact analytical solution is not available. In contrast to the previous cases, all coefficients retain the same values, except for the coupling coefficients, which are specified as $a_1 = a_2 = 0.5$. Furthermore, in all considered settings, the spatial variable satisfies $x \in \qbrac{0, 2}$. The length of the interval is chosen to be $2$ because it coincides with the extent of the orthogonality interval of the standard Legendre polynomials.
	
	The approximation errors for each temporal layer $k = 2,3,\ldots n$ (recall that the cases $k = 0$ and $k = 1$ correspond to the prescribed initial data and are therefore excluded) are defined as the $L^2 \brac{0,\ell}$-norms of the differences between the exact solutions and their corresponding numerical approximations. Thus, we set:
	\begin{equation*}
		E_{1,k} = \norm{u\brac{\cdot, t_k} - \tilde{u}_{k,N}\brac{\cdot}}\,,\quad \text{and}\quad E_{2,k} = \norm{v\brac{\cdot, t_k} - \tilde{v}_{k,N}\brac{\cdot}}\,.
	\end{equation*}
	
	Since the accuracy and computational efficiency of the proposed algorithm are closely related to the effective computation of the integrals involved, these integrals are approximated using a Gauss–Legendre quadrature rule with an error-controlled subdivision of the integration interval.
	
	To visualize the results for all benchmark problems under consideration, we opted to use the Okabe–Ito palette. The provided figures show the solid (green) line depicting the exact analytical solution, whereas the dashed (orange) line represents the corresponding numerical approximation. The approximation errors $E_{1,k}$ and $E_{2,k}$ are illustrated using blue and orange lines, respectively, with the temporal grid nodes indicated by circular and square markers.
	
	We now turn our attention to the benchmark problems.
	
	\testsubsection		% produces: "Test 1"
	Consider the case in which the exact analytical solutions are given by the following functions:
	\begin{equation*}
		u\brac{x, t} = \sin\brac{\frac{\pi}{2} t} \sin\brac{\frac{\lambda \pi}{\ell} x}\,, \quad \text{and} \quad v\brac{x, t} = \sin\brac{\frac{\pi}{2} t} \sin\brac{\frac{\lambda \pi}{\ell} x}\,.
	\end{equation*}
	
	In this setting, the temporal interval is defined by $0 \leq t \leq 1$, and the oscillation parameter is taken as $\lambda = 14$. Moreover, in this benchmark problem, the temporal interval is uniformly partitioned into $n = 256$ subintervals.
	
	\begin{figure}
		\centering
		\begin{subfigure}{.49\textwidth}
			\centering
			\includegraphics[width=\textwidth,height=\textheight,keepaspectratio]{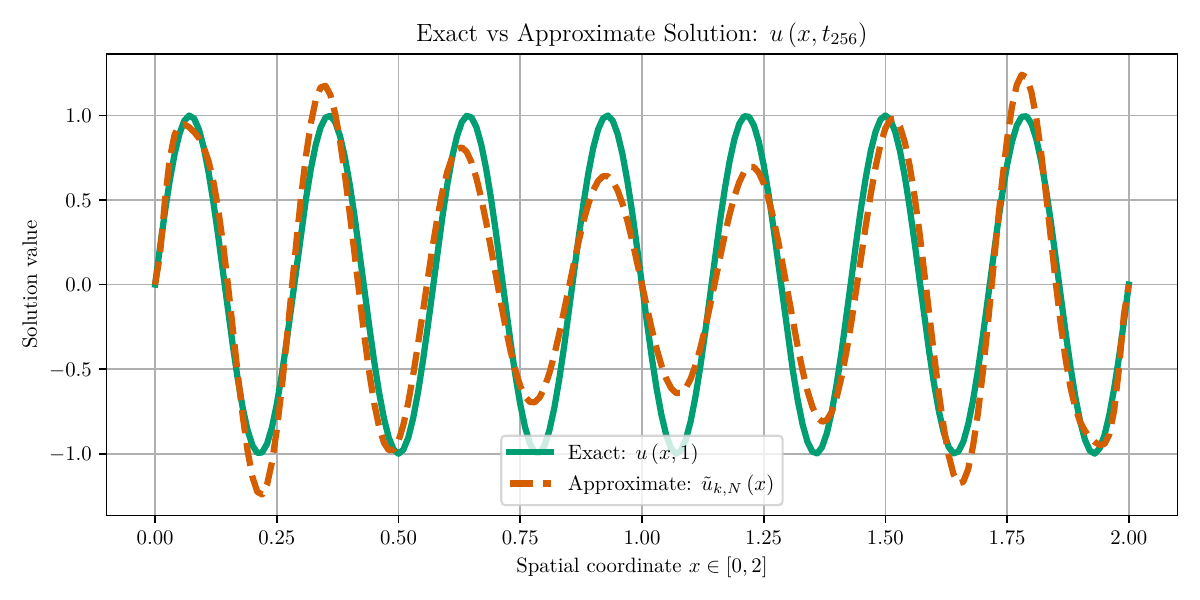}
			\caption{Corresponds to the solution $u\brac{x, t}$.}
			\label{fig:test1_u_case1}
		\end{subfigure}
		\hfill
		\begin{subfigure}{.49\textwidth}
			\centering
			\includegraphics[width=\textwidth,height=\textheight,keepaspectratio]{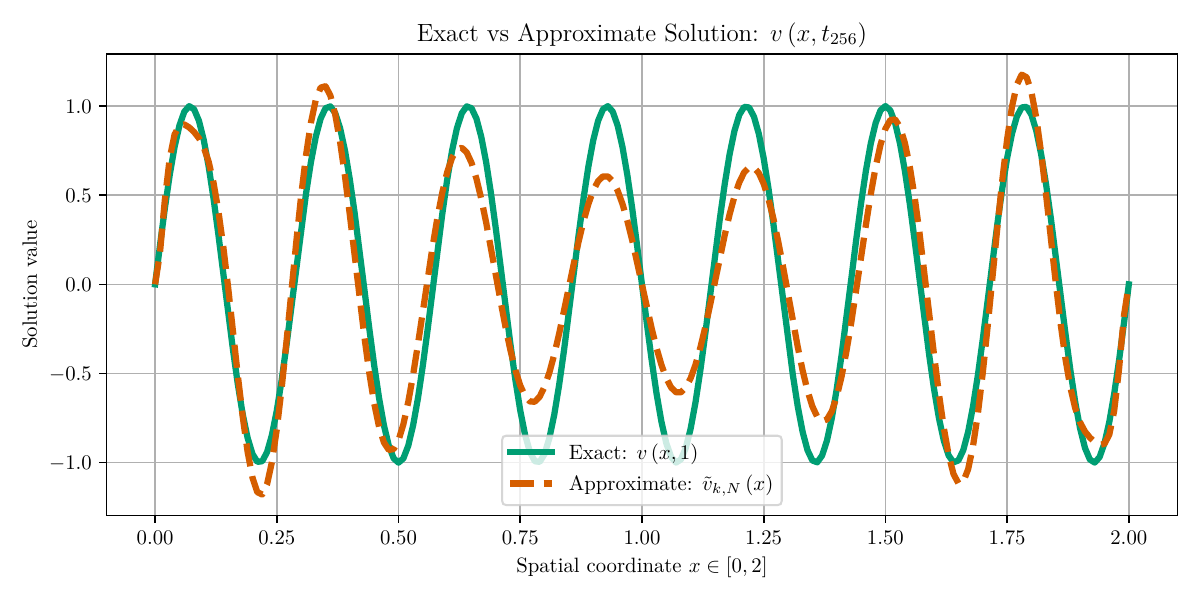}
			\caption{Corresponds to the solution $v\brac{x, t}$.}
			\label{fig:test1_v_case1}
		\end{subfigure}
		\caption{Comparison between the exact solutions and their numerical approximations at the final time layer, where the number of trial functions is specified as $N = 20$.}
		\label{fig:test1_both_solns}
	\end{figure}
	
	Observe that, in this scenario, the accuracy attained with respect to the time variable is approximately $10^{-5}$. Owing to the large number of spatial oscillations, it may be presumed that using $N = 20$ trial functions is insufficient to yield a highly accurate approximation (see \figref{fig:test1_both_solns}). Nevertheless, it is important that the combined scheme maintains and replicates the qualitative structure of the exact solution. Thus, it is reasonable to expect that, upon increasing $N$, the approximate solution should coincide with the exact solution with high accuracy.
	
	Consider the case in which the previously used number of basis functions is increased by $15$, resulting in $N = 35$.
	
	\begin{figure}
		\centering
		\begin{subfigure}{.49\textwidth}
			\centering
			\includegraphics[width=\textwidth,height=\textheight,keepaspectratio]{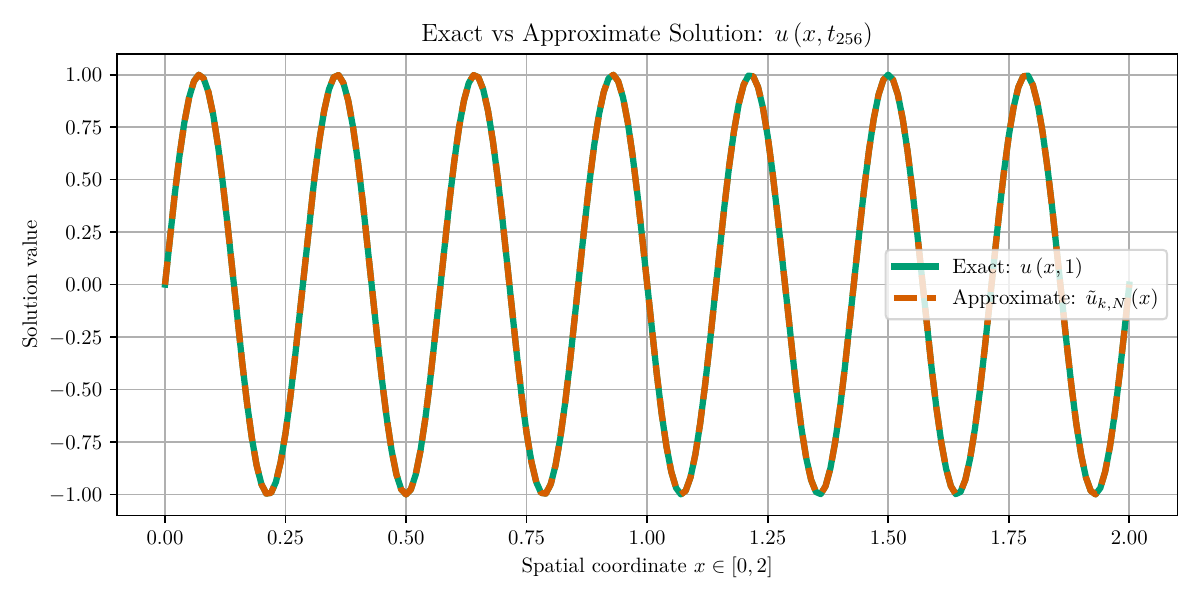}
			\caption{Solution $u$: the exact solution and its approximation at the final time layer.}
			\label{fig:test1_u_case2}
		\end{subfigure}
		\hfill
		\begin{subfigure}{.49\textwidth}
			\centering
			\includegraphics[width=\textwidth,height=\textheight,keepaspectratio]{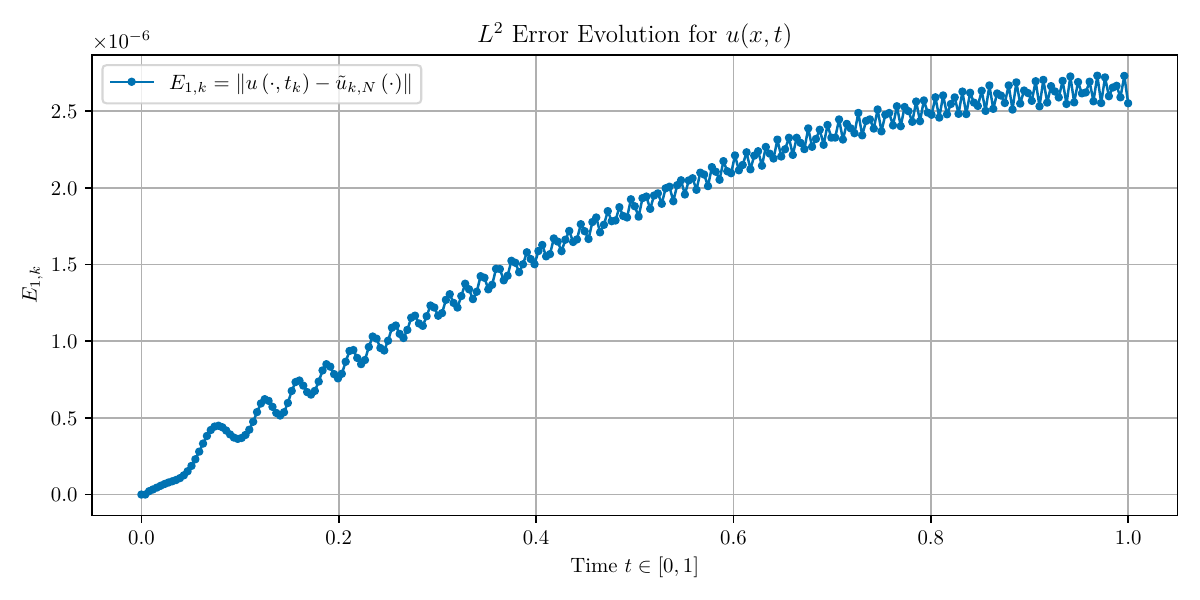}
			\caption{Temporal evolution of the error $E_{1,k}$ over all time layers.}
			\label{fig:error_test1_u_case2}
		\end{subfigure}
		\hfill
		\begin{subfigure}{.49\textwidth}
			\centering
			\includegraphics[width=\textwidth,height=\textheight,keepaspectratio]{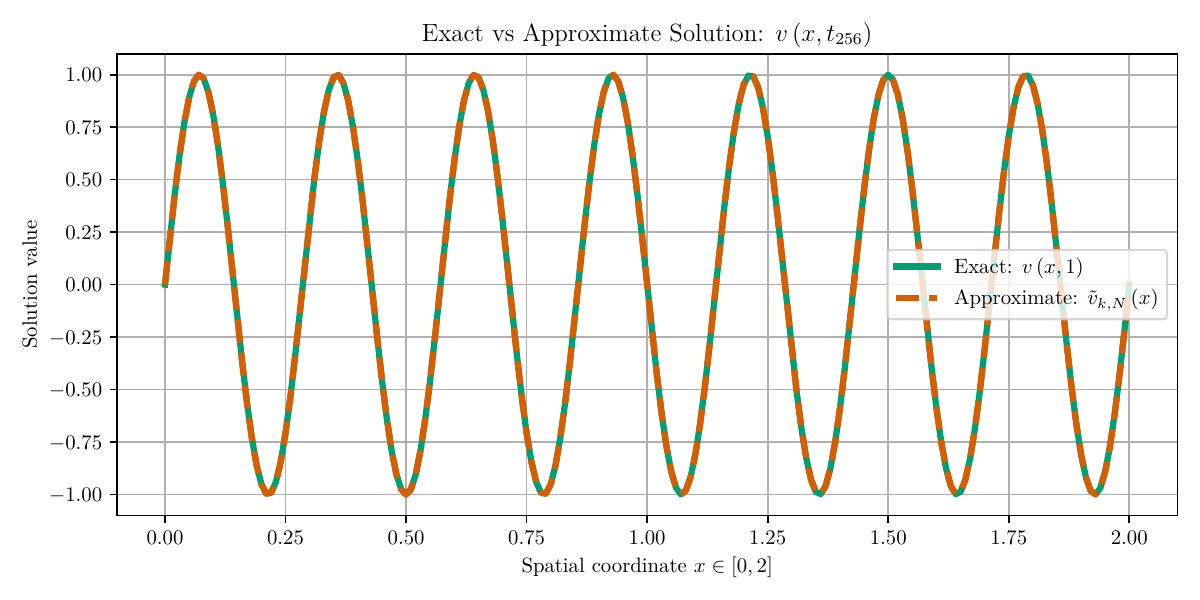}
			\caption{Solution $v$: the exact solution and its approximation at the final time layer.}
			\label{fig:test1_v_case2}
		\end{subfigure}
		\hfill
		\begin{subfigure}{.49\textwidth}
			\centering
			\includegraphics[width=\textwidth,height=\textheight,keepaspectratio]{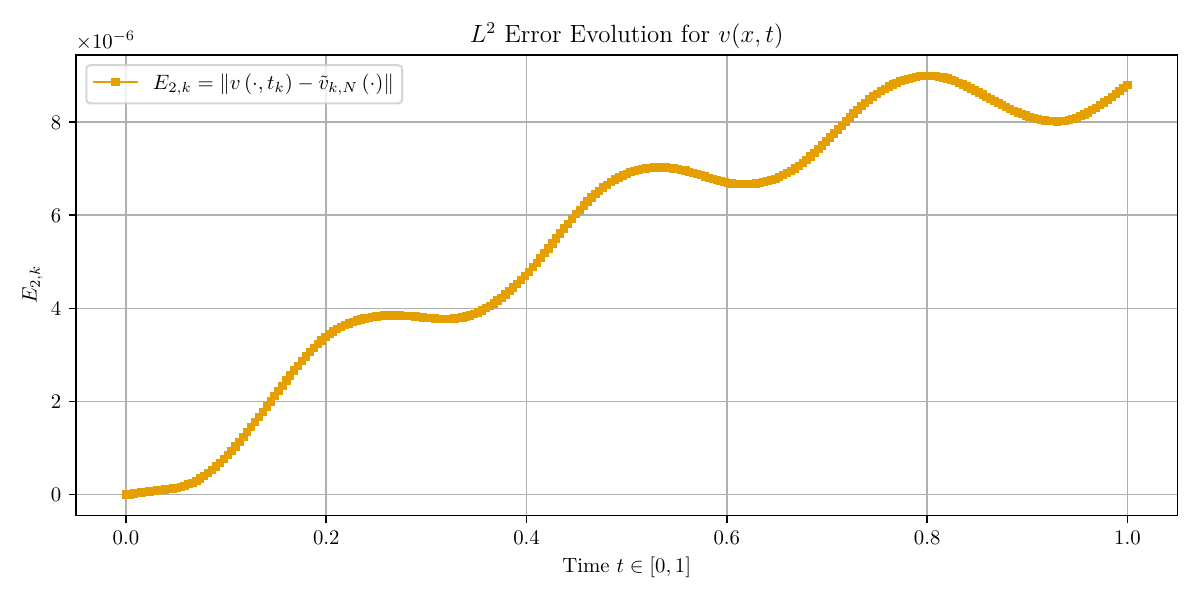}
			\caption{Temporal evolution of the error $E_{2,k}$ over all time layers.}
			\label{fig:error_test1_v_case2}
		\end{subfigure}
		\caption{Comparison of the exact solutions and their numerical approximations at the final temporal layer using $N = 35$ approximation basis functions, along with the evolution of the associated errors across all time layers.}
		\label{fig:test1_all_solns_with_error}
	\end{figure}
	
	As we expected, by increasing the number of basis functions $N$, the combined numerical scheme captures the oscillatory solution quite well, as illustrated in \figref{fig:test1_all_solns_with_error}. Moreover, \figsubref{fig:test1_all_solns_with_error}{fig:error_test1_u_case2} and \figsubref{fig:test1_all_solns_with_error}{fig:error_test1_v_case2} indicate that $\max_{0 \leq k \leq 256} E_{1,k}$ and $\max_{0 \leq k \leq 256} E_{2,k}$ are order of $10^{-6}$.
	
	\testsubsection		% produces: "Test 2"
	We now consider the setting in which the exact solutions assume the following form:
	\begin{equation*}
		u\brac{x, t} = v\brac{x, t} = A \brac{1 + \cos\brac{\frac{\lambda_1 \pi}{T} t}} \exp\brac{-\frac{\brac{2x - \ell}^2}{c^2}} \sin\brac{\frac{\lambda \pi}{\ell} x}\,.
	\end{equation*}
	
	As in the preceding benchmark problem, the temporal variable is specified on the interval $t \in \qbrac{0, 1}$, which is uniformly divided into $n = 256$ subintervals. In contrast to \testref{1}, the amplitude of the sine function now varies with both the spatial and temporal variables. For the numerical computations, the parameters appearing in the given solutions are prescribed as $A = 0.5$, $\lambda_1 = 2$, $c = 1$, and $\lambda = 19$.
	
	\begin{figure}
		\centering
		\begin{subfigure}{.49\textwidth}
			\centering
			\includegraphics[width=\textwidth,height=\textheight,keepaspectratio]{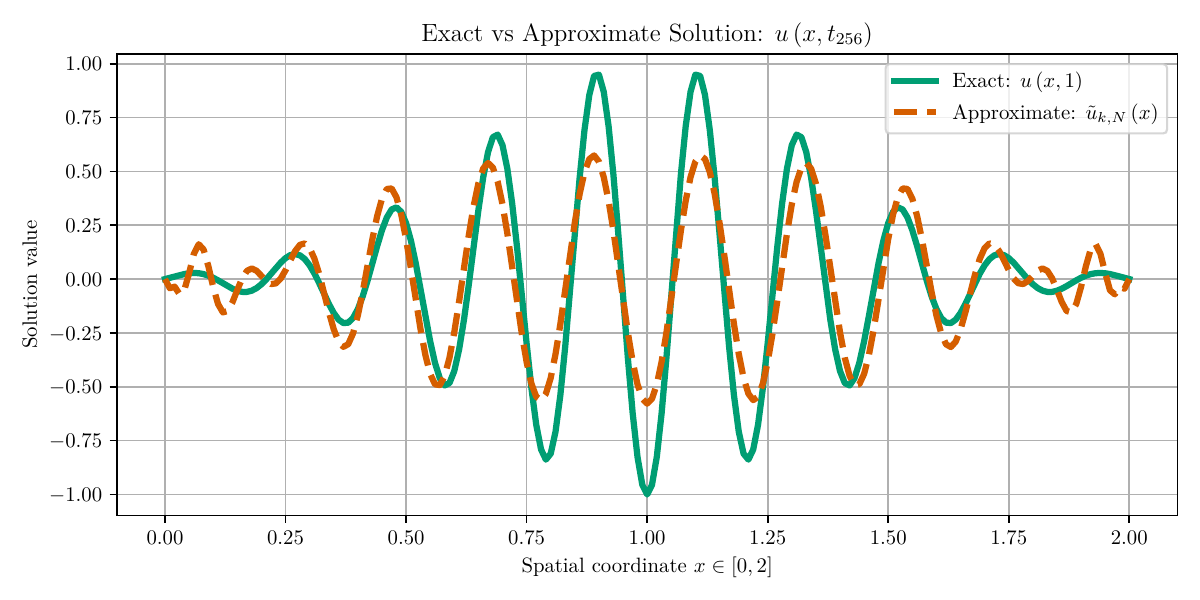}
			\caption{Depicts the solution $u\brac{x, t}$.}
			\label{fig:test2_u_case1}
		\end{subfigure}
		\hfill
		\begin{subfigure}{.49\textwidth}
			\centering
			\includegraphics[width=\textwidth,height=\textheight,keepaspectratio]{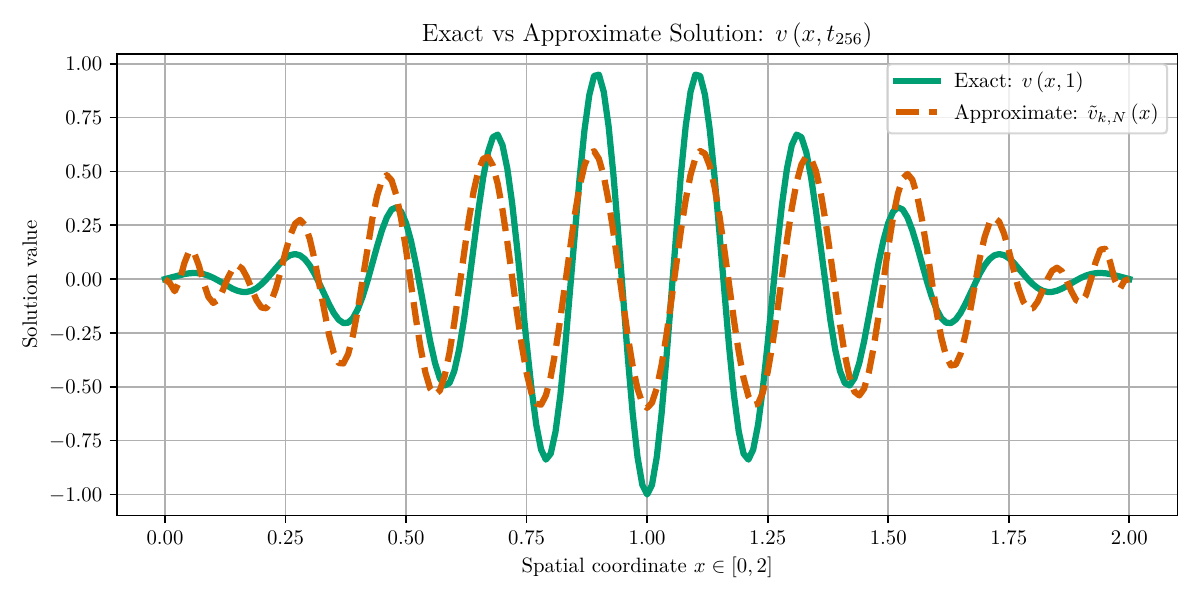}
			\caption{Depicts the solution $v\brac{x, t}$.}
			\label{fig:test2_v_case1}
		\end{subfigure}
		\caption{Exact solutions and their numerical approximations at the final time instant $t = 1$, computed with $N = 29$ trial functions.}
		\label{fig:test2_both_solns}
	\end{figure}
	
	Because the oscillation parameter $\lambda = 19$ is relatively large, employing $N = 29$ trial functions does not yield an approximation of high accuracy. On the other hand, it is essential that the numerical solution faithfully replicates the qualitative features of the exact solution (see \figref{fig:test2_both_solns}). The adopted temporal discretization (namely, the grid length $\tau = 2^{-8}$) is fully adequate for approximating the solution with respect to the temporal variable.
	
	\begin{figure}
		\centering
		\begin{subfigure}{.49\textwidth}
			\centering
			\includegraphics[width=\textwidth,height=\textheight,keepaspectratio]{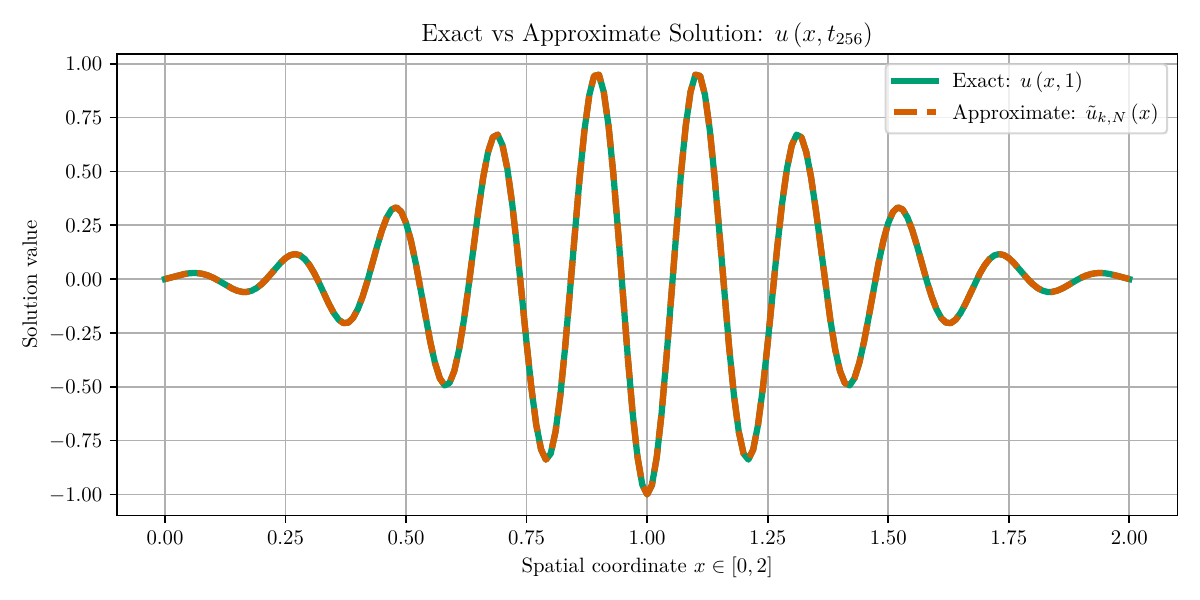}
			\caption{Solution $u$: comparison of the exact solution with its numerical approximation at the final time layer.}
			\label{fig:test2_u_case2}
		\end{subfigure}
		\hfill
		\begin{subfigure}{.49\textwidth}
			\centering
			\includegraphics[width=\textwidth,height=\textheight,keepaspectratio]{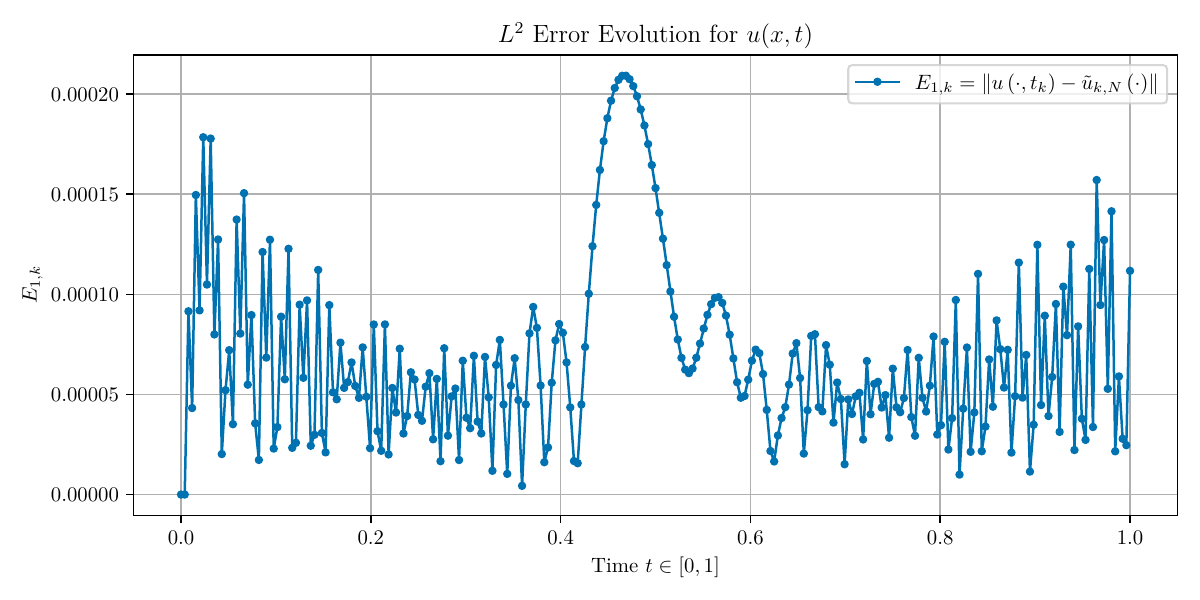}
			\caption{Temporal evolution of the error $E_{1,k}$ throughout all time layers.}
			\label{fig:error_test2_u_case2}
		\end{subfigure}
		\hfill
		\begin{subfigure}{.49\textwidth}
			\centering
			\includegraphics[width=\textwidth,height=\textheight,keepaspectratio]{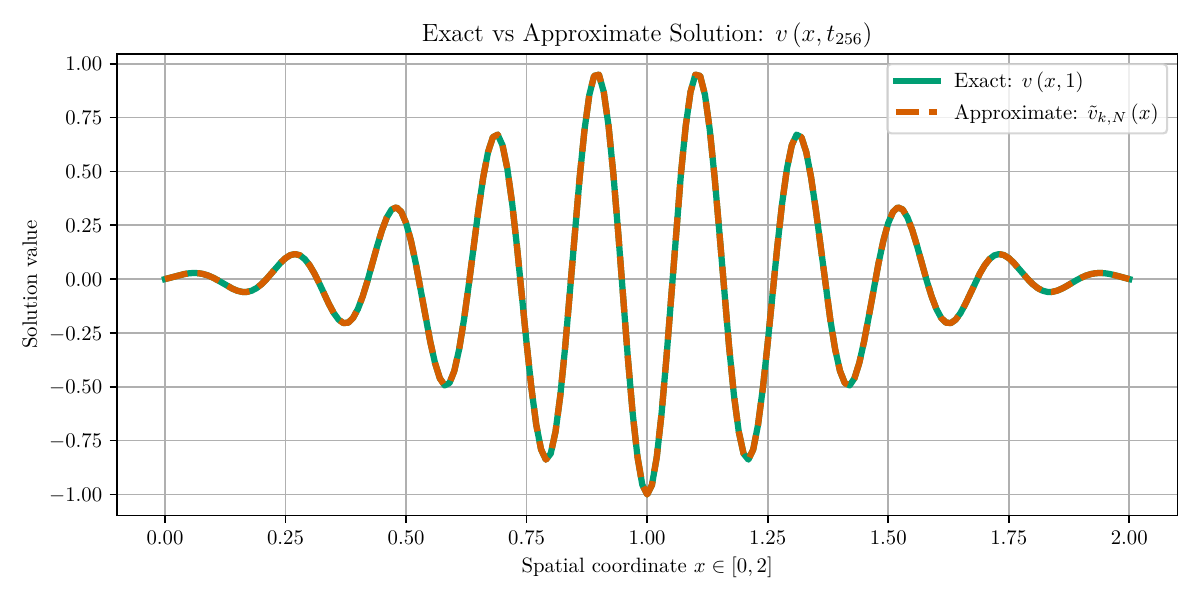}
			\caption{Solution $v$: comparison of the exact solution with its numerical approximation at the final time layer.}
			\label{fig:test2_v_case2}
		\end{subfigure}
		\hfill
		\begin{subfigure}{.49\textwidth}
			\centering
			\includegraphics[width=\textwidth,height=\textheight,keepaspectratio]{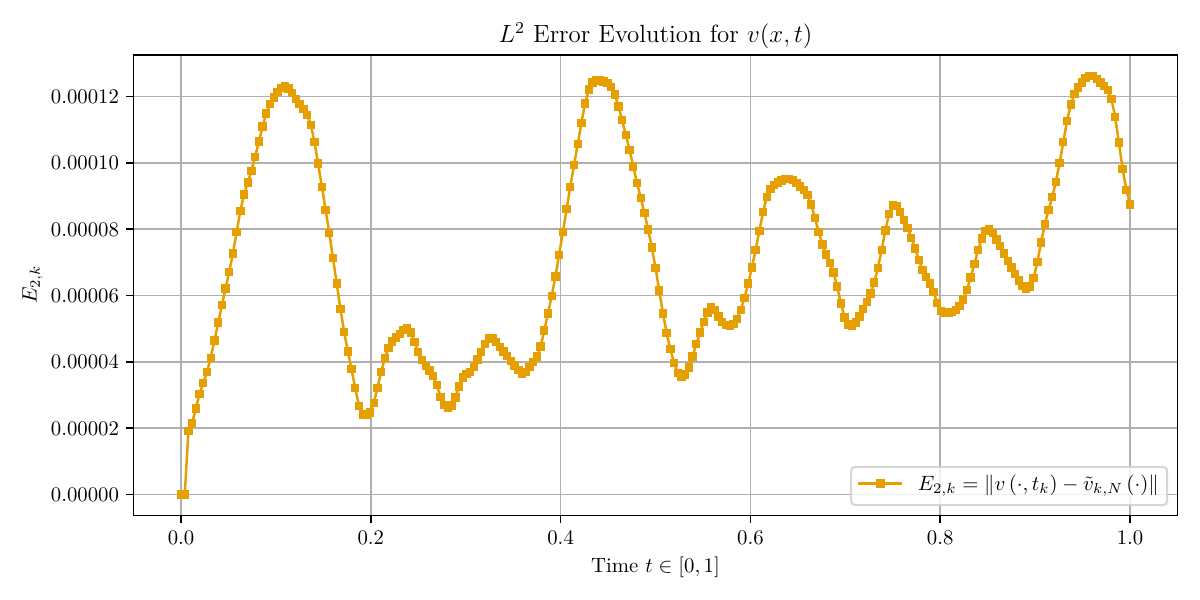}
			\caption{Temporal evolution of the error $E_{2,k}$ throughout all time layers.}
			\label{fig:error_test2_v_case2}
		\end{subfigure}
		\caption{Exact and numerical solutions at the final temporal layer for $N = 45$ trial functions, and evolution of the errors $E_{1,k}$ and $E_{2,k}$ across all time layers, illustrating the accuracy of the combined numerical scheme.}
		\label{fig:test2_all_solns_with_error}
	\end{figure}
	
	In \figref{fig:test2_all_solns_with_error}, we observe that increasing the number of approximation basis functions to $N = 45$ is sufficient to obtain numerical solutions that coincide with the exact analytical solutions with good accuracy. Furthermore, \figsubref{fig:test2_all_solns_with_error}{fig:error_test2_u_case2} and \figsubref{fig:test2_all_solns_with_error}{fig:error_test2_v_case2} demonstrate that the maximum values of the corresponding errors are approximately of the order of $10^{-4}$.
	
	\testsubsection		% produces: "Test 3"
	Let the exact solutions be given by the following functions:
	\begin{equation*}
		u\brac{x, t} = \frac{1}{4} \exp\brac{\frac{\pi}{T} t} \sin\brac{\frac{\lambda \pi}{\ell} x}\,, \quad \text{and} \quad v\brac{x, t} = \frac{1}{4} \exp\brac{\frac{\pi}{T} t} \sin\brac{\frac{\lambda \pi}{\ell} x}\,.
	\end{equation*}
	
	For this benchmark problem, unlike \testref{1} and \testref{2}, we consider a wider time frame, specifically the temporal interval satisfying $0 \leq t \leq 4$. The oscillation parameter appearing in the sine functions is prescribed as $\lambda = 5$. Initially, the uniform temporal grid spacing is chosen as $\tau = 2^{-7}$, while the number of basis functions is set to $N = 7$. The numerical results obtained under this configuration are shown in \figref{fig:test3_both_solns}.
	
	\begin{figure}
		\centering
		\begin{subfigure}{.49\textwidth}
			\centering
			\includegraphics[width=\textwidth,height=\textheight,keepaspectratio]{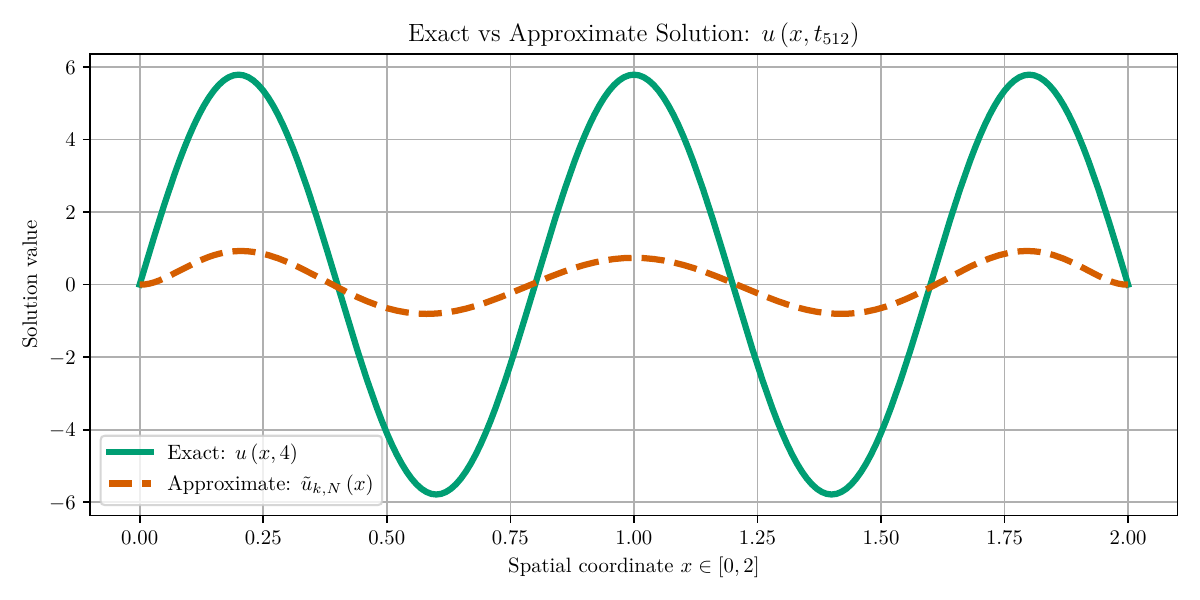}
			\caption{Illustration of the solution $u\brac{x, t}$.}
			\label{fig:test3_u_case1}
		\end{subfigure}
		\hfill
		\begin{subfigure}{.49\textwidth}
			\centering
			\includegraphics[width=\textwidth,height=\textheight,keepaspectratio]{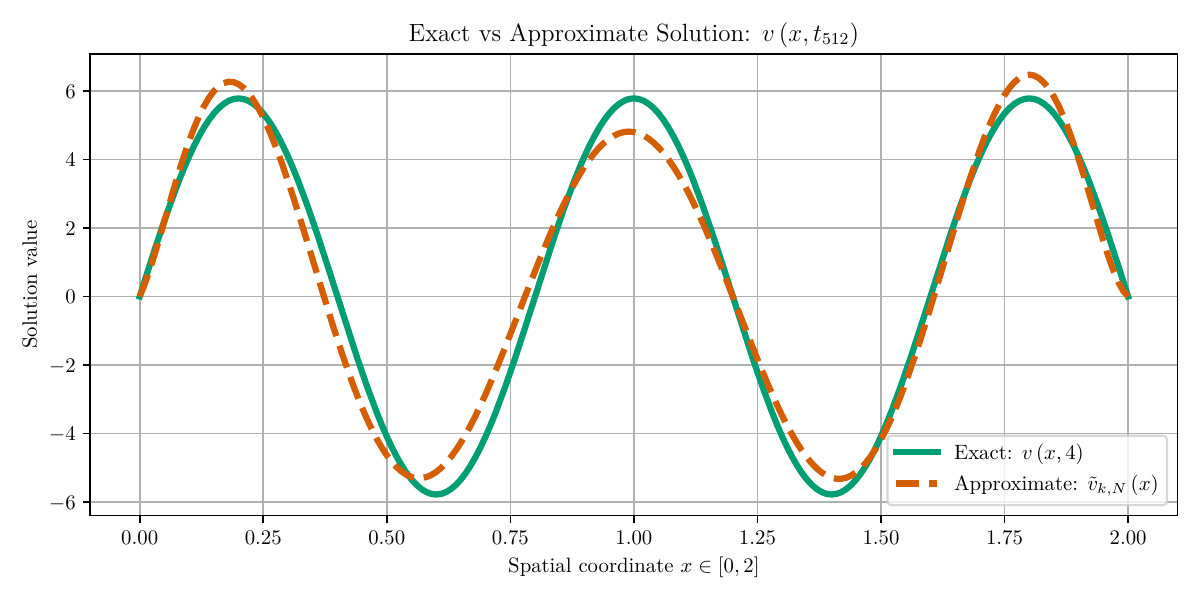}
			\caption{Illustration of the solution $v\brac{x, t}$.}
			\label{fig:test3_v_case1}
		\end{subfigure}
		\caption{Exact solutions together with their corresponding numerical approximations at the final time instant $t = 4$, obtained using $N = 7$ trial functions.}
		\label{fig:test3_both_solns}
	\end{figure}
	
	As illustrated in \figref{fig:test3_both_solns}, the chosen configuration is inappropriate for accurately approximating the exact solutions using the proposed combined numerical scheme. Although the oscillation parameter is not particularly large, the time-dependent amplitude of the sine function grows rapidly. Given these facts, it is reasonable to refine both the temporal grid spacing and the number of basis functions, thereby setting $\tau = 2^{-8}$ and $N = 15$, respectively.
	
	\begin{figure}
		\centering
		\begin{subfigure}{.49\textwidth}
			\centering
			\includegraphics[width=\textwidth,height=\textheight,keepaspectratio]{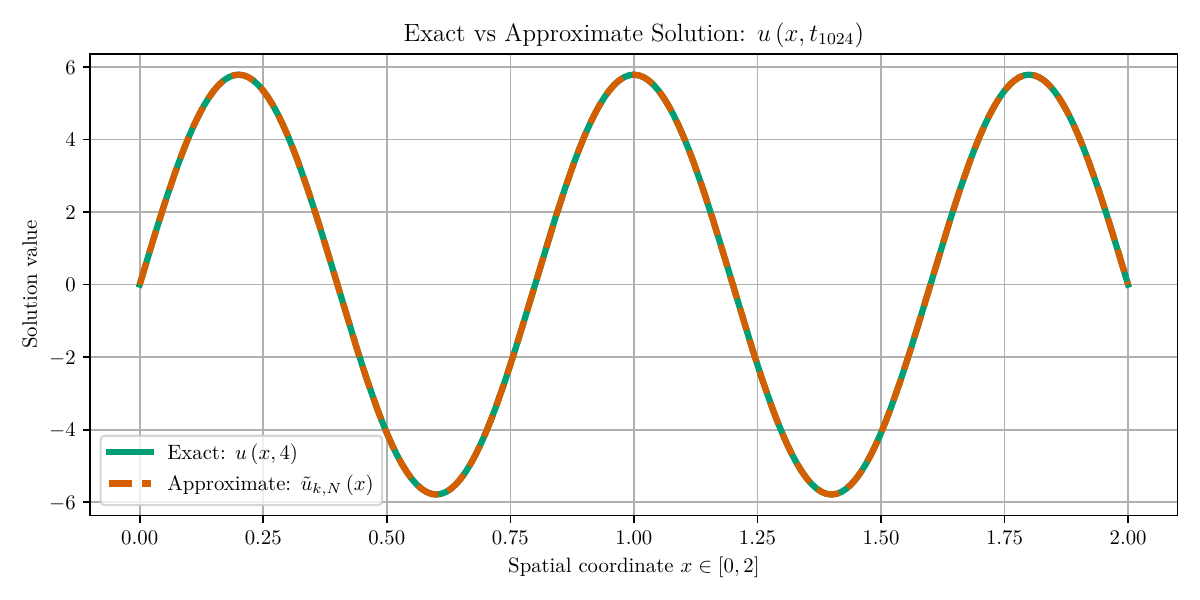}
			\caption{Exact and numerical solutions for $u\brac{x, t}$ at the final time layer $t = 4$.}
			\label{fig:test3_u_case2}
		\end{subfigure}
		\hfill
		\begin{subfigure}{.49\textwidth}
			\centering
			\includegraphics[width=\textwidth,height=\textheight,keepaspectratio]{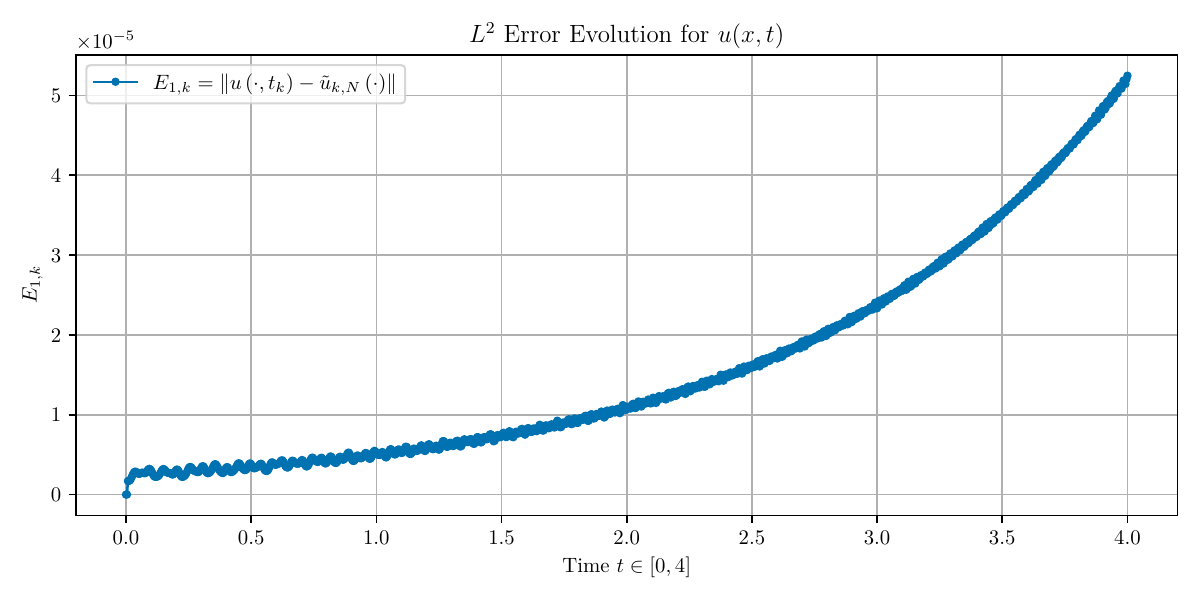}
			\caption{Evolution of the error $E_{1,k}$ over the full sequence of time layers.}
			\label{fig:error_test3_u_case2}
		\end{subfigure}
		\hfill
		\begin{subfigure}{.49\textwidth}
			\centering
			\includegraphics[width=\textwidth,height=\textheight,keepaspectratio]{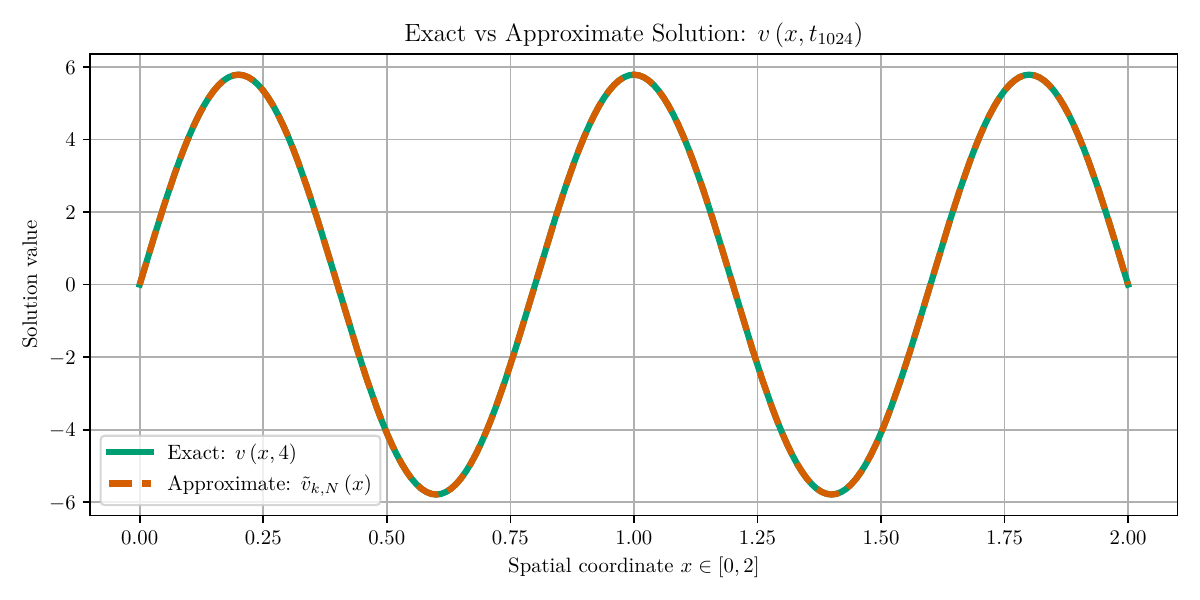}
			\caption{Exact and numerical solutions for $v\brac{x, t}$ at the final time layer $t = 4$.}
			\label{fig:test3_v_case2}
		\end{subfigure}
		\hfill
		\begin{subfigure}{.49\textwidth}
			\centering
			\includegraphics[width=\textwidth,height=\textheight,keepaspectratio]{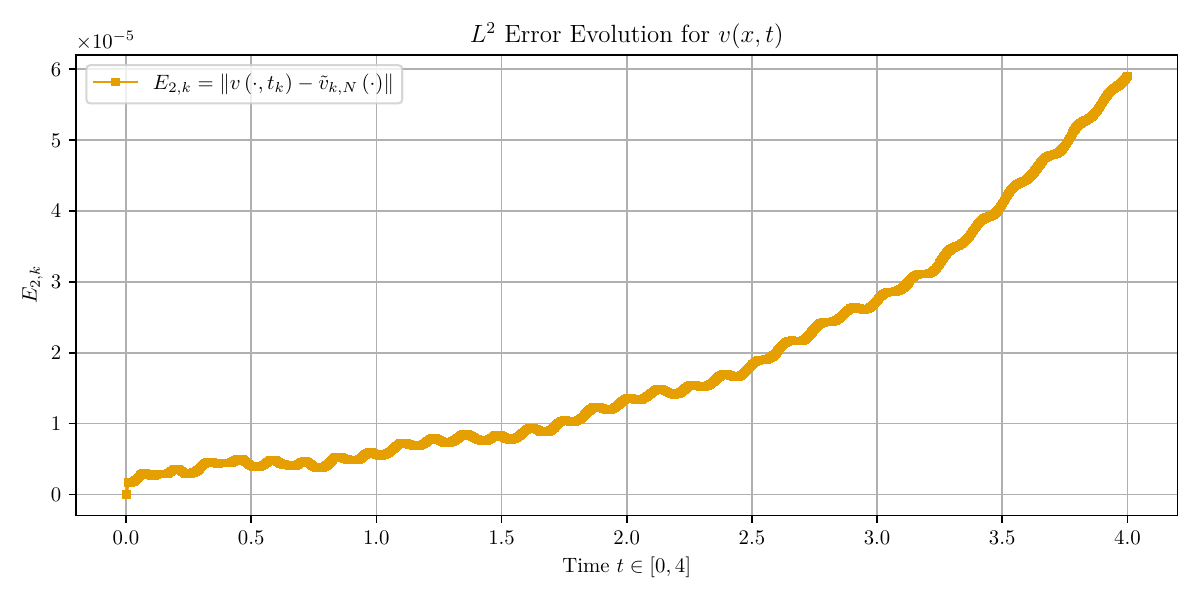}
			\caption{Evolution of the error $E_{2,k}$ over the full sequence of time layers.}
			\label{fig:error_test3_v_case2}
		\end{subfigure}
		\caption{Exact and numerical solutions at the final temporal layer $k = 1024$ for $N = 15$ trial functions, supplemented by the evolution of the errors $E_{1,k}$ and $E_{2,k}$ over all time layers, thereby demonstrating the accuracy of the combined numerical scheme.}
		\label{fig:test3_all_solns_with_error}
	\end{figure}
	
	As depicted in \figref{fig:test3_all_solns_with_error}, improving both the temporal grid length to $\tau = 2^{-8}$ and the number of basis functions to $N = 15$ proved successful, resulting in an evident similarity between the exact and approximate solutions displayed in \figsubref{fig:test3_all_solns_with_error}{fig:test3_u_case2} and \figsubref{fig:test3_all_solns_with_error}{fig:test3_v_case2}. It may be observed that the approximation errors associated with each solution are small; refer to \figsubref{fig:test3_all_solns_with_error}{fig:error_test3_u_case2} and \figsubref{fig:test3_all_solns_with_error}{fig:error_test3_v_case2}.
	
	\begin{remark*}
		It is clear that if the solution of problem \eqref{eq:timosh_spec_nonlinear}-\eqref{eq:timosh_spec_bound_conds} is the product of a polynomial in the spatial variable $x$ and a linear function in the temporal variable $t$, then the theoretical solution derived by the proposed combined scheme coincides precisely with the exact analytical solution. In this situation, the error of the numerically computed approximation is of the order of machine precision.
	\end{remark*}
	
	\testsubsection\label{bm:test4}		% produces: "Test 4"
	We consider the homogeneous reformulation of the system \eqref{eq:timosh_spec_nonlinear}-\eqref{eq:timosh_spec_linear}, endowed with the initial-boundary conditions \eqref{eq:timosh_spec_init_conds}-\eqref{eq:timosh_spec_bound_conds}, for which no exact analytical solutions are currently known. In this formulation, the spatial and temporal domains are taken to be $x \in \qbrac{0, 2}$ and $t \in \qbrac{0, 1}$, respectively.
	
	The initial data in \eqref{eq:timosh_spec_init_conds}-\eqref{eq:timosh_spec_bound_conds} are specified by the following functions:
	\begin{equation*}
		\varphi_0\brac{x} = \psi_0\brac{x} = A \exp\brac{-\frac{\brac{2x - \ell}^2}{c^2}} \sin\brac{\frac{\lambda \pi}{\ell} x} \quad \text{and} \quad \varphi_1\brac{x} = \psi_1\brac{x} = 0\,.
	\end{equation*}
	In this setup, the prefactor multiplying the sine term is a Gaussian function. For the numerical experiments presented below, the parameters of the Gaussian are chosen as follows: the amplitude is set to $A = 1$, the shift parameter to $\ell = 2$, and the width (shape) parameter to $c = 0.5$. The oscillation parameter appearing in the sine function is fixed at $\lambda = 10$. For this configuration, the graph of the function $\varphi_0\brac{x}$ is depicted in \figref{fig:test3_initial_datum_varphi}.
	
	\begin{figure}
		\centering
		\includegraphics[width=0.55\linewidth]{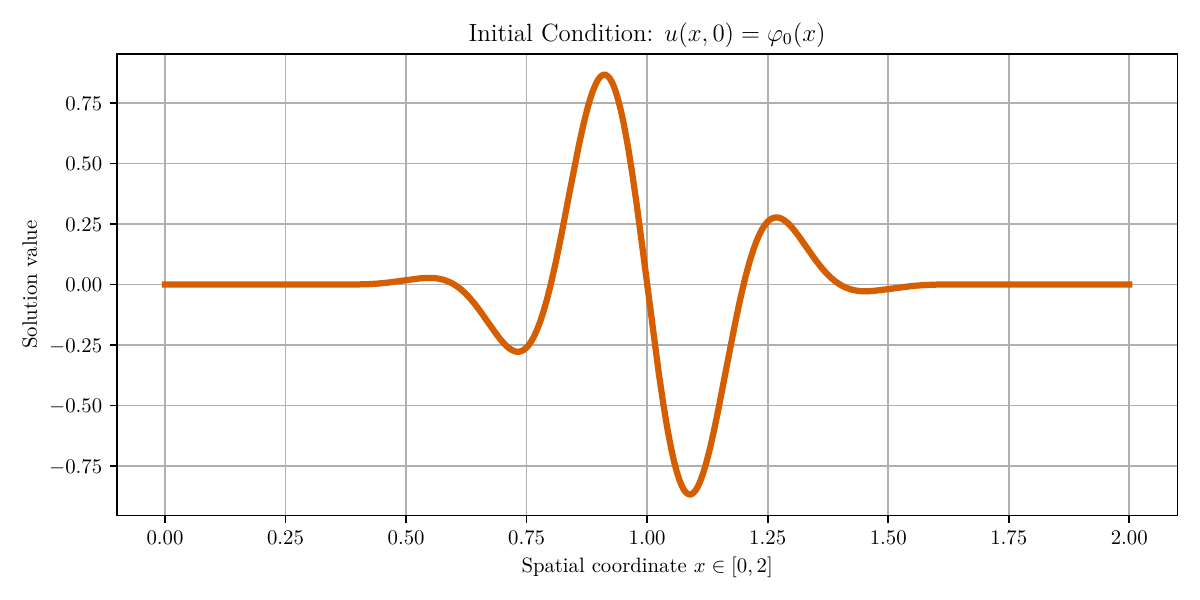}
		\caption{Graph of the initial-condition function $\varphi_0\brac{x} = \psi_0\brac{x}$}
		\label{fig:test3_initial_datum_varphi}
	\end{figure}
	
	To describe the approximate solution procedure for the problem under consideration, we introduce the notation required for the subsequent discussion. Let $n_0$ denote the initial number of subdivisions of the temporal interval $\qbrac{0,T}$, and let $N_0$ denote the initial number of basis functions. At each refinement step, the number of temporal subdivisions is doubled, while at each substep the number of basis functions is increased by one. Thus, if $n_i$ denotes the number of temporal subdivisions at the $i$-th refinement step and $N_j$ denotes the number of basis functions at the $j$-th substep, then:
	\begin{gather*}
		n_i = 2^i n_0\,,\quad t_k^{\brac{i}} = k \tau_i\,,\quad \tau_i = \frac{T}{n_i}\,,\quad \text{for}\quad i = 0,1,\ldots,\hat{i}\,. \\
		N_j = N_0 + j\,,\quad \text{for}\quad j = 0,1,\ldots,\hat{j}\,.
	\end{gather*}
	
	Let $\widetilde{u}_{k,N_j}^{\brac{i}} \brac{x}$ and $\widetilde{v}_{k,N_j}^{\brac{i}} \brac{x}$, for $k = 2,3,\ldots,n_i$, denote the solution corresponding to problem \eqref{eq:timosh_spec_nonlinear}-\eqref{eq:timosh_spec_bound_conds}. Then, according to relation \eqref{eq:galerkin_ansatzes}, we have:
	\begin{equation*}
		u \brac{x, t_k^{\brac{i}}} \approx \widetilde{u}_{k,N_j}^{\brac{i}} \brac{x} = \sum_{m = 1}^{N_j} u_{k,m}^{\brac{i,j}} \phi_m \brac{x}\,,\quad v \brac{x, t_k^{\brac{i}}} \approx \widetilde{v}_{k,N_j}^{\brac{i}} \brac{x} = \sum_{m = 1}^{N_j} v_{k,m}^{\brac{i,j}} \phi_m \brac{x}\,,
	\end{equation*}
	where the coefficients $u_{k,m}^{\brac{i,j}}$ and $v_{k,m}^{\brac{i,j}}$ are obtained by solving the associated Galerkin linear system (see \refitem{Subsection}{subsec:galerkin}) at refinement step $i$ and substep $j$.
	
	The underlying idea of the proposed procedure is as follows. For each fixed value of $n_i$, the computation in the spatial variable is continued (that is, the substep index $j$ is incremented by one) until convergence is achieved at each temporal level $k$, in the sense that the $L^2$-norms of the differences between two consecutive approximate solutions corresponding to $u$ and $v$ both become less than or equal to a prescribed tolerance $\mathrm{tol} > 0$. More precisely, the process is terminated once the conditions
	\begin{equation}\label{eq:num_res_norm_diffs}
		E_{k,j}^{\brac{i}} = \norm{\widetilde{u}_{k,N_j}^{\brac{i}} - \widetilde{u}_{k,N_{j - 1}}^{\brac{i}}} \leq \mathrm{tol} \quad \text{and} \quad \hat{E}_{k,j}^{\brac{i}} = \norm{\widetilde{v}_{k,N_j}^{\brac{i}} - \widetilde{v}_{k,N_{j - 1}}^{\brac{i}}} \leq \mathrm{tol}
	\end{equation}
	are simultaneously satisfied for each temporal layer $k = 2,3,\ldots,n_i$.
	
	If the convergence criterion fails to be satisfied during the spatial iteration, namely upon incrementing the iteration index $j$ by one, then the temporal discretization is refined by doubling the number of time-step subdivisions, and the procedure is restarted from the initial stage. The indices $\hat{i}$ and $\hat{j}$ are fixed positive integers, typically chosen in accordance with the constraints imposed by available computational resources, feasible runtime bounds, and the effective stability margins of the proposed combined scheme.
	
	\begin{remark*}
		The $L^2$-norms of the differences between two successive approximate solutions defined in \eqref{eq:num_res_norm_diffs} are computed by the following formulas:
		\begin{equation*}
			\norm{\widetilde{u}_{k,N_j}^{\brac{i}} - \widetilde{u}_{k,N_{j - 1}}^{\brac{i}}} = \frac{\ell}{2} \sqrt{\vm{e}_k^{\top} \vm{\H}_{N_j} \vm{e}_k}\,,\quad \norm{\widetilde{v}_{k,N_j}^{\brac{i}} - \widetilde{v}_{k,N_{j - 1}}^{\brac{i}}} = \frac{\ell}{2} \sqrt{\vm{\hat{e}}_k^{\top} \vm{\H}_{N_j} \vm{\hat{e}}_k}\,.
		\end{equation*}
		In these expressions, $\vm{\H}_{N_j}$ denotes the matrix introduced in \refitem{Subsec.}{subsec:galerkin}. Recall that $\vm{\H}_{N_j}$ is a symmetric tridiagonal matrix with a gap: its nonzero entries are $C_m$, $m = 1,2,\ldots,N_j$, on the main diagonal, and $-B_m$, $m = 2,3,\ldots,N_j - 1$, on the second off-diagonals. In addition, the components of the vector $\vm{e}_k = \tps{\brac{e_{k,1},e_{k,2},\ldots,e_{k,N_j}}}$ are given by
		\begin{equation*}
			e_{k,m} =
			\left\{
			\begin{array}{rl}
				u_{k,m}^{\brac{i,j}} - u_{k,m}^{\brac{i,j - 1}}\,, & \hbox{if $m \leq N_{j - 1}$,} \\
				u_{k,m}^{\brac{i,j}}\,, & \hbox{if $m > N_{j - 1}$.}
			\end{array}
			\right.
		\end{equation*}
		The vector $\vm{\hat{e}}_k$ is defined analogously.
	\end{remark*}
	The preceding formula follows immediately from the standard properties of the Legendre polynomials. Indeed, we have:
	\begin{align*}
		\norm{\widetilde{u}_{k,N_j}^{\brac{i}} - \widetilde{u}_{k,N_{j - 1}}^{\brac{i}}}^2 &= \norm{\sum_{m = 1}^{N_j} e_{k,m} \phi_m}^2 = \sum_{m = 1}^{N_j} e_{k,m} \sum_{s = 1}^{N_j} e_{k,s} \brac{\phi_s, \phi_m} \\
		&= \frac{{\ell}^2}{4} \sum_{m = 1}^{N_j} e_{k,m} \brac{- B_{m - 1} e_{k,m - 2} + C_m e_{k,m} - B_{m + 1} e_{k,m + 2}} = \frac{{\ell}^2}{4} \vm{e}_k^{\top} \vm{\H}_{N_j} \vm{e}_k\,.
	\end{align*}
	
	In the numerical experiments, the computational algorithm is initialized with $n_0 = 2048$ time steps and $N_0 = 41$ spatial (Galerkin) modes. Time refinement is controlled by the admissible exponent $\hat{i} = 1$ (yielding a maximal refinement level of $n_1 = 4096$), while spatial refinement is restricted to a maximum index $\hat{j} = 4$ (corresponding to an upper limit of $N_4 = 45$). The stopping criterion of the algorithm is based on the $L^2$-norm difference, with the tolerance set to $\mathrm{tol} = 10^{-4}$.
	
	\begin{figure}
		\centering
		\begin{subfigure}{.49\textwidth}
			\centering
			\includegraphics[width=\textwidth,height=\textheight,keepaspectratio]{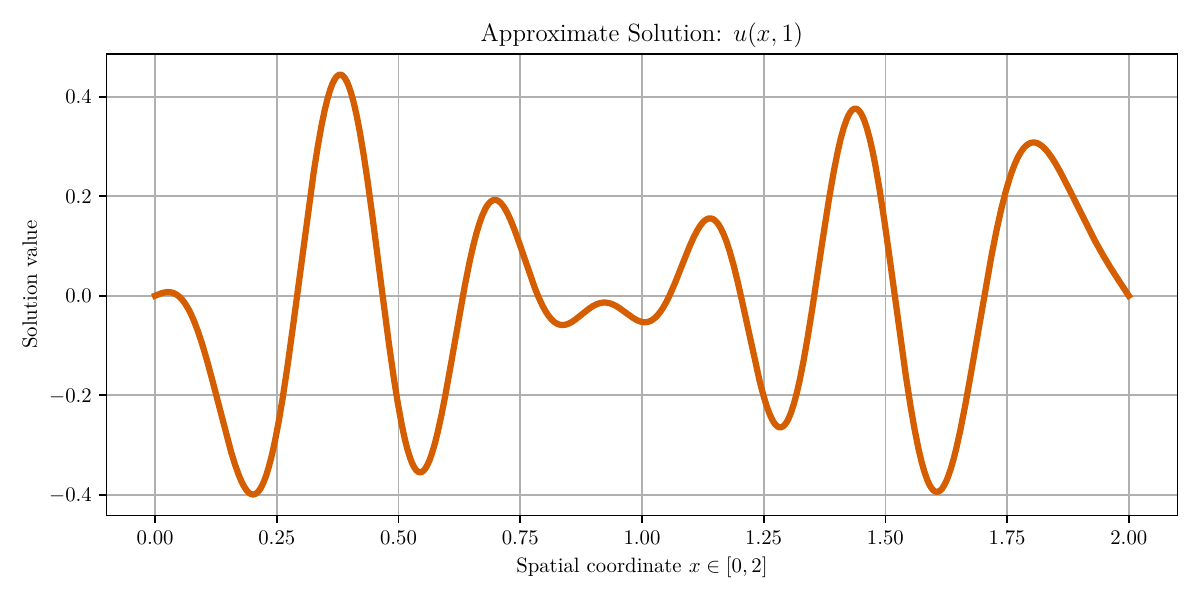}
			\caption{$t = 1.00$, $E_{2048,1}^{\brac{0}} \approx 10^{-5}$.}
			\label{fig:test4_u_t2048_N42}
		\end{subfigure}
		\hfill
		\begin{subfigure}{.49\textwidth}
			\centering
			\includegraphics[width=\textwidth,height=\textheight,keepaspectratio]{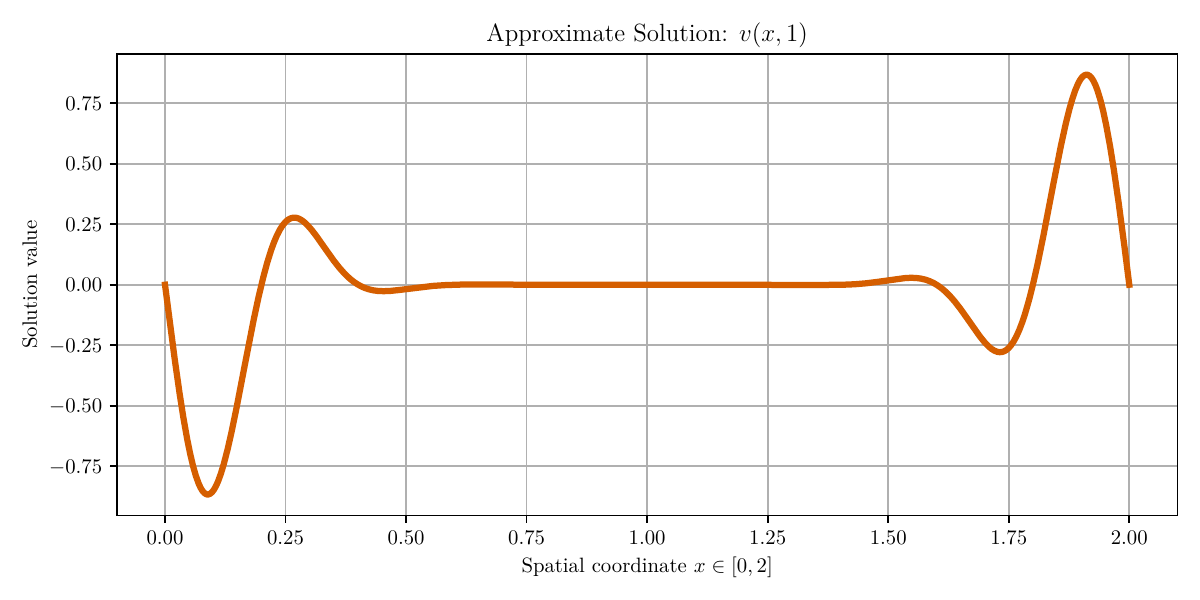}
			\caption{$t = 1.00$, $\hat{E}_{2048,1}^{\brac{0}} \approx 10^{-5}$.}
			\label{fig:test4_v_t2048_N42}
		\end{subfigure}
		\caption{Profiles of the resulting wave-type solutions $u$ and $v$ at the time level $t = 1.00$. The left and right panels display $u \brac{x,1}$ and $v \brac{x,1}$, respectively.}
		\label{fig:test4_uv_solns_layers}
	\end{figure}
	
	For a given tolerance, the computational procedure is terminated when the parameters attain the values $n_0 = 2048$ and $N_1 = 42$, corresponding to the temporal refinement level $i = 0$ and the spatial iteration index $j = 1$, respectively. Numerical experiments conducted for different tolerance values, namely $\mathrm{tol} = 10^{-s}$ with $s = 1,2,3,4$, yield results whose qualitative structures are consistent with those illustrated, for instance, in \figref{fig:test4_uv_solns_layers} for the case $\mathrm{tol} = 10^{-4}$ at the time level $t = 1$. The observed stabilization thus provides strong evidence for the robustness and reliability of the computed numerical solutions.
	
	Note that the functions defining the initial conditions of the problem at hand are symmetric, and the system is considered in the homogeneous setting, {\ie}, in the absence of source terms. Therefore, one should expect the solution of the system to inherit a certain symmetry. This property is indeed observed in the numerical results displayed in \figref{fig:test4_uv_solns_layers}.
	
	Initially, the wave packets are localized near the center of the string; see \figref{fig:test3_initial_datum_varphi}. Then, according to linear theory, the initial profile splits into two equal parts that propagate in opposite directions: one to the left and the other to the right. Thus, at the time level $t = 1$, the two halves of the initial profile have moved away from the center. Therefore, the central part of the string is almost at the equilibrium position; see \figsubref{fig:test4_uv_solns_layers}{fig:test4_v_t2048_N42}. Furthermore, as depicted in \figsubref{fig:test4_uv_solns_layers}{fig:test4_v_t2048_N42}, when these separated parts reach the fixed endpoints, reflection with a change of sign occurs. Since we are solving a coupled system and have shown that the solution $v$ corresponding to the linear equation is physically realistic, we may conclude that the solution $u$, corresponding to the nonlinear equation, is also physically realistic.
	
	\section*{Acknowledgments}
	\noindent The authors sincerely thank Dr. Andreas A. Buchheit and Dr. Daniel Seibel for their careful reading of the initial version of the manuscript and for their insightful comments, which helped improve the paper. We also wish to express our gratitude to Dr. Giorgi Rukhaia for his fruitful remarks during the development of the programming code for the proposed algorithm. Furthermore, the authors would like to thank the anonymous reviewers for their helpful comments, which contributed to improving the final version of the article.
	
	\section*{Funding statement}
	\noindent The second author, Z.V., was supported by the Shota Rustaveli National Science Foundation of Georgia (SRNSFG) under grant number FR-25-215.
	
	\section*{Conflict of interest}
	\noindent The authors declare that they have no conflict of interest.
	
	\section*{Data availability statement}
	\noindent The research data associated with this work are included in the article itself. The source code used for the implementation of the proposed algorithm is publicly available on GitHub \url{https://github.com/zv1991/abstract_timoshenko_semidiscrete_scheme} and in an open-access Zenodo repository \cite{RogVashCode2026}.

	% ------------------------------------------------------------
	% Float Barriers (placeins)
	% ------------------------------------------------------------
	% \FloatBarrier  % Insert a manual float barrier (use only if \usepackage{placeins} is loaded WITHOUT the [section] option)
	
	% ------------------------------------------------------------
	% Bibliography Environment
	% ------------------------------------------------------------
	\phantomsection                % Ensure correct hyperlink anchor for references (used with hyperref)
	\begingroup                    % Localize relaxed settings for bibliography only
	\sloppy                        % Loosen justification to prevent overfull lines in references
	\bibliography{bibsource}       % Load references from "bibsource.bib"
	\endgroup                      % Restore normal settings after bibliography
	
	% ------------------------------------------------------------
	% Restore Default Justification
	% ------------------------------------------------------------
	\fussy                         % Re-enable strict justification for the rest of the document
	
	% % ------------------------------------------------------------
	% % Create a Colored and Framed Box (with Light Background)
	% % ------------------------------------------------------------
	% \begin{tcolorbox}[
		%     colback=gray!5,              % Very light gray background (5% gray)
		%     colframe=black!75!gray,      % Frame color: 75% black + 25% gray
		%     boxrule=0.3mm,               % Border thickness
		%     arc=3mm,                     % Rounded corners
		%     width=\textwidth,            % Full text width
		%     enlarge left by=0mm,         % No enlargement on the left
		%     enlarge right by=0mm         % No enlargement on the right
		%     ]
		%     % ------------------------------------------------------------
		%     % Box Content: Justified Citation Text
		%     % ------------------------------------------------------------
		%     \begin{justify}
			%         \textbf{How to cite this article:} \\[0.5em]
			%         Rogava, J., Vashakidze, Z.:
			%         On convergence of a three-layer semi-discrete scheme for the non-linear dynamic string equation of Kirchhoff-type with time-dependent coefficients.
			%         \textit{Z. Angew. Math. Mech.} e202300608 (2024). \href{https://doi.org/10.1002/zamm.202300608}{https://doi.org/10.1002/zamm.202300608}
			%     \end{justify}
		% \end{tcolorbox}
	
\end{document}